\documentclass[11pt,letter paper, oneside]{amsart}
                      
\usepackage[T1]{fontenc}
\usepackage{lmodern}
\usepackage{graphicx}
\usepackage{amsmath, amsthm, amssymb, amsthm, xfrac, mathtools}
\usepackage{stmaryrd}
\usepackage{tikz}
\usepackage{tikz-cd}
\usepackage[pdftex,lmargin=1in,rmargin=1in,tmargin=1in,bmargin=1in]{geometry}
\theoremstyle{definition}
\newtheorem{definition}{Definition}[section]

\newtheorem{example}[definition]{Example}
\newtheorem{construction}[definition]{Construction}
\theoremstyle{theorem}
\newtheorem{theorem}[definition]{Theorem}
\newtheorem{corollary}[definition]{Corollary}
\newtheorem{proposition}[definition]{Proposition}

\newtheorem{lemma}[definition]{Lemma}
\theoremstyle{remark}
\newtheorem{remark}[definition]{Remark}

\newcommand{\Real}{\mathbb{R}}
\newcommand{\Number}{\mathbb{N}}
\newcommand{\Complex}{\mathbb{C}}
\newcommand{\Rational}{\mathbb{Q}}
\newcommand{\Integer}{\mathbb{Z}}

\usepackage[style=backend=biber,style=alphabetic,doi=false,isbn=false,url=false,dateabbrev=true]{biblatex}

\usepackage{capt-of}

\usepackage{amsmath}
\usepackage{amssymb}

\ifcsname endofdump\endcsname\endofdump\fi

\AtEveryBibitem{\clearfield{version}}
\renewbibmacro{in:}{}
\DeclareFieldFormat*{title}{#1}
\usepackage{xpatch}
\xpatchbibmacro{journal+issuetitle}{\usebibmacro{issue+date}}{, \usebibmacro{date}}{}{}

\DefineBibliographyStrings{english}{pages={, pages },page={p\adddot}}
\author{Yu Huang}
\date{\today}
\title{Pseudo-Anosov flow and dynamics on guts}
\usepackage[bookmarks=true, bookmarksopen=true, colorlinks=true, linkcolor=blue]{hyperref}
\hypersetup{
 pdfauthor={Yu Huang},
 pdftitle={},
 pdfkeywords={},
 pdfsubject={},
 pdfcreator={},
 pdflang={English}}
\begin{document}

\address{Beijing International Center for Mathematical Research, Peking University\\
				Beijing 100871, China P.R.}
\email{hytopaz@stu.pku.edu.cn}
\theoremstyle{theorem}
\newtheorem*{thm:flowinguts}{Theorem~\ref{thm:flowinguts}}
\newtheorem*{thm:annulusisotopy}{Theorem~\ref{thm:annulusisotopy}}
\newtheorem*{thm:encodeclosedorbit}{Theorem~\ref{thm:encodeclosedorbit}}
\newtheorem*{thm:facetinvariance}{Theorem~\ref{thm:facetinvariance}}
\newtheorem*{thm:openfaceinvariance}{Theorem~\ref{thm:openfaceinvariance}}
\newtheorem*{cor:homologicallytrivialorbits}{Corollary~\ref{cor:homologicallytrivialorbits}}
\newtheorem*{cor:distinguished}{Corollary~\ref{cor:distinguished}}
\newtheorem*{thm:growthrate}{Theorem~\ref{thm:growthrate}}
\newtheorem*{thm:entropy}{Theorem~\ref{thm:entropy}}

\newtheorem*{prop:shadow3}{Theorem~\ref{prop:shadow3}}

\newtheorem*{thm:example}{Theorem~\ref{cor:example}}
\newtheorem*{thm:secondmain}{Theorem~\ref{thm:secondmain}}
\newtheorem*{thm:hyperbolicity}{Theorem~\ref{thm:hyperbolicity}}
\begin{abstract}
We relate the dynamics on a closed 3-manifold to the topological invariant, homology guts, proposed by Agol and Zhang \autocite{AgolZhang2022_GutsInSuturedDecompositionsAndTheThurstonNorm_ARXv1}.
Given a closed manifold \(M\) that admits a pseudo-Anosov flow \(\varphi\) without perfect fits, we construct a canonical semiflow on the homology guts associated to homology classes carried by \(\varphi\).
We show that, up to orbit equivalent outside half annuli, the semiflow on the guts is invariant on each Thurston open cone.
As an application, the fundamental group and the sutured structure of homology guts encode closed orbits of \(\varphi\) with particular type.
Besides, for a positive class in \(H^{1}(M)\), the canonical semiflow on the guts has a well-defined growth rate.
\end{abstract}
\maketitle
\tableofcontents
\section{Introduction}
\label{sec:org2e6590d}
\subsection{Overview}
\label{sec:org404197b}
It has long been a central question to understand the interplay between topology and dynamics on 3-manifolds.
Recently, there have been many works on this topic via combinatorial objects and circle actions.
One promising direction is to arrange topological objects in a suitable position relative to the flow.
For instance, a modified JSJ decomposition for Anosov flows is known to hold: up to some technical details, any Anosov flow on a toroidal manifold can be canonically decomposed into ``building blocks'', by cutting along quasi-transverse tori (see \autocite[Section 2.2]{BarthelmeFenley2017_CountingPeriodicOrbitsOfAnosovFlowsInFreeHomotopyClasses_PUB}).
Following this line, this paper relates the dynamics on a 3-manifold to a topological invariant, the homology guts, proposed by Agol and Zhang \autocite{AgolZhang2022_GutsInSuturedDecompositionsAndTheThurstonNorm_ARXv1}.

To this end, we construct a canonical semiflow on the homology guts associated to a class carried by a pseudo-Anosov flow without perfect fits (Theorem \ref{thm:flowinguts}).
On the one hand, the original pseudo-Anosov flow is obtained by gluing back the semiflows on the guts and some product flows blocks (Lemma \ref{lem:windowproduct}).
On the other hand, the canonical semiflow offers a new way to bridge dynamics and topology, and suggests that the guts encodes dynamical information in a natural way.
For example, the fundamental group of homology guts encodes closed orbits with certain homology classes: the homology class of a closed orbit is in a closed face \(F\) of the Fried cone if and only if the closed orbit can be isotoped into the guts \(\mathrm{Guts}(F^{\vee})\) associated to the dual open Thurston cone \(F^{\vee}\subset H^{1}(M)\).

We will prove that the semiflow is independent of all choices made in its construction, up to orbit equivalence outside half annuli.
Moreover, like the homology guts, the semiflow is invariant on each Thurston open cone.
Similar to the pseudo-Anosov flow, for a positive class in \(H^{1}(M)\), the growth rate of the canonical semiflow is shown to exist.

One important ingredient of the construction is the Annulus Isotopy Theorem (Theorem \ref{thm:annulusisotopy}).
While the Transverse Surface Theorem \autocite[Theorem A]{LandryMinskyTaylor2025_TransverseSurfacesAndPseudoAnosovFlows_PUB} characterizes when a taut surface \(S\) can be ``transverse'' to the flow, the Annulus Isotopy Theorem shows how to put a product annulus in \(M-S\) in an aligned position with the flow.
It provides the second step of a flow-version of the sutured decomposition sequence.
We conjecture that similar results hold for any decomposition surface with suitable conditions.
See Section \ref{sec:discussion} for more discussion.

Before stating our results we introduce some broad-strokes definitions.
Let \(\varphi\) be a pseudo-Anosov flow without perfect fits on a closed 3-manifold \(M\).
We say \(\varphi\) \emph{carries} a class \(z\in H_{2}(M)\) if the negative Euler class of \(\varphi\) matches the Thurston norm at \(z\).
All classes carried by \(\varphi\) form a (possibly empty) closed cone \(\mathcal{C}_{2}(\varphi)\) in \(H_{2}(M)\) which is spanned by a closed face of the Thurston unit ball.
In \autocite{Mosher1992_DynamicalSystemsAndTheHomologyNormOfA3ManifoldIi_PUB}, Mosher shows that \(\varphi\) \emph{represents} \(\mathcal{C}_{2}(\varphi)\), in the sense that the carried cone \(\mathcal{C}_{2}(\varphi)\) of \(\varphi\) is dual to its Fried cone \(\mathcal{C}_{\mathrm{Fr}}(\varphi)\), spanned by homology classes of closed \(\varphi\)-orbits.

Our main theorem is:
\begin{thm:flowinguts}
Let \(M\) be a closed 3-manifold and \(\varphi\) be a pseudo-Anosov flow without perfect fits on \(M\).
Denote by \(\mathcal{C}_{2}(\varphi)\) the carried cone of \(\varphi\) in \(H_{2}(M)\).
For any primitive element \(z\) in \(\mathcal{C}_{2}(\varphi)\), the homology guts \(\mathrm{Guts}(M,z)\) admits a tight semiflow induced by \(\varphi\).
\end{thm:flowinguts}
In the sense of Fenley-Mosher \autocite{FenleyMosher2001_QuasigeodesicFlowsInHyperbolic3Manifolds_PUB}, semiflows simply generalize flows from closed manifolds to compact manifolds, which are allowed to stop at the boundary (see also Definition \ref{def:partialflow}).
Since the guts is a sutured submanifold, we require the semiflow on it to be ``compatible'' with the sutured structure (see Definition \ref{def:tight}).
\subsection{Construction}
\label{sec:org15b126b}
Given a second homology class \(z\), the construction of the homology guts \(\mathrm{Guts}(M,z)\) consists of two main steps.
Firstly, one removes a maximal collection of embedded, disjoint, taut surfaces with class \(z\).
Next, one cuts along a maximal collection of product annuli and disks.
After discarding all product sutured manifold components called \emph{windows}, one obtains the \emph{homology guts} \(\mathrm{Guts}(z)\).
Surprisingly, Agol and Zhang \autocite{AgolZhang2022_GutsInSuturedDecompositionsAndTheThurstonNorm_ARXv1} show that \(\mathrm{Guts}(z)\) is an topology invariant of \(M\), only depending on \(z\) up to equivalence of guts.
Furthermore, the homology guts is invariant on each Thurston open cone and the guts of different homology classes are related by sutured decompositions \autocite[Theorem 1.2]{AgolZhang2022_GutsInSuturedDecompositionsAndTheThurstonNorm_ARXv1}.
The homology guts is closely related to many other invariants, such as volume, sutured Floer homology and \(L^{2}\)-torsion.

The construction in Theorem \ref{thm:flowinguts} relies on two theorems: the Transverse Surface Theorem \autocite[Theorem A]{LandryMinskyTaylor2025_TransverseSurfacesAndPseudoAnosovFlows_PUB} for the first step, and the Annulus Isotopy Theorem stated below for the second.
We summarize the complete procedure in Construction \ref{cst:flow}.
\begin{thm:annulusisotopy}[Annulus Isotopy Theorem]
Assume that \(M\) is atoroidal and admits an almost pseudo-Anosov flow \(\varphi\) without perfect fits.
Let \(S\) be an embedded, compact, oriented, taut surface in \(M\) which is positively transverse to \(\varphi\) minimally.
Suppose that \(A\) is a homotopically non-trivial, oriented annulus with two boundary components lying on opposite sides of \(S\).
In other words, \(A\) is a product annulus in the sutured manifold \((M-S,S_{+},S_{-},\emptyset)\)..

Then, \(A\) is ambiently isotopic in \(M-S\) to an annulus \(A_{\varphi}\), which is either product \(\varphi\)-saturated or transverse to a further blowup of \(\varphi\).
We call that \(A_{\varphi}\) is \emph{in aligned position} with \(\varphi\).
\end{thm:annulusisotopy}
Theorem \ref{thm:annulusisotopy} is of independent interest since it is the second step in a flow-version of sutured decomposition sequence.
One may ask if a similar result holds for any decomposition surface with suitable conditions.
See Section \ref{sec:discussion} for more discussions.
We remark that in \autocite{HuangTaylor2026_PseudoAnosovFlowsHyperbolicGeometryAndTheCurveGraph_ARXv1}, Huang and Taylor proved a slightly weaker version of the Annulus Isotopy Theorem independently (see Remark \ref{rem:HTwork}).
\subsection{Invariance}
\label{sec:org6fac001}
Interestingly, the semiflow constructed above is a canonical one, independent of choices in the construction, up to orbit equivalence outside half annuli (Definition \ref{def:halfannulus}).
\begin{thm:facetinvariance}
Let \(M\) be an oriented connected closed 3-manifold and let \(\varphi\) be a pseudo-Anosov flow without perfect fits in \(M\). Suppose that \(z\in \mathcal{C}_{2}(\varphi)\) is primitive.

Then, up to orbit equivalence outside half annuli, the semiflow on \(\mathrm{Guts}(M,z)\) is independent of choices in Construction \ref{cst:flow}.
Precisely speaking, it doesn't depend on choices of the facet \(F(z)\), the product annuli collection \(\mathcal{A}\), the isotopy choices of \(F_{\varphi}(z)\) and \(\mathcal{A}_{\varphi}\), and the blowup ways of \(\varphi^{\sharp}\) and \(\varphi^{\sharp\sharp}\).
\end{thm:facetinvariance}
Here we eliminate considerable flexibility of blowup constructions by removing all half annuli.

Moreover, it is invariant on each Thurston open cone, similar to the homology guts.
\begin{thm:openfaceinvariance}
Assume that \(M\) be an oriented connected closed 3-manifold which admits a pseudo-Anosov flow \(\varphi\) without perfect fits.
Suppose that \(y,z\in \mathcal{C}_{2}(\varphi)\) are primitive and in the same open face.
Then, the guts \(\mathrm{Guts}_{\varphi}(M,y)\) is orbit equivalent to \(\mathrm{Guts}_{\varphi}(M,z)\) outside half annuli.
\end{thm:openfaceinvariance}
\subsection{Application}
\label{sec:org01d2ba6}
The canonical semiflow gives a new bridge between dynamics and topology, and suggests that the guts encodes dynamical information naturally.
As a direct application, the fundamental group of homology guts encodes closed orbits with certain homology class.
\begin{thm:encodeclosedorbit}
Let \(F\) be a closed face of \(\mathcal{C}_{\mathrm{Fr}}(\varphi)\) and let \(F^{\vee}\) be its dual open face in \(\mathcal{C}_{2}(\varphi)\).
Assume that \(\gamma\) is a closed orbit of \(\varphi\).
Then \([\gamma]\in F\) if and only if \(\gamma\) can be isotopic into \(\mathrm{Guts}(F^{\vee})\).
Equivalently, \([\gamma]\in F\) if and only if \(\pi_{1}(\gamma)\) and \(\pi_{1}(\mathrm{Guts}(F^{\vee}))\), as conjugate classes, have nonempty intersection.
\end{thm:encodeclosedorbit}
In particular, the collection of fundamental group of homology guts satisfies cluster property, namely, it's a finite collection of infinite-index quasi-convex subgroups of \(\pi_{1}(M)\) which encodes closed orbits on closed Thurston face. See Definition \ref{def:cluster} for details.
This property plays a central role in many works of Yi Liu \autocite{Liu2020_VirtualHomologicalSpectralRadiiForAutomorphismsOfSurfaces_PUB,Liu2022_FiniteVolumeHyperbolic3ManifoldsAreAlmostDeterminedByTheirFiniteQuotientGroups_PUB,Liu2025_TheEulerClassOneConjectureForFillableContactStructures_ARXv2}.
In \autocite{Liu2020_VirtualHomologicalSpectralRadiiForAutomorphismsOfSurfaces_PUB}, Liu gave a construction using Markov partition in the suspension flow case.
Liu's construction may depend on the choice of Markov partition and is merely a collection of subgroups.
Theorem \ref{thm:encodeclosedorbit} shows that homology guts gives a canonical submanifold realization of cluster property since the guts is a topological invariant of \(M\).
We note that Theorem \ref{thm:encodeclosedorbit} is stronger than the cluster property.

As an interesting corollary, if \(\varphi\) represents a top-dimensinal face, then all homologically trivial closed orbits are encoded by a homology guts.
\begin{cor:homologicallytrivialorbits}
Suppose that \(\varphi\) represents a top-dimensional closed cone \(\mathcal{C}_{2}(\varphi)\).
Then a closed orbit of \(\gamma\) is homologically trivial if and only if \(\gamma\) can be isotoped into \(\mathrm{Guts}(T)\) where \(T\) is the top open face of \(\mathcal{C}_{2}(\varphi)\).
\end{cor:homologicallytrivialorbits}
Beyond the fundamental group, the sutured structure of the guts also encodes some typical closed orbits. 
We say that a closed orbit \(\gamma\) is \(z\)-\emph{distinguished} in the direction of a prong \(P\) if there is a surface with \([S]=z\) so that \(P\cap S\) contains a closed curve component.
\begin{cor:distinguished}
Let \(z\in \mathcal{C}_{2}(\varphi)\) and let \(D_{n}(z)\) be the set of regular closed orbits which is \(z\)-distinguished in at least \(n\) directions. Then
\begin{enumerate}
\item each orbit in \(D_{1}(z)\) has a power which is isotoped into \(R_{\pm}\)-part of \(\mathrm{Guts}(z)\),
\item \(\#D_{2}(z)\leq\) the number of sutured annuli in \(\mathrm{Guts}(z)\),
\item \(\#D_{4}(z)\leq\) the number of 4-STs in \(\mathrm{Guts}(z)\).
\end{enumerate}
\end{cor:distinguished}
\subsection{Growth rate}
\label{sec:orgcc7abd2}
We show that for a positive class in \(H^{1}(M)\), the canonical semiflow has a well-defined growth rate and entropy.
For any open face \(F\) of \(\mathcal{C}_{2}(\varphi)\), denote by \(\varphi_{F}\) the canonical semiflow on \(\mathrm{Guts}(M,F)\) constructed in Theorem \ref{thm:flowinguts}.
We define the exponential growth rates as
\[ \mathrm{gr}_{\varphi_{F}}(\eta)= \lim_{L\to \infty} \#\{\gamma\in \mathcal{O}(\varphi_{F}): \eta(\gamma)\leq L\}^{\frac{1}{L}}, \]
where \(\mathcal{O}(\varphi_{F})\) the set of closed orbits of \(\varphi_{F}\) and \(\eta\in H^{1}(M)\).
\begin{thm:growthrate}[Growth rate]
For any positive class \(\eta\in H^{1}(M)\), the growth rate \(\mathrm{gr}_{\varphi_{F}}(\eta)\) exists.
The growth rate \(\mathrm{gr}_{\varphi_{F}}(\eta)\) is strictly greater than 1 for every positive \(\eta\in H^{1}(M)\) if and only if \(\varphi_{F}\) contains infinitely many primitive closed orbits.
\end{thm:growthrate}

As the logarithm of the growth rate, the entropy plays a more central role in many fields.
We summarize the basic properties of the entropy function.
Let \(\mathcal{C}^{+}(F)\) be the cone in \(H^{1}(M)\) which consists of positive classes.
The associated \emph{entropy function} is
\[ \begin{aligned}
\mathrm{ent}_{\varphi_{F}}: \mathcal{C}^{+}(F) &\to [0,+\infty)\\
\eta&\mapsto \log( \mathrm{gr}_{\varphi_{F}}(\eta) ).
\end{aligned} \]
Combining Theorem \ref{thm:growthrate} and \autocite[Theorem 9.1]{LandryMinskyTaylor2022_FlowsGrowthRatesAndTheVeeringPolynomial_PUB}, we have
\begin{thm:entropy}
The entropy function \(\mathrm{ent}_{\varphi_{F}}\) is continuous, convex, and homogeneous of degree -1.
\end{thm:entropy}
\subsection{Methods}
\label{sec:orgd006c21}
We explain the key ideas toward proofs of the main theorems.
To obtain the semiflow, we are required to put the embedded taut surfaces, and product annuli and disks in a good position with the flow \(\varphi\).
The surface case is solved by Landry, Minsky and Taylor \autocite{LandryMinskyTaylor2025_TransverseSurfacesAndPseudoAnosovFlows_PUB}: any taut surface \(S\) with \([S]\in \mathcal{C}_{2}(\varphi)\) is isotopic to be almost transverse to the flow.
The product annulus case is solved in the Annulus Isotopy Theorem (Theorem \ref{thm:annulusisotopy}), and it suffices to cut all product annuli due to Remark \ref{rem:productdisk}.

The Annulus Isotopy Theorem is proved using the orbit space and shadows, both of which are developed by Fenley in a series of works.
The orbit space \(P_{\varphi}\) is obtained by quotienting orbits in the universal cover \(\widetilde{M}\), and is homeomorphic to a plane.
A lift of an embedded transverse surface projects to an open region whose boundary consists of slice leaves.
Suppose that a product annulus \(A\) connects \(S_{1}\) and \(S_{2}\).
Then after fixing lifts, shadows of \(\widetilde{S}_{1}\) and \(\widetilde{S}_{2}\) will contains a \(g\)-invariant line respectively, where \(g\in \pi_{1}(M)\) represents the fundamental group of \(A\).
By Lemma \ref{lem:invariantline}, it suffice to find a \(g\)-invariant line in the intersection of their shadows
But it doesn't hold always.
Such invariant line exists when \(g\)-action is hyperbolic or fixes a regular point (Lemma \ref{lem:hyperbolic_invariantline} and Lemma \ref{lem:regular-invariantline}).
Otherwise, if the fixed point is singular, sometimes we are required to blow up \(\varphi\) further such that \(A\) can be isotopic to be a transverse annulus (Lemma \ref{lem:singular-invariantline}).

To prove two invariance theorems (Theorem \ref{thm:facetinvariance} and Theorem \ref{thm:openfaceinvariance}), we devide the guts into \emph{singular} and \emph{regular parts} (Section \ref{sec:singularguts}).
Roughly speaking, the singular part consists of the components near a singular orbit; these are either 4-STs or non-longitudinal 2-STs.
The invariance is established separately, using different arguments.
One reason we split the proof into two parts is that \emph{bad regions} appear only in singular guts (see Remark \ref{rem:badregions}).
\subsection{Organization}
\label{sec:org896c16b}
In Section \ref{sec:preliminary}, we review pseudo-Anosov flows and sutured manifolds, and introduce the semiflow.
In Section \ref{sec:orbitspace}, we review the orbit space and shadows.
In Section \ref{sec:annulusisotopy}, we prove the Annulus Isotopy Theorem (Theorem \ref{thm:annulusisotopy}), and in Section \ref{sec:construction} we construct the semiflow on the guts (Theorem \ref{thm:flowinguts}).
Section \ref{sec:invariancefacet} and Section \ref{sec:invarianceopenface} are devoted to the proofs of two invariance theorems (Theorem \ref{thm:facetinvariance} and Theorem \ref{thm:openfaceinvariance}).
We present applications in Section \ref{sec:application} and conclude with further discussions and questions in Section \ref{sec:discussion}.
\subsection*{Acknowledgement}
\label{sec:org90c59ca}
The author thanks his advisor Yi Liu for helpful comments, and thanks Yaoping Xie and Jianru Duan for valuable conversations.
The author also thanks Samuel Taylor and Chicheuk Tsang for their helpful comments on an earlier draft.
\section{Preliminary}
\label{sec:preliminary}
In this paper, we sometimes omit the \(\Integer\) coefficient in homology and cohomology groups.
We always denote by \(N_{\epsilon}(X)\) a small enough, open, (tubular) neighborhood of \(X\) (if \(X\) is a submanifold).
If \(S\) is a properly embedded surface in a 3-manifold \(M\),  we rename \(M\setminus N_{\epsilon}(S)\) by \(M-S\).
\subsection{Pseudo-Anosov flow}
\label{sec:pre:pAflow}
We refer the reader to \autocite[Section 4]{FenleyMosher2001_QuasigeodesicFlowsInHyperbolic3Manifolds_PUB} and \autocite[Section 5]{AgolTsang2024_DynamicsOfVeeringTriangulations_PUB} for a thorough discussion of the definition of pseudo-Anosov flow on a closed 3-manifold.
There are two versions, smooth one and topological one. Up to orbit equivalence, two versions are equivalent for transitive flows \autocite{Shannon2020_DehnSurgeriesAndSmoothStructuresOn3DimensionalTransitiveAnosovFlows_PUB,AgolTsang2024_DynamicsOfVeeringTriangulations_PUB}.
We will typically work with transitive flow so we won't distinguish them.
In fact, in our main theorem, \(M\) is assumed to be atoroidal, so the pseudo-Anosov flow on it must be transitive due to \autocite[Proposition 2.7]{Mosher1992_DynamicalSystemsAndTheHomologyNormOfA3ManifoldI_PUB}.

We summarize the definition as follows. Let \(\varphi\) be a (smooth) pseudo-Anosov flow.
\begin{enumerate}
\item There is a finite collection of closed orbits \(\{\gamma_{i}\}_{i\in I}\), called \emph{singular orbits} of \(\varphi\) and a Riemannian metric \(g\) on \(M\) removing \(\cup_{i\in I}\gamma_{i}\), such that \(\varphi\) is smooth away from the singular orbits.
\item Outside these singular orbits \(\{\gamma_{i}\}_{i\in I}\), there is a splitting of the tangent bundle into three \(\varphi\)-invariant line bundles,
\[ TM=E^{s}\oplus E^{u} \oplus T\varphi, \]
such that
\[ |d\varphi_{t}(v)|<C\lambda^{-t}|v|,\quad \forall v\in E^{s},t>0, \]
and
\[ |d\varphi_{t}(v)|<C\lambda^{t}|v|, \quad \forall v\in E^{u},t<0, \]
hold for some uniform parameters \(C,\lambda>1\).
\item Each singular orbit is a \emph{pseudo-hyperbolic orbit}: \(\gamma_{i}\) has a neighborhood \(N_{i}\) and a map \(f_{i}\) sending \(N_{i}\) to a \emph{standard pseudo-hyperbolic local model} \(N_{n_{i},k_{i},\lambda}\) for some \(n_{i}\geq 3\), such that,
\begin{enumerate}
\item \(f_{i}\) is bi-Lipschitz on \(N_{i}\) and smooth away from those singular orbits \(\gamma_{i}\).
\item \(f_{i}\) preserves orbits and maps \(E^{s}\) and \(E^{u}\) to line bundles tangent to local stable foliation \(\Lambda^{s}\) and local unstable foliation \(\Lambda^{u}\) respectively.
\end{enumerate}
In this case, \(\gamma_{i}\) is called \(n_{i}\)-\emph{pronged}.
\end{enumerate}
\subsection{Dynamic blowup}
\label{sec:pre:blowup}
We refer the reader to \autocite[Section 3]{LandryMinskyTaylor2025_TransverseSurfacesAndPseudoAnosovFlows_PUB} for a detailed discussion of the dynamic blowup and almost pseudo-Anosov flows.
We summarize the procedure.

To blowup a singular orbit \(\gamma\) locally, we choose a smooth transverse disk \(D\) to \(\gamma\), and translate the flow near \(\gamma\) to a return map \(f\) on \(D\) (which may not defined on the whole \(D\)).
\(f\) has only one fixed point \(O\). \(O\) has an even number of \emph{prongs} embedded in \(D\).
For some \(k\), \(f^{k}\)-action fixes every prong and is a contraction or an expansion.
We denote by \(T\) the union of \(O\) and its prongs, which is a star shape in \(D\).
On each component of \(D\setminus T\), the \(f^{k}\)-action is topologically equivalent to standard Anosov map.

A return map \(f^{\sharp}\) is a \emph{blowup} of \(f\) if:

There is a surjective map \(\pi:D\to D\), which semiconjugates \(f^{\sharp}\) to \(f\).
The preimage of \(T\) is a properly embedded tree \(T^{\sharp}\), and the preimage \(O^{\sharp}\) of \(O\) is a subtree of \(T^{\sharp}\).
The periodic points of \(f^{\sharp}\) are the vertices of \(T^{\sharp}\), and a power of \(f^{\sharp}\) acts by translation on edges of \(O^{\sharp}\).
On each open quadrants of \(D\setminus T^{\sharp}\), \(\pi\) conjugates \(f^{\sharp}\)-action to \(f\)-action.
On each edge of \(O^{\sharp}\), the direction of \(f^{\sharp}\)-action is compactible with the \(f^{\sharp}\)-action on near open quadrants.
See the right of Figure \ref{fig:blowup}.

Blowing up \(\gamma\) locally is to replace the (semi)flow in a neighborhood of \(\gamma\) by the suspension flow of \(f^{\sharp}\).
Edges of \(O^{\sharp}\) suspend to \(\varphi^{\sharp}\)-invariant annuli, called \emph{blown-up annuli}.
On each blown annulus, flowlines are asymptotic to one boundary component in backward time and to another boundary component in forward time.
The suspension of \(O^{\sharp}\) is a union of blown annuli, called the \emph{blown-up complex} associated to \(\gamma\).

Globally, we say a flow \(\varphi^{\sharp}\) on \(M\) is a \emph{(dynamic) blowup} of a pseudo-Anosov flow \(\varphi\), if there exists a blowdown map \(\pi:\varphi^{\sharp}\to \varphi\), which satisfies:
\begin{enumerate}
\item \(\pi\) is a semiconjugacy,
\item there is a finite collection \(\Pi\) of singular orbits of \(\varphi\), such that, \(\pi\) is a homeomorphism and a flow conjugacy from \(\pi^{-1}(M\setminus \Pi)\) to \(M\setminus \Pi\),
\item Each singular orbit \(\gamma\) in \(\Pi\) has a neighborhood \(N_{\gamma}\) such that \(\varphi^{\sharp}\) restricting on \(\pi^{-1}(N_{\gamma})\) is a local blowup of \(\varphi\) restricting on \(N_{\gamma}\).
\end{enumerate}
The boundary-fixing homeomorphic type of the blown-up complex associated to \(\varphi^{\sharp}\) is called the \emph{combinatorial type} of \(\varphi^{\sharp}\).

One can blow up a closed orbit to a hollow torus, namely, the Fried blowup. 
Locally, it blow up the transverse disk \(D\) to an annulus, and blow up the unique fixed point to a circle.
See the middle of Figure \ref{fig:blowup} and \autocite[Section 3]{LandryMinskyTaylor2025_TransverseSurfacesAndPseudoAnosovFlows_PUB} for detailed construction 
Giving a closed orbit \(\gamma\), the Fried blowup \(\varphi^{\circ}_{\gamma}\) at \(\gamma\) is not unique up to conjugacy (see \autocite[Chapter 1, Section 7.4]{Shannon2020_DehnSurgeriesAndSmoothStructuresOn3DimensionalTransitiveAnosovFlows_PUB}).
We note that the flow on a blown-up annulus is unique up to orbit equivalence, and so is the flow on \(\partial \varphi_{\gamma}^{\circ}\).

\begin{figure}[htbp]
\centering
\includegraphics[scale=0.6]{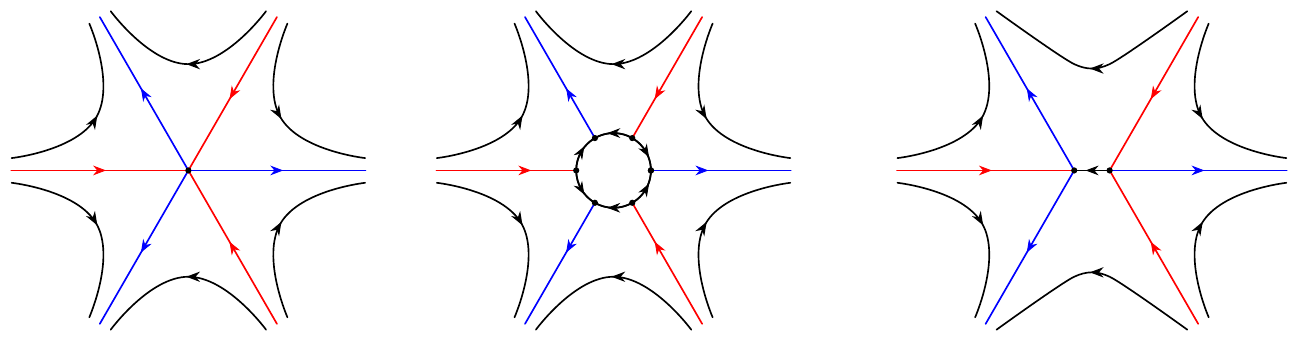}
\caption{\label{fig:blowup}From left to right: the local picture of a singular orbit in a pseudo-Anosov flow, the Fried blowup result, a dynamic blowup result.}
\end{figure}
\subsection{Thurston norm}
\label{sec:org1412951}
This section reviews the Thurston norm and taut surfaces.
\begin{definition}
Let \(S\) be a compact connected orientable surface. We define \(\chi_{-}(S)=\max \{0,-\chi(S)\}\).
If \(S\) is disconnected, we define
\[ \chi_{-}(S)=\sum_{S_{i}\subset S \text{ a component}} \chi_{-}(S_{i}). \]
\end{definition}
Based on this, Thurston defines a pseudo-norm on \(H_{2}(M,\partial M;\Real)\) and \(H_{2}(M;\Real)\) in \autocite{Thurston1986_NormForTheHomologyOf3Manifolds_PUB}.
\begin{definition}[Thurston norm]
Let \(M\) be a compact oriented 3-manifold and \(Y\) be a subsurface of \(\partial M\).
For any \(z\in H_{2}(M,Y)\), we define the \emph{Thurston norm} of \(z\) to be
\[ x(z)=\min \{\chi_{-}(S): (S,\partial S)\subset (M,Y), [S]=z\in H_{2}(M,K)\}. \]
\end{definition}
Thurston proved that \(x(nz)=nx(z)\), so \(x\) extends to a pseudo-norm on \(H_{2}(M,K;\Real)\) by linearity and continuity.
When \(M\) is atoroidal, the Thurston norm is non-degenerate.

We denote by \(\mathcal{B}_{\mathrm{Th}}(M)\) the unit ball of the Thurston norm on \(H_{2}(M,\partial M)\). Thurston \autocite[Theorem 2]{Thurston1986_NormForTheHomologyOf3Manifolds_PUB} proves that \(\mathcal{B}_{\mathrm{Th}}(M)\)  is a rational polyhedron.
To simplify the notation, we call \(\mathcal{B}_{\mathrm{Th}}(M)\) the \emph{Thurston ball} and call \(\partial \mathcal{B}_{\mathrm{Th}}(M)\) the \emph{Thurston sphere}.
A \emph{Thurston cone} is refered to the cone formed by the origin and a (open or closed) face of the Thurston ball.

\begin{definition}[taut surface]
Let \(M\) be a compact oriented 3-manifold, and let \(S\) be an oriented surface representing nontrivial \(\alpha\in H_{2}(M,\partial M)\).
We say that \(S\) is \emph{taut} (\emph{norm-minimizing}) if
\begin{enumerate}
\item \(S\) contains no null-homologous component in \(H_{2}(M,\partial M)\),
\item \(\chi_{-}(S)=x_{M}(\alpha)\).
\end{enumerate}
\end{definition}

There is a standard surgery to eliminate the interesection between two surfaces, called the \emph{double curve sum}.
Precisely speaking, let \(S\) and \(T\) be two properly embedded surface in general position. One applies cut-and-paste surgery to the intersection curves of \(S\) and \(T\), and turn the immersed surface \(S\cup T\) to a properly embedded oriented surface \(S\oplus T\).
\begin{lemma}
If every component of \(S\cap T\) is essential in \(S\) and \(T\), then \(\chi_{-}(S)+\chi_{-}(T)=\chi_{-}(S\oplus T)\).
\end{lemma}
\subsection{Sutured manifold}
\label{sec:orgf99a34e}
To study taut surfaces and taut foliations, Gabai \autocite{Gabai1983_FoliationsAndTheTopologyOf3Manifolds_PUB} introduces sutured manifolds and sutured decomposition.
\begin{definition}
A \emph{sutured manifold} \((N,R_{+},R_{-},\gamma)\) is a compact oriented 3-manifold \(N\) with boundary together with a set \(\gamma\subset \partial N\) satisfying the following conditions:
\begin{enumerate}
\item \(\gamma\) consists of pairwise disjoint annuli \(A(\gamma)\) and tori \(T(\gamma)\). A component of \(A(\gamma)\) is called a \emph{sutured annulus} and a component of \(T(\gamma)\) is called a \emph{sutured torus}.
\item Each sutured annulus contains a homologically nontrivial oriented simple closed curve, called a \emph{suture}.
\item \(R_{+}\cup R_{-}= \partial N\setminus \operatorname{int}(\gamma)\) is oriented such that the orientation of \(R_{\pm}\) is coherent with the orientation of sutures in \(\partial R_{\pm}\) by boundary orientation principle.
Denote by \(R_{+}(\gamma)\) (resp. \(R_{-}(\gamma)\)) the components of \(\partial N\setminus \operatorname{int}(\gamma)\) whose normal vectors point out of (resp. into) \(N\).
\end{enumerate}
In practice, we usually abbreviate the notation to \((N,\gamma)\), or more simply to \(N\), when there is no ambiguity.
\end{definition}

\begin{definition}
Let \((N,R_{+},R_{-},\gamma)\) be a sutured manifold and \(S\) be a properly embedded surface in \(N\).
We call \(S\) a \emph{decomposition surface} for \((N,\gamma)\) if the following conditions hold:
\begin{enumerate}
\item For any sutured torus \(T\), the intersection \(S\cap T\) is either empty or a collection of essential simple closed curves, all of which represent the same homology class in \(H_{1}(T)\).
\item For any sutured annulus \(A\), \(S\cap A\) is either empty, a simple closed curve with the same homology class as the suture \(A\cap s(\gamma)\), or a properly embedded nonseparating arc.
\item No component of \(\partial S\) bounds a disc in \(R(\gamma)\) and no component of \(S\) is a disc \(D\) with \(\partial D\subset R_{+}\cup R_{-}\).
\end{enumerate}
\end{definition}
Given a decomposition surface \(S\) in a sutured manifold \((N,\gamma)\), we construct a suture structure on \(N'= N- S\) as follows:
\[ \begin{aligned}
\gamma' &= (\gamma\cap N') \cup \overline{N_{\epsilon}(S_{+}'\cap R_{-})} \cup  \overline{N_{\epsilon}(S_{-}'\cap R_{+})}\\
R_{+}' &= ((R_{+}\cap M')\cup S_{+}') \setminus \operatorname{int}(\gamma') \\
R_{-}' &= ((R_{-}\cap M')\cup S_{-}') \setminus \operatorname{int} (\gamma')
\end{aligned} \]
Here, \(S_{+}'\) (resp. \(S_{-}'\)) is the union of the component of \(\partial N_{\epsilon}(S)\) whose normal vector points out of (resp. into) \(N\).

We say that \((N,R_{+},R_{-},\gamma) \stackrel{S}{\rightsquigarrow} (M',R_{+}',R_{-}',\gamma')\) is a \emph{sutured manifold decomposition}.
\begin{definition}
A sutured manifold \((N,\gamma)\) is \emph{taut} if \(N\) is irreducible and \(R_{+}\) and \(R_{-}\) (pushed slightly into \(N\)) are both taut surfaces in \(N\).
\end{definition}

Tautness of  sutured manifold could be pullback by a sutured manifold decomposition.
\begin{lemma}[{\cite[Lemma 0.4]{Gabai1987_FoliationsAndTheTopologyOf3ManifoldsII_PUB}}]
Suppose \((N,\gamma)\) is a sutured manifold and \((N,\gamma)\stackrel{S}{\rightsquigarrow} (N',\gamma')\) is a sutured manifold decomposition. If \((N',\gamma')\) is taut, then so is \((N,\gamma)\).
\end{lemma}

There are two types of simplest decomposition surfaces. For these decomposition surfaces, tautness is preserved under the sutured manifold decomposition.
\begin{definition}
Let \((N,\gamma)\) be a sutured manifold. A properly embedded essential annulus A in \(M\) is called a \emph{product annulus} if it does not cobound a solid cylinder and its two boundary components \(\partial _{+}A\) and \(\partial _{-}A\) of \(A\) lies in \(R_{+}\) and \(R_{-}\), respectively.
A properly embedded disk \(D\) in \(M\) is called a \emph{product disk} if its boundary intersects \(s(\gamma)\) at exactly two points.
We call both product annuli and product disks \emph{product decomposition surfaces}.

In the contrast, we say \(A\) is a \emph{non-product annulus} if it is neither a product annulus nor boundary-parallel to \(R_{+}\cup R_{-}\).
\end{definition}

\begin{lemma}[{\cite[Lemma 3.12]{Gabai1983_FoliationsAndTheTopologyOf3Manifolds_PUB}}]
Let \(S\) be a product annulus or a product disk and \((N,\gamma)\stackrel{S}{\rightsquigarrow} (N',\gamma')\) be a sutured decomposition. Then \((N,\gamma)\) is taut if and only if \((N',\gamma')\) is taut.
\end{lemma}
Based on taut sutured decompositions, Gabai \autocite{Gabai1983_FoliationsAndTheTopologyOf3Manifolds_PUB} constructs a taut foliation on a Haken manifolds. Furthermore, the unwritten works of Gabai and Mosher construct a pseudo-Anosov flow which is almost transverse to that foliation above.
This offers a glimpse into the subtle interplay among flows, sutured manifold, and taut foliation.
\subsection{Semiflows on sutured manifolds}
\label{sec:orgd060fc9}
\begin{definition}[Semiflow]
Let \(N\) be a compact manifold.
A \emph{semiflow} on \(N\) is a continuous map \(\varphi\) mapping a subset \(U\subset N\times \Real\) to \(N\), which satisfies that:
\begin{enumerate}
\item for each \(x\in N\), \(\{t\in \Real:(x,t)\in U\}\) is a closed interval of \(\Real\) and contains \(0\).
\item if \((x,t)\in U\), then \((\varphi(x,t),-t)\in U\),
\item \(\varphi(x,0)=x\) and \(\varphi(\varphi(x,t_{1}),t_{2})=\varphi(x,t_{1}+t_{2})\).
\end{enumerate}
\label{def:partialflow}
\end{definition}
Semiflows generalize flows in closed 3-manifolds to compact ones naturally.
It is introduced by Fenley-Mosher in \autocite{FenleyMosher2001_QuasigeodesicFlowsInHyperbolic3Manifolds_PUB}.
A similar definition is ``pseudo flows'' given by Shannon in \autocite[Section 2.1]{Shannon2025_HyperbolicModelsOfTransitiveTopologicallyAnosovFlowsInDimensionThree_ARXv2}.
By Condition 2 and 3, we know that \(\varphi(x,t)=y\) implies \(\varphi(y,-t)=x\). The (partial) orbit of a point \(x\in M\) is homeomorphic to a closed interval, or a half line, or \(\Real\).
Therefore, a semiflow generates a partition of \(N\) into orbits.
If \(\varphi\) is a (honest) flow in a closed manifold \(M\), and \(N\) is a closed subset of \(M\), then \(N\) has a naturally semiflow by restricting \(\varphi\).

\begin{example}
Let \(N=(S\times I,S\times \{1\},S\times \{0\},\partial S\times I)\) be a product sutured manifold. There is a natural semiflow \(\varphi_{t}\) on \(N\) up to orbit equivalence:
\[ \varphi_{t}((s,u))=(s,u+t)\in S\times I, \text{ where } s\in S,u\in I, t\in [-u,1-u]. \]
We call \(\varphi_t\) the \emph{product flow} on \(N\). See Figure \ref{fig:productflow} for an example.
\end{example}

\begin{figure}[htbp]
\centering
\includegraphics[scale=1.8]{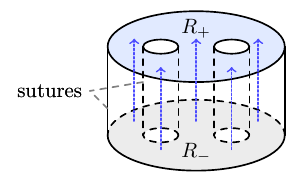}
\caption{\label{fig:productflow}The product flow on \(S\times I\) where \(S\) is a pair of pants.}
\end{figure}

We are interested in the orbit equivalent class of semiflows.
\begin{definition}[Orbit equivalence]
Let \(N_{1}\) and \(N_{2}\) be oriented compact 3-manifolds, which admit semiflows \(\varphi_{t}^{1}\) and \(\varphi_{t}^{2}\), respectively.
The semiflows \(\varphi_{t}^{1}\) and \(\varphi_{t}^{2}\) are \emph{orbit equivalent} if there exists a homeomorphism \(H:N_{1}\to N_{2}\) such that for every \(x_{1}\in N_{1}\), \(H\) sends the orbit \(\mathcal{O}^{1}(x_{1})\) homeomorphically onto the orbit \(\mathcal{O}^{2}(H(x_{1}))\), preserving the orientation of orbits.

If, in addition, \(N_{1}\) and \(N_{2}\) are sutured manifolds, we require that \(H\) preserves the suture structure.
\end{definition}

\begin{definition}[Saturated subset]
Let \(\varphi\) be a semiflow on a compact manifold \(N\). A subset \(W\subset N\) is called (\(\varphi\)-)/saturated/ if for every point \(p\in W\), the entire flowline through \(p\) is contained in \(W\).

If, in addition, \(W\) is a \(I\)-bundle \(F\times I\) and the \(\varphi|_{W}\) is orbit equivalent to a product flow on \(F\times I\), then \(W\) is called \emph{product saturated}.
\label{def:saturated}
\end{definition}

For a semiflow on a sutured manifold, we add compatible condition naturally.
\begin{definition}
Let \((N,R_{+},R_{-},\gamma)\) be a sutured manifold and \(\varphi\) be a semiflow on \(N\).
We call \(\varphi\) is a \emph{compatible semiflow} in \(N\) if the following conditions hold:
\begin{enumerate}
\item Each orbit of \(\varphi\) is transverse to \(R_{+}\) and \(R_{-}\). Further, there exists \(\epsilon>0\) such that for any point \(r\in R_{+}\), \(\varphi_{[-\epsilon,0]}(r)\) exists; similarly, for any point \(r\in R_{-}\), \(\varphi_{[0,\epsilon]}(r)\) exists.
\item \(\gamma\) is a \(\varphi\)-saturated set and the restriction \(\varphi|_{\gamma}\) of \(\varphi\) on \(\gamma\) is transverse to \(\partial \gamma\).
\end{enumerate}

If, in addition,
\begin{enumerate}
\item the restriction \(\varphi|_{A(\gamma)}\) on sutured annuli is product saturated (Definition \ref{def:saturated}) and,
\item the restriction \(\varphi|_{T(\gamma)}\) on sutured tori forms a taut foliation,
\end{enumerate}
then \(\varphi\) is called a \emph{tight semiflow} in \(N\).
\label{def:tight}
\end{definition}
Clearly, the product flow is tight and serves as the natural flow analogue of a product sutured manifold.

Tight flows behave well under decomposing along transverse surface and product saturated surface.
\begin{lemma}
Suppose that \((N,R_{+},R_{-},\gamma)\) is a sutured manifold which admits a tight semiflow \(\varphi\). Let \(S\) be a decomposition surface in \(N\) such that either
\begin{enumerate}
\item \(S\) is a product annulus or product disk and is product \(\varphi\)-saturated, or,
\item \(S\) is positively transverse to \(\varphi\),
\end{enumerate}
Then, there is a natural tight semiflow structure on \(N-S\).
\label{lem:decompose}
\end{lemma}
\begin{proof}
As a compact manifold, \(N-S\) is the metric completion of \(N\setminus S\) and hence inherits a natural semiflow structure from \(N\).
If \(S\cap R_{+}\) or \(S\cap R_{-}\) is nonempty and \(S\) is positively transverse to \(\varphi\), then we blunt the corner around \(S\cap R_{\pm}\) in \(N-S\) to build a product saturated suture there.See Figure \ref{fig:blunt}.
Other cases are direct consequences of definitions.
So, the natural semiflow structure on \(N- S\) is tight.

In both case, we have another way to give a semiflow structure on \(N-S\). When \(S\) is product saturated, choose \(N_{\epsilon}(S)\) as a small product saturated open bicollar neighborhood of \(S\). When \(S\) is positively transverse to \(\varphi\), choose \(N_{\epsilon}(S)\) as a small open bicollar neighborhood of \(S\) which has product flow structure outside \(S\cap R_{+}\) and \(S\cap R_{-}\).
Then the restriction of \(\varphi\) on \(N\setminus N_{\epsilon}(S)\) gives a desired semiflow on \(N- S\).

It is easy to see the semiflow structure on \(N-S\) is independent of the choice of bicollar neighborhood. Also, two ways above give orbit equivalent semiflow on \(N-S\).

\begin{figure}[htbp]
\centering
\includegraphics[width=10cm]{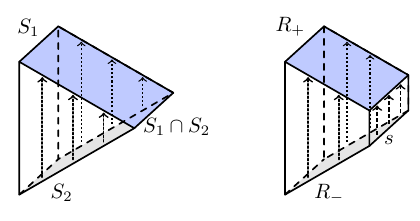}
\caption{\label{fig:blunt}Blunt the sharp corner.}
\end{figure}
\end{proof}

\begin{example}
The second simplest and most frequently encountered example is a solid torus with sutures.

Let \((T,\gamma,R_{\pm})\) be a sutured torus. The sutures consist of an even number of parallel closed curves on \(\partial T\).
\(T\) is taut if and only if the sutures are not meridians.
If \(T\) is taut, \(T\) admits a canonical tight semiflow, namely, the standard pseudo-hyperbolic local model (see Section \ref{sec:pre:pAflow} and Figure \ref{fig:4ST}).

We focus on two cases.
A \emph{4-ST} is a taut sutured torus with four longitudes as sutures.
A \emph{non-longitudinal 2-ST} is a taut sutured torus with two non-longitudinal closed curves as sutures.
\end{example}

\begin{figure}[htbp]
\centering
\includegraphics[width=6.5cm]{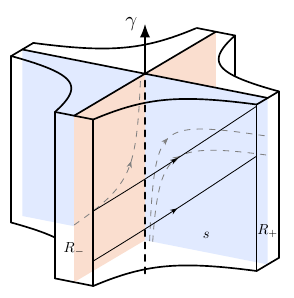}
\caption{\label{fig:4ST}A 4-ST with the standand pseudo-hyperbolic local model.}
\end{figure}
\section{Orbit space and shadows}
\label{sec:orbitspace}
A pseudo-Anosov flow \(\varphi\) on a closed, orientable 3-manifold \(M\) lifts to a flow \(\widetilde{\varphi}\) on the universal covering \(\widetilde{M}\). The \emph{orbit space} \(P_{\varphi}\), which is the quotient of \(\widetilde{M}\) by the orbits of \(\widetilde{\varphi}\), is homeomorphic to a plane, according to \autocites{Barbot1995_CaracterisationDesFlotsDanosovEnDimension3ParLeursFeuilletagesFaibles_PUB}[][]{Fenley1994_AnosovFlowsIn3Manifolds_PUB} for Anosov flows and \autocite{FenleyMosher2001_QuasigeodesicFlowsInHyperbolic3Manifolds_PUB} for pseudo-Anosov case.
It is clear that the deck transformation action of \(\pi_{1}(M)\) on \(\widetilde{M}\) takes orbits to orbits, and thus induces an action \(\pi_{1}(M)\curvearrowright P_{\varphi}\).
A point in \(P_{\varphi}\) is called \emph{singular} if it is the image of a singular closed orbit, and called \emph{regular} or \emph{non-singular} otherwise.
A point in \(P_{\varphi}\) is called a \emph{fixed point} if an element \(g\in \pi_{1}(M)\) fixes it.

By classical stable manifold theory, a transverse pair of foliations \(W^{s}_{\varphi}\) and \(W_{\varphi}^{u}\) in \(M\) encodes the asymptotic properties of \(\varphi\)-orbits.
Their lifts \(\widetilde{W}^{s}_{\varphi}, \widetilde{W}^{u}_{\varphi}\) in \(\widetilde{M}\) is preserved by the action of \(\pi_{1}(M)\) and projects to a transverse pair of \(1\)-dimensional \(\pi_{1}(M)\)-invariant singular foliations \(\mathcal{F}^{s}\) and \(\mathcal{F}^{u}\) in the orbit space \(P_{\varphi}\).
We denote the leaf of \(\mathcal{F}^{s/u}\) traversing the point \(x\in P_{\varphi}\) by \(\mathcal{F}^{s/u}(x)\).

In \autocite{Fenley2012_IdealBoundariesOfPseudoAnosovFlowsAndUniformConvergenceGroupsWithConnectionsAndApplicationsToLargeScaleGeometry_PUB}, Fenley compactified the orbit space \(P_{\varphi}\) by an ideal circle \(\partial P_{\varphi}\) at infinity and proved that the compactification result \(P_{\varphi}\cup \partial P_{\varphi}\) is a topological disk.
The \(G\)-action on \(P_{\varphi}\) extends naturally to a orientation-preserving homeomorphic \(G\)-action on the compactification disk \(P_{\varphi}\cup \partial P_{\varphi}\).
Such \(G\)-action reveals dynamics of the flow. For instance, an element \(g\in \pi_{1}(M)\) fixes a point in orbit space \(P_{\varphi}\) if and only if it represents a power of homotopy class of a closed orbit.

There are some typical objects on the orbit space \(P_{\varphi}\).
\begin{definition}[Perfect fits]
Two leaves \(F\in \mathcal{F}^{s}\) and \(G\in \mathcal{F}^{u}\), form a \emph{perfect fit} if \(F\) has a half ray \(F'\) and \(G\) has a half ray \(G'\), such that
\begin{enumerate}
\item \(F\cap G=\emptyset\), \(F'\) and \(G'\) converge to same ideal point \(p\),
\item there is leaf segments \(U\in \mathcal{F}^{u}\) and \(V\in \mathcal{F}^{s}\) such that \(F,G,U,V,p\) forms a rectangle with one ideal vertex \(R\),
\item the restriction of \(\mathcal{F}^{s}\) and \(\mathcal{F}^{u}\) on \(R\) is product (namely, conjugate to the vertical and horizontal lines on \([0,1]^{2}\setminus \{(0,0)\}\)). In particular, \(R\) contains no singular points.
\end{enumerate}

See the left side of Figure \ref{fig:perfectfits}.
\end{definition}

\begin{definition}[Lozenge]
A \emph{lozenge} is a region whose closure is homeomorphic to a rectangle-with-two-ideal-vertices \(R\) and the restriction of \(\mathcal{F}^{s/u}\) on \(R\) is product.
A lozenge has two \emph{corners} \(p,q\). \(\mathcal{F}^{s}(p)\) and \(\mathcal{F}^{u}(q)\) form a perfect fit while \(\mathcal{F}^{u}(p)\) and \(\mathcal{F}^{s}(q)\) form a perfect fit.
See the right side of Figure \ref{fig:perfectfits}.
\end{definition}

Following the notations in \autocite{BarthelmeFrankelMann2025_OrbitEquivalencesOfPseudoAnosovFlows_PUB}, a point in \(P_{\varphi}\) is called \emph{corner point} if it is a corner of some lozenge, and called \emph{non-corner point} otherwise.

\begin{figure}[htbp]
\centering
\includegraphics[width=14cm]{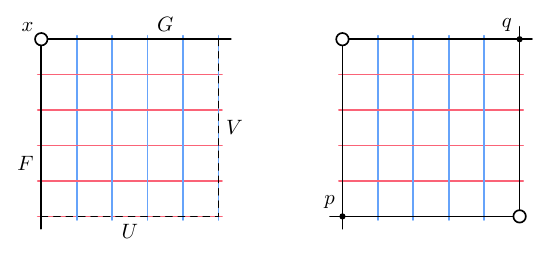}
\caption{\label{fig:perfectfits}A perfect fit and a lozenge.}
\end{figure}

\begin{definition}[scalloped region]
A \emph{scalloped region} is an open set \(U\subset P_{\varphi}\) satisfying following conditions:
\begin{enumerate}
\item The boundary \(\partial U\) consists of four biuinfinite families of leaves, \(l_{k}^{1,s},l_{k}^{2,s}\) in \(\mathcal{F}^{s}\) and \(l_{k}^{1,u},l_{k}^{2,u}\) in \(\mathcal{F}^{u}\), where \(k\in \Integer\).
\item Each \(l_{k}^{i,s}\), \(i=1,2\), contains a fixed point \(p_{k}^{i,s}\) such that \(p_{k}^{1,s},p_{k}^{2,s}\)  are two corners of a lozenge, \(p_{k}^{1,s},p_{k}^{2,s}\) are two corners of a lozenge. The union of these lozenges is equal to \(U\).
Moreover, \(\mathcal{F}^{u}(p_{k}^{i,s})\) accumulates on \(\cup_{j\in \Integer} l_{j}^{1,u}\) as \(k\to \infty\), and accumulates on \(\cup _{j\in \Integer}l_{j}^{2,u}\) as \(k\to u\infty\).
The analogous statement holds for the other two families \(l_{k}^{i,u}\).
\item The restriction of \(\mathcal{F}^{s/u}\) on \(U\) is product.
\end{enumerate}
\end{definition}
We won't discuss the scalloped region in depth in this paper; interested readers are referred to \autocite[Section 2.6]{BarthelmeFrankelMann2025_OrbitEquivalencesOfPseudoAnosovFlows_PUB} for a detailed discussion and vivid figures.
In \autocite{BarthelmeFrankelMann2025_OrbitEquivalencesOfPseudoAnosovFlows_PUB}, the authors axiomatize the \(\pi_{1}(M)\)-action on \(P_{\varphi }\) via conditions imposed on the group action on a bifoliated plane \((P,\mathcal{F}^{s},\mathcal{F}^{u})\).
For our purposes, we fix one element \(g\in G\) and analyse the \(g\)-action on the orbit space \(P_{\varphi}\).
This is analogous to the trichotomy of elements of a hyperbolic 3-manifold group according to the number of fixed points.
\begin{proposition}
Let \(M\) be a closed, orientable 3-manifold and \(\varphi\) is a pseudo-Anosov flow not conjugate to suspension Anosov flow.  Denote by \(P_{\varphi}\cup \partial P_{\varphi}\) Fenley's compactification disk of orbit space.
Let \(\mathrm{Fix}(g)\) denote the fixed point set of \(g\) in \(P_{\varphi}\), and let \(\mathrm{Fix}_{\partial }(g)\) denote the fixed point set of \(g\) in \(\partial P_{\varphi}\).

\begin{enumerate}
\item If \(g\) is elliptic, namely, \(\mathrm{Fix}(g)\) is nonempty, exactly one of the following holds:
\begin{enumerate}
\item \(\mathrm{Fix}(g)\) is a non-corner point \(x\).
In this case, for any power \(g^{l}\), \(\mathrm{Fix}(g^{l})\) equals \(x\).
For some \(k\in \Integer\), \(g^{k}\) fixes prongs of \(x\) and hence \(\mathrm{Fix}_{\partial }(g^{k})\) equals the union of all endpoints of \(\mathcal{F}^{s}(x)\) and \(\mathcal{F}^{u}(x)\).
\item \(\mathrm{Fix}(g)\) contains a corner point \(x\).
In this case, for some power \(g^{k}\), we have:
 \(\mathrm{Fix}(g^{k})\) equals the set of all corners in the maximal chain \(\mathcal{C}\) of lozenges containing \(x\); and \(\mathrm{Fix}_{\partial }(g^{k})\) equals the closure of the set of ideal endpoints of the sides of lozenges in \(\mathcal{C}\).
\end{enumerate}

\item If \(g\) acts freely on \(P_{\varphi}\), namely, \(\mathrm{Fix}(g)=\emptyset\),  then exactly one of the following holds:
\begin{enumerate}
\item \(g\) is hyperbolic; that is, \(\mathrm{Fix}_{\partial }(g)\) is two points in ideal boundary \(\partial P_{\varphi}\). In this case, \(g\) hyperbolic-like dynamics on \(\partial P_{\varphi}\) with one fixed point attracting and the other repelling.
\item \(g\) is parabolic; that is, it fixes exactly one point in \(\partial P_{\varphi}\). In this case, \(g\) has parabolic-like dynamics on \(\partial P_{\varphi}\) and fixes a perfect fit horoball.
\item \(g\) fixes exactly one scallopsed region. In this case,
\(g\) has exactly four fixed points in \(\partial P_{\varphi}\), namely, the four ideal boundary of the scallopsed region.
\end{enumerate}
\end{enumerate}
\label{prop:actionclassification}
\end{proposition}
\begin{proof}
The proof of such classification is fairly standard but scattered throughout the literature (see \autocites[Lemma 4.9]{Fenley2012_IdealBoundariesOfPseudoAnosovFlowsAndUniformConvergenceGroupsWithConnectionsAndApplicationsToLargeScaleGeometry_PUB}[Proposition 3.6]{BarthelmeFrankelMann2025_OrbitEquivalencesOfPseudoAnosovFlows_PUB} for partial discussion). We include a detailed proof for completion.

If \(g\) is elliptic and fixes a corner point \(p\) of some lozenge \(L\), there is some power \(g^{k}\) fixing all leaf rays starting from \(p\). Then, by \autocite[Theorem 3.3]{Fenley1995_QuasigeodesicAnosovFlowsAndHomotopicPropertiesOfFlowLines_PUB} (see also \autocite[Lemma 2.23]{BarthelmeFrankelMann2025_OrbitEquivalencesOfPseudoAnosovFlows_PUB} for self-contained proof), \(g^{k}\) must fix all corner points of the maximal chain of lozenges \(\mathcal{C}\) containing \(L\).
Meanwhile, according to \autocites[Theorem 4.8]{Fenley1999_FoliationsWithGoodGeometry_PUB}[Theorem 2.5]{Fenley2012_IdealBoundariesOfPseudoAnosovFlowsAndUniformConvergenceGroupsWithConnectionsAndApplicationsToLargeScaleGeometry_PUB}, if two points are both fixed by \(g\), then they are connected by a chain of lozenges and hence both corner points.
Therefore, whenever \(g\) fixes a non-corner point, it has no other fixed points.
The conclusion on fixed points of \(g^{k}\) in \(\partial P_{\varphi}\) follows directly from \autocite[Proposition 3.6]{BarthelmeFrankelMann2025_OrbitEquivalencesOfPseudoAnosovFlows_PUB}.
At this point, we finish the proof of elliptic case.

Assume that \(g\) is non-elliptic and hence fixes at least one point in the ideal boundary \(\partial P_{\varphi}\) by Brouwer fixed point theorem.

Recall that \(\mathcal{F}^{s}\) and \(\mathcal{F}^{u}\) are stable and unstable foliations on the orbit space \(P_{\varphi}\).
If \(g\) leaves invariant a leaf in \(\mathcal{F}^{s/u}\), then there is a \(\varphi\)-orbit fixed by \(g\), contradicting non-elliptic assumption.
Therefore, \(g\) acts freely on the leaf space \(\mathcal{L}^{s}\) of \(\mathcal{F}^{s}\), which is a non-Hausdorff tree.
By \autocite[Theorem A]{Fenley2002_FoliationsTopologyAndGeometryOf3Manifolds_PUB}, \(g\) leaves invariant a translation axis for its action in \(\mathcal{L}^{s}\) and acts as a translation on it.
The translation axis consists of exactly those leaves \(L\) of \(\mathcal{L}^s\) so that \(g(L)\) separates \(L\) from \(g^{2}(L)\). This implies that \(\{g^{n}(L)\}_{n\in\Integer}\) is a nested collection of leaves.
If \(g^{n}(L)\) does not escape compact sets in \(P\) as \(n\) tends to infinity, the sequence must converge to a single leaf or a union of non-separated leaves \(\Lambda\) with index set \(\mathcal{I}\).
The single leaf case can't happen since \(g\) is non-elliptic.
As stated in \autocites[Theorem 4.9]{Fenley1999_FoliationsWithGoodGeometry_PUB}[Theorem 2.6]{Fenley2012_IdealBoundariesOfPseudoAnosovFlowsAndUniformConvergenceGroupsWithConnectionsAndApplicationsToLargeScaleGeometry_PUB}, the index set \(\mathcal{I}\) is order isomorphic to either \(\Integer\) or \(\{1,\ldots,k\}\), and \(g\) preserves or reverses the order.
If \(g\) reverses the order, then \(g\) fixes a leaf or swaps two consecutive leaves \(\alpha^{s},\beta^{s}\in \Lambda\). In the latter case, by \autocites[Theorem 4.9]{Fenley1999_FoliationsWithGoodGeometry_PUB}[Theorem 2.6]{Fenley2012_IdealBoundariesOfPseudoAnosovFlowsAndUniformConvergenceGroupsWithConnectionsAndApplicationsToLargeScaleGeometry_PUB}, \(\alpha^{s}\) and \(\beta^{s}\) are separated by a unique unstable leaf \(\gamma^{u}\) and hence \(g\) fixes \(\gamma^{u}\), leading to a contradiction.
Similarly, it is impossible that \(\mathcal{I}\) is a finite set and \(g\) preserves the order.
So, \(\{g^{n}(L)\}_{n\in \Integer}\) either escapes compact sets of \(P_{\varphi}\), or accumulates to a bi-infinite union of non-separated leaves. In the latter case, \(g\) acts on the bi-infinite union by order-preserving translation.
We divide the argument into two cases.

If \(\{g^{n}(L)\}_{n\in \Integer}\) converges to a bi-infinite collection \(\Lambda\) of leaves non-separated from each other, there is a scalloped region \(\mathcal{S}\) associated to \(\Lambda\) according to \autocite[Section 4]{Fenley1998_StructureOfBranchingInAnosovFlowsOf3Manifolds_PUB}.
\(g\) acts as a translation on the four sides of \(\mathcal{S}\), otherwise \(g\) would have a fixed point in \(\mathcal{S}\).
Therefore, \(g\) has exactly four fixed points in \(\partial P_{\varphi}\), namely, the four ideal vertices of \(\mathcal{S}\), and acts by translation on the four complement segments in \(\partial P_{\varphi}\).
This is because that for each ideal point \(p\) not the ideal vertices of \(\mathcal{S}\), there is a leaf \(L_{p}\) separating \(p\) from \(\mathcal{S}\). Since \(g\) acts on \(L_{p}\) by translation and fixes \(\mathcal{S}\), \(g\) cannot fix \(p\).
Consequently, \(g\) fixes exactly one scalloped region and four ideal points, which places us in case (2c).

If \(\{g^{n}(L)\}_{n\in \Integer}\) escapes compact sets, \(\{g^{n}(L)\}_{n\in \Integer^{+}}\) defines a nested escaping sequence of convex polygonal paths and defines a unique ideal point \(s_{+}\) by the definition of ideal boundary.
Similarly, the negative part \(\{g^{n}\}_{n\in \Integer^{-}}\) defines an ideal point \(s_{-}\) in \(\partial P_{\varphi}\).
Clearly, \(g\) fixes \(s_{\pm}\), and the ideal boundary of \(g^{n}(L)\) converges to \(s_{+}\) when \(n\to +\infty\) and to \(s_{-}\) when \(n\to -\infty\).
Moreover, \(g\) maps arcs of \(\partial P_{\varphi}\) between \(g^{k}(L)\) and \(g^{k+1}(L)\) to the arcs between \(g^{k+1}(L)\) and \(g^{k+2}(L)\).

If \(s_{-}\neq s_{+}\), then they form a source/sink pair under \(g\)-action, one repelling and the other attracting. Hence \(g\) has exactly two fixed point in \(\partial P_{\varphi}\) and has hyperbolic-like dynamics in the circle \(\partial P_{\varphi}\).
We are in case (2a) now.

If \(s_{-}=s_{+}\), then \(g\) has parabolic-like dynamic action on \(\partial P_{\varphi}\), with a unique fixed point \(s_{\pm}\).
Furthermore, \(g^{n}(L),n\in \Integer\) share a common ideal point \(s_{\pm}\) and hence form perfect fits pairwisely.
Due to \autocite[Lemma 3.20]{Fenley2012_IdealBoundariesOfPseudoAnosovFlowsAndUniformConvergenceGroupsWithConnectionsAndApplicationsToLargeScaleGeometry_PUB}, \(\{g^{n}(L)\}_{n\in \Integer}\) forms a perfect fit horoball, which is invariant under \(g\).
\end{proof}

If \(\varphi\) is a pseudo-Anosov flow without perfect fits and not conjugate to a suspension Anosov flow, only case (1a) and case (2a) can happen.
To accommodate the blown-up case, a minor adaptation of the classification is required.
\subsection{Blowup settings}
\label{sec:orgc12de66}
Let \(\varphi\) be an almost pseudo-Anosov flow, which is a blowup of \(\varphi^{\natural}\).
We build the orbit space \(P_{\varphi}\) as before.

The blowdown map \(\Phi:\varphi\to \varphi^{\natural}\) can be homotopic to be identity outside a small neighborhood of the blown-up complex; furthermore, \(\Phi\) sends orbits of \(\varphi^{\sharp}\) to orbits of \(\varphi\) and preserves orientation.
Hence, it induces a map \(\Phi:P_{\varphi}\to P_{\varphi^{\natural}}\).
The orbit space of \(\varphi\) differs that of \(\varphi^{\natural}\) by the existence of blown-up segments, namely, the images of blown-up annuli.
The blowdown map \(\Phi\) takes blown-up segments to singular points; a blown-up lozenge can have one of these segments connected to its corner.

The blown-up result of stable and unstable foliation projects to singular foliations \(\mathcal{F}^{s/u}\) on \(P_{\varphi}\).
We note that the interior of a blown-up segment belong to both of \(\mathcal{F}^{s/u}\).
A \emph{stable} or \emph{unstable slice leaf} is refered to be a properly embedded \(\Real\) in \(P_{\varphi}\) that is contained in a leaf of \(\mathcal{F}^{s}\) or \(\mathcal{F}^{u}\) respectively.

\begin{corollary}
Let \(M\) be a closed, orientable 3-manifold and \(\varphi\) be an almost pseudo-Anosov flow without perfect fits and not conjugate to suspension Anosov flow.
Let \(g\in \pi_{1}(M)\) act on \(P_{\varphi}\) by deck transformation.
Then, exactly one of the following holds:
\begin{enumerate}
\item \(g\) fix a regular point in \(P_{\varphi}\)
\item \(g\) fix a singular point \(p\) in \(P_{\varphi^{\natural}}\) and hence all the vertices of blown-up tree associated to \(p\) in \(P_{\varphi}\).
\item \(g\) acts hyperbolically on \(\partial P_{\varphi}\) with one fixed point attracting and the other repelling.
\end{enumerate}
\label{prop:actionclassification_perfectfits}
\end{corollary}
For the remainder of this paper, whenever we say \(\varphi\) is without perfect fits, we implicitly assume that it is not a suspension Anosov flow.
\subsection{Shadow of almost transverse surfaces}
\label{sec:orgb75d130}
Let \(\varphi\) be an almost pseudo-Anosov flow and let \(S\) be an embedded surface transverse to \(\varphi\).
We assume that \(\varphi\) is \emph{minimal} with respect to the surface \(S\).
By this we mean that blowing down any blown-up annulus would destroy transversity with \(S\).

Let \(\widetilde{S}\) be a lift of \(S\). \(\widetilde{S}\) separates \(\widetilde{M}\) into positive and negative components according to orientation. We say that a subset \(O\subset \widetilde{M}\) \emph{lies above} (\emph{below}, resp.) \(\widetilde{S}\) if \(O\) lies in the positive (negative, resp.) component.
Any orbit of \(\widetilde{\varphi}\) intersects \(\widetilde{S}\) at most once.
Hence, the projection \(\Theta:\widetilde{S}\to \Theta(\widetilde{S})\) is injective, where \(\Theta\) is the canonical projection from \(\widetilde{M}\) to \(P_{\varphi}\).
The image \(\Theta(\widetilde{S})\) is called the \emph{shadow} of \(\widetilde{S}\).
The boundary of a shadow is well-organized.
\begin{proposition}
Assume that \(\varphi\) is a almost pseudo-Anosov flow on a closed manifold \(M\) and \(S\) is an embedded surface transverse to \(\varphi\) minimally.
We fix a universal lift \(\widetilde{S}\) of \(S\) in \(\widetilde{M}\). Then,
\begin{enumerate}
\item \(\Theta(\widetilde{S})\) is an open subset of \(P_{\varphi}\),
\item \(\Theta(\widetilde{S})\) is convex, in the sense that, for any leaf \(L\) of \(\mathcal{F}^{s/u}\), the intersection \(L\cap \Theta(\widetilde{S})\) is connected.
\end{enumerate}
\label{prop:shadow1}
\end{proposition}
\begin{proof}
The first statement follows from the fact that \(P_{\varphi}\) is Hausdorff and the local picture of \(\varphi\). The second one is due to \autocite[4.2]{Fenley2009_GeometryOfFoliationsAndFlowsI_PUB}.
\end{proof}

The following proposition is a reformulation of \autocite[Proposition 7.2]{LandryMinskyTaylor2025_TransverseSurfacesAndPseudoAnosovFlows_PUB} in the closed case, which characterizes the local figure around a singular point and its associated blown-up tree.
\begin{proposition}
Let \(T\) be a blown-up tree in the orbit space \(P_{\varphi}\).
Then one of the following holds:
\begin{enumerate}
\item \(T\cap \operatorname{fr}(\Theta(\widetilde{S}))\) is empty.
\item \(T\cap \operatorname{fr}(\Theta(\widetilde{S}))\) is the union of finitely many vertices and blown-up segments.
\end{enumerate}

In the latter case, each connected component of \(N_{\epsilon}(T)\cap \Theta(\widetilde{S})\) is either
\begin{enumerate}
\item a pair of adjacent open quadrants and the prong between them, or
\item a blown-up segment together with the two open quadrants on its two sides.
\end{enumerate}
\label{prop:shadow2}
\end{proposition}

\begin{proposition}
Let \(l\) be a component of the frontier of \(\Theta(\widetilde{S})\) in \(P_{\varphi}\). Then,
\begin{enumerate}
\item \(l\) is a slice leaf of \(\mathcal{F}^{s/u}\). If it is a slice leaf \(l\) of \(\mathcal{F}^{s}\), then as a sequence of point \(\{p_{n}\}_{n\in \Number}\) tends to a point in \(l\), the preimages \((\Theta|_{\widetilde{S}})^{-1}(p_{n})\) in \(\widetilde{S}\) escape in the positive direction.
Similar conclusion holds for unstable boundary slice leaves.
\item \(l\) is periodic, namely, there is an element \(g\in \pi_{1}(M)\) such that \(g\) fixes \(l\) and hence a unique point in \(l\).
\end{enumerate}
\label{prop:shadow3}
\end{proposition}
The first statement of Proposition \ref{prop:shadow3} follows from \autocite[Proposition 4.1]{Fenley2009_GeometryOfFoliationsAndFlowsI_PUB} or \autocite[7.8]{LandryMinskyTaylor2025_TransverseSurfacesAndPseudoAnosovFlows_PUB}.
If \(\varphi\) is transitive, the second one is proven in \autocite[8.7]{LandryMinskyTaylor2025_TransverseSurfacesAndPseudoAnosovFlows_PUB} using veering triangulation.
In general, one can adapt Fenley's approach \autocite[Section 6]{Fenley1999_SurfacesTransverseToPseudoAnosovFlowsAndVirtualFibersIn3Manifolds_PUB} to non-transitive setting with minor modifications.
We provide a concise adaptation in Appendix \ref{app:A}.

In summary, the shadow of a transverse, connected surface is a simply-connected open set bounded by (infinitely many) slice leaves of \(\mathcal{F}^{s/u}\).

We note that the case \(\partial \Theta(\widetilde{S})\neq \emptyset\) happens frequently.
\begin{corollary}
Assume that \(S\) is an embedded surface which is transverse to the almost pseudo-Anosov flow \(\varphi\). If \(S\) is not a fiber of \(M\), then \(\partial \Theta(\widetilde{S})\neq \emptyset\) and any infinite leaf ray \(L\) in \(\mathcal{F}_{S}^{\pm}\) must limit to a closed leaf and spiral toward it.
\end{corollary}
\begin{proof}
If \(\partial \Theta(\widetilde{S})= \emptyset\), then \(\widetilde{S}\) intersects all orbits of \(\widetilde{\varphi}\). Therefore, \(S\) intersect all orbits of \(\varphi\) in a uniformly bounded time, and is a fiber of \(M\), contradicting the assumption.
Proposition \ref{prop:shadow3} implies that \(\partial \Theta(\widetilde{S})\) contains at least one closed leaf.
Let \(L\) be an arbitrary infinite leaf ray of \(\mathcal{F}_{S}^{+}\). Since \(S\) is compact, it follows that a sequence \(\{p_{i}\}_{i\in \Number}\) of points in \(L\) converges to a point \(b\in S\).
By the local product structure of the flow, \(L\) limits on \(L'=\mathcal{F}_{S}^{+}(b)\).
The closure of \(L'\) in \(\mathcal{F}_{S}^{+}\) has a minimal sublamination, which must contain a closed leaf by Proposition \ref{prop:shadow3}.
Consequently, \(L'\) limits on a closed leaf, and thus so does \(L\).
\end{proof}
\section{Annulus Isotopy theorem}
\label{sec:annulusisotopy}
From now on, we always assume that \(M\) is a closed, oriented 3-manifold.
The rest of this section is to prove the following key theorem.
\begin{theorem}[Annulus Isotopy Theorem]
Assume that \(M\) is atoroidal and admits an almost pseudo-Anosov flow \(\varphi\) without perfect fits.
Let \(S\) be an embedded, compact, oriented, taut surface in \(M\) which is positively transverse to \(\varphi\) minimally.
Suppose that \(A\) is a homotopically non-trivial, oriented annulus with two boundary components lying on opposite sides of \(S\).
In other words, \(A\) is a product annulus in the sutured manifold \((M-S,S_{+},S_{-},\emptyset)\)..

Then, \(A\) is ambiently isotopic to an annulus \(A_{\varphi}\) in \(M-S\), which is either product \(\varphi\)-saturated or transverse to a further blowup of \(\varphi\).
We call that \(A_{\varphi}\) is \emph{in aligned position} with \(\varphi\).
\label{thm:annulusisotopy}
\end{theorem}
The assumption that \(S\) is taut is redundant. Indeed, by \autocite[Proposition 2.7]{Mosher1992_DynamicalSystemsAndTheHomologyNormOfA3ManifoldI_PUB}, \(\varphi\) is transitive, which implies that \(S\) is taut according to \autocite[Theorem B]{LandryMinskyTaylor2025_TransverseSurfacesAndPseudoAnosovFlows_PUB}.

Let \(\pi:\widetilde{M}\to M\) be the universal covering map of \(M\) and \(\Theta:\widetilde{M}\to P_{\varphi}\) be the canonical projection to orbit space.
Assume that \(S_{1}\) and \(S_{2}\) are the (possibly identical) connected components of \(S\) containing two boundary circles \(\gamma_{1},\gamma_{2}\) of \(A\) respectively, and \(A\) connects the negative side of \(S_{1}\) to the positive side of \(S_{2}\).
We fix a universal lift \(\widetilde{A}\subset \widetilde{M}\) of \(A\) and a generator \(g\) of \(\pi_{1}(A)\cong \Integer\) that fixes \(\widetilde{A}\).
Further, we denote by \(\widetilde{S}_{i},i=1, 2\) the lift of \(S_{i}\) containing the boundary line \(\widetilde{\gamma}_{i}\) of \(\widetilde{A}\).
Clearly, as a deck transformation, \(g\) preserves \(\widetilde{A}\) and \(\widetilde{\gamma}_{i}\).
In summary, we have two shadows of almost transverse surface \(\Theta(\widetilde{S}_{i})\), each of which is \(g\)-invariant and contains a \(g\)-invariant line \(\Theta(\widetilde{\gamma}_{i})\).

We first show that, given a \(g\)-invariant line in \(\Theta(\widetilde{S}_{1})\cap \Theta(\widetilde{S}_{2})\), one can construct a product saturated annulus isotopic to \(A\).

\begin{lemma}
Assume that \(\gamma\) is a \(g\)-invariant line in \(\Theta(\widetilde{S}_{1})\cap \Theta(\widetilde{S}_{2})\).
There is an embedded product saturated annulus \(A'\), so that \(A'\) connects the negative side of \(S_{1}\) to the positive side of \(S_{2}\) and \(\mathrm{int}(A')\) is disjoint from \(S\).
If \(M\) is atoroidal, \(A'\) is ambiently isotopic to \(A\) in \(M-S\).
% similar result for hyperbolic case holds by careful basepoint choice.
\label{lem:invariantline}
\end{lemma}
\begin{proof}
Since \(A\) connects the negative side of \(S_{1}\) to the positive side of \(S_{2}\), \(\widetilde{S}_{1}\) must lie above \(\widetilde{S}_{2}\).
Denote by \(H\subset \widetilde{M}\) be the closed subset bounded by \(\widetilde{S}_{1}\) and \(\widetilde{S}_{2}\).
The intersection \(H\cap \widetilde{\varphi}_{\Real}(\gamma)\) is homeomorphic to \(\Real\times [0,1]\) and \(g\)-invariant.
Therefore, \(A'=\pi(H\cap \widetilde{\varphi}_{\Real}(\gamma))\) is an immersed product saturated annulus in \(M\), which connects the negative side of \(S_{1}\) to the positive side of \(S_{2}\).
The boundary component of \(H\cap \widetilde{\varphi}_{\Real}(\gamma)\) in \(\widetilde{S}_{i}\) projects to an immersed closed curve \(\gamma_{i}'\) in \(S_{i}\).
Along flowlines, one defines a canonical homeomorphism \(\eta:\gamma_{2}'\to \gamma_{1}'\).
If \(\gamma_{2}'\) is not simple, then there is a bigon \(B_{2}\) of \(\gamma_{2}'\) in \(S_{2}\).
Since \(S_{2}\) is incompressible and \(M\) is irreducible, \(\eta(\partial B_{2})\subset \gamma_{1}'\) bounds a disk (bigon) \(B_{1}\) in \(S_{1}\); \(B_{1}\), \(B_{2}\), and the flowlines joining their boundaries bound a 3-ball in \(M\) disjoint from \(S_{1}\cup S_{2}\).
Hence, via a cut-and-paste isotopy of \(A'\), the bigon is removed and \(A'\) remains saturated.
In summary, we obtain an embedded product saturated annulus \(A'\).

We note that, different \(g\)-invariant lines on a fixed lift \(\widetilde{S}\) project to isotopic closed curves on \(S\).
Up to isotopy of \(A\), assume that the boundaries of \(A'\) and \(A\) are disjoint and parallel in \(S_{1}\cup S_{2}\).
Then \(A\cap A'\) is a union of homotopically trivial circles and core circles. We can remove these circles one by one via an isotopy of \(A\) because \(M\) is irreducible and atoroidal.
As a result, \(A\) is parallel to \(A'\).

It remains to prove that the interior of \(A'\) is disjoint from \(S\). Otherwise, assume that \(S_{0}\subset S\) intersects \(\operatorname{int}(A')\).
Then \(S_{0}\) has a lift \(\widetilde{S}_{0}\) lying below \(\widetilde{S}_{1}\) and above \(\widetilde{S}_{2}\), so \(\widetilde{S}_{1}\) and \(\widetilde{S}_{2}\) are in different components of \(\widetilde{M}\setminus \widetilde{S}_{0}\). Consequently, \(\widetilde{A}\) must intersect \(\widetilde{S}_{0}\), contradicting that \(\operatorname{int}(\widetilde{A})\) is disjoint from \(S\).
\end{proof}

Our goal becomes finding a \(g\)-invariant line in \(\Theta(\widetilde{S}_{1})\cap \Theta(\widetilde{S}_{2})\).
However, it does not always occur.
We divide into several cases according to the classification of \(g\)-action in Corollary \ref{prop:actionclassification_perfectfits}.

We first show that it happens when \(g\) act hyperbolically on the orbit space.
\begin{lemma}
If \(g\) acts freely on \(P_{\varphi}\) and is of hyperbolic type, then \(\Theta(\widetilde{S}_{1})\cap \Theta(\widetilde{S}_{2})\) is non-empty and contains a \(g\)-invariant line \(\gamma\).
\label{lem:hyperbolic_invariantline}
\end{lemma}

To this end, we figure out all intersection cases between two slice leaves.
To simplify the notation, for two distinct points \(p,q\in \partial P_{\varphi}\), we denote by \((p,q)\) the open clockwise arc of \(\partial P_{\varphi}\) from \(p\) to \(q\).
\begin{lemma}
Let \(a,b\) be two different slice leaves in \(\mathcal{F}^{\pm}\).
Then \(a\) and \(b\) intersects at at most one point.

Suppose that, in addition, there are two distinct points \(p,q\) in \(\partial P_{\varphi}\) such that \(\partial a\in (p,q)\) and \(\partial b\in (q,p)\).
Denote by \(R_{a},R_{b}\) the complementary regions of \(a,b\) whose boundaries avoid \(p,q\), respectively.
Then, \(R_{a}\cap R_{b}= \emptyset\).
Further, \(\overline{R}_{a}\cap \overline{R}_{b}\neq \emptyset\) if and only if \(a\cap b\) is a singular point \(p\); in this case, \(\overline{R}_{a}\cap \overline{R}_{b}=p\) and we say that \(a\) and \(b\) are \emph{tangent} at \(p\).
\label{lem:sliceleaves_intersection}
\end{lemma}
\begin{proof}
If both \(a\) and \(b\) are in \(\mathcal{F}^{+}\) (or both in \(\mathcal{F}^{-}\)), then the conclusion is immediate.
So without loss of generosity, assume that \(a\subset \mathcal{F}^{+},b\subset \mathcal{F}^{-}\).
In nutshell, since two complementary regions of \(a\) are both convex, it follows from Claim 1 in \autocite[Lemma 3.6]{Fenley2012_IdealBoundariesOfPseudoAnosovFlowsAndUniformConvergenceGroupsWithConnectionsAndApplicationsToLargeScaleGeometry_PUB} that \(a\) and \(b\) intersect at at most one point.
We now provide a brief elaboration for completion.
Recall that a \emph{polygonal path} in \(P_{\varphi}\) is a properly embedded, bi-infinite path formed by concatenatening by finitely many leaf segments and leaf rays in \(\mathcal{F}^{\pm}\).
A region \(R\) with polygonal boundaries is called \emph{convex} if any local part at any boundary vertex of its complementary region \(R^{c}\) is not a single quadrant.
By definition, both complementary region of a slice leaf are convex.
If \(a\cap b\) contains at least two points, then there is a compact disk whose boundary consists of a subinterval of \(a\) and a subinterval of \(b\).
By calculating Euler characteristic via Poincaré index formula, existence of such disk contradicts convexness of \(R_{b}\).
Therefore, \(a\) and \(b\) intersect at at most one point.

Suppose that there are two distinct points \(p,q\) in \(\partial P_{\varphi}\) such that \(\partial a\in (p,q)\) and \(\partial b\in (q,p)\).
If \(R_{a}\cap R_{b}\neq \emptyset\), then \(a\cap b\) contains at least two points, which is a contradiction.
Assume that \(\overline{R}_{a}\cap \overline{R}_{b}\) is a point \(p\). Since \(R_{a}\cap R_{b}=\emptyset\) and \(a\) and \(b\) are different, it follows that \(p\) has at least \(6\) quadrants. So \(p\) is a singular point.
\end{proof}

\begin{proof}[Proof of Lemma \ref{lem:hyperbolic_invariantline}]
Let \(p\in \partial P_{\varphi}\) be the attracting point of \(g\)-action and \(q\in \partial P_{\varphi}\) be the repelling point.
According to Corollary \ref{prop:actionclassification_perfectfits}, \(\partial \widetilde{\gamma}_{i}\) is equal to \(\{p,q\}\) by \(g\)-invariance.

We first suppose that \(\widetilde{\gamma}_{1}\) and \(\widetilde{\gamma}_{2}\) are disjoint.
(Althought this proof could be merged into the general case, we present it separately because it is intuitive and enlightening.)
Denote by \(R_{1}\) (resp. \(R_{2}\)) the component of the complement of  \(\widetilde{\gamma}_{1}\) (resp. \(\widetilde{\gamma}_{2}\)) containing the other curve.
Without loss of generosity, assume that \(\partial R_{1}= \widetilde{\gamma}_{1}\cup (p,q)\) and \(\partial R_{2}= \widetilde{\gamma}_{2}\cup (q,p)\).
Clearly, \(R_{1}\cap R_{2}\) is homeomorphic to a \(g\)-invariant open disk and, the intersection of their closures \(\overline{R}_{1}\cap \overline{R}_{2}\) is homeomorphic to a \(g\)-invariant closed disk.
See Figure \ref{fig:orbitspace-curves}.

\begin{figure}[htbp]
\centering
\includegraphics[width=17cm]{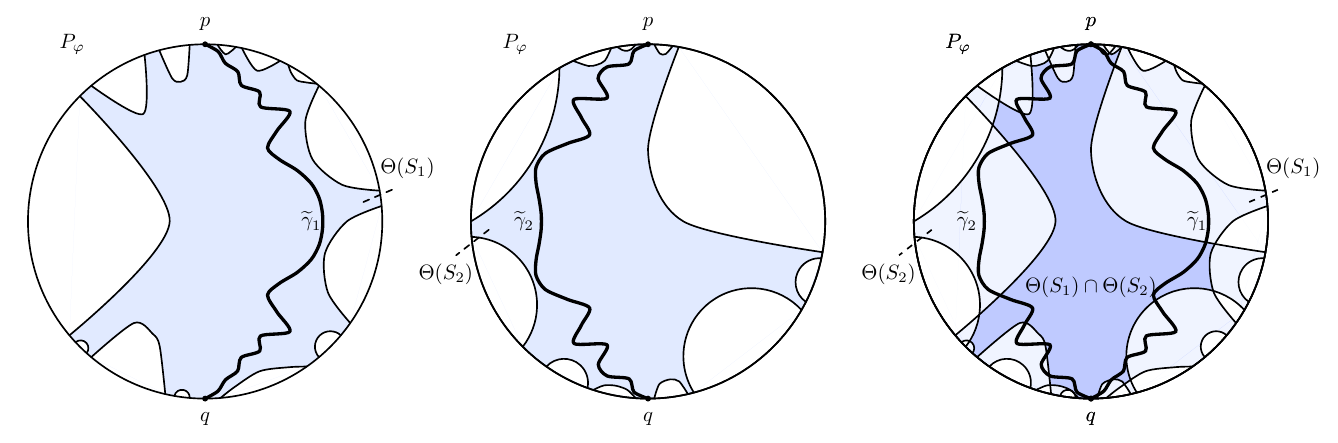}
\caption{\label{fig:orbitspace-curves}The shadows \(\Theta(\widetilde{S}_{1})\) and \(\Theta(\widetilde{S}_{2})\), and their intersection.}
\end{figure}

By Proposition \ref{prop:shadow1} and Proposition \ref{prop:shadow3}, \(\Theta(\widetilde{S}_{i})\) is an open region, bounded by (infinitely many) slice leaves and containing \(\widetilde{\gamma}_{i}\).
In other words,
\[ \Theta(\widetilde{S}_{i})= P_{\varphi}\setminus \bigcup_{k\in \mathcal{I}} \overline{R}_{\widetilde{\beta}^{i}_{k}}, \]
where \(\{\widetilde{\beta}^{i}_{k}\}_{k\in \mathcal{I}}\) is the set of all boundary slice leaves of \(\Theta(\widetilde{S}_{i})\) and \(R_{\widetilde{\beta}^{i}_{k}}\) denotes the unique complementary region of \(\widetilde{\beta}^{i}_{k}\) disjoint from \(\Theta(\widetilde{S}_{i})\).
By \(g\)-invariance of \(\Theta(\widetilde{S}_{i})\), the ideal endpoints of \(\widetilde{\beta}^{i}_{k}\) either both lie in \((p,q)\), both lie in \((q,p)\), or are exactly \(\{p,q\}\).
The last case cannot occur.
Otherwise, if a boundary slice leaf \(l=\widetilde{\beta}_{k}^{i}\) (for some \(i,k\)) has endpoints \(p,q\), then \(l\) is \(g\)-invariant. This is because given two distinct points on \(\partial P_{\varphi}\), there is at most one slice leaf connecting them.
However, \(l\) contains a unique fixed point \(f\) by Proposition \ref{prop:shadow3}, and hence an infinite set \(\{g^{t}(f)\}_{t\in \Integer}\) of fixed points, which is a contradiction.

Set
\[ \mathcal{R}=(\overline{R}_{1}\cap \overline{R}_{2})\setminus \left(( \bigcup_{k\in \mathcal{I}}\overline{R}_{\widetilde{\beta}_{k}^{1}}) \cup (\bigcup_{k\in \mathcal{I}} \overline{R}_{\widetilde{\beta}^{2}_{k}}) \cup \partial P_{\varphi}  \right) \subset \Theta(\widetilde{S}_{1})\cap \Theta(\widetilde{S}_{2}). \]
We aim to prove that \(\mathcal{R}\) is non-empty and contains a \(g\)-invariant bi-infinite line.
Each \(\overline{R}_{\widetilde{\beta}_{k}^{i}}\) is a ``half plane'' disjoint from \(\widetilde{\gamma}_{i}\) for \(i=1,2\). Therefore,
\[ \mathcal{R}=(\overline{R}_{1}\cap \overline{R}_{2})\setminus \left(( \bigcup_{k\in \mathcal{I}_{1}}\overline{R}_{\widetilde{\beta}_{k}^{1}}) \cup (\bigcup_{k\in \mathcal{I}_{2}} \overline{R}_{\widetilde{\beta}^{2}_{k}}) \cup \partial P_{\varphi}  \right) \subset \Theta(\widetilde{S}_{1})\cap \Theta(\widetilde{S}_{2}). \]
where \(\mathcal{I}_{1}=\{k\in \mathcal{I}: \partial \widetilde{\beta}_{k}^{1}\in (p,q) \}\) and \(\mathcal{I}_{2}=\{k\in \mathcal{I}: \partial \widetilde{\beta}_{k}^{2}\in (q,p)\}\).
We claim that \(\widetilde{\beta}_{k}^{1},k\in \mathcal{I}_{1}\) is disjoint from \(\widetilde{\beta}_{j}^{2},j\in \mathcal{I}_{2}\).
Otherwise, by Lemma \ref{lem:sliceleaves_intersection}, \(\widetilde{\beta}_{k}^{1}\) is tangent to \(\widetilde{\beta}_{j}^{2}\) at a singular point \(x\). However, around \(x\), \(\widetilde{\beta}_{k}^{1}\) bounds only two quadrants and so does \(\widetilde{\beta}_{j}^{2}\), which is a contradiction.

\begin{figure}[htbp]
\centering
\includegraphics[scale=0.7]{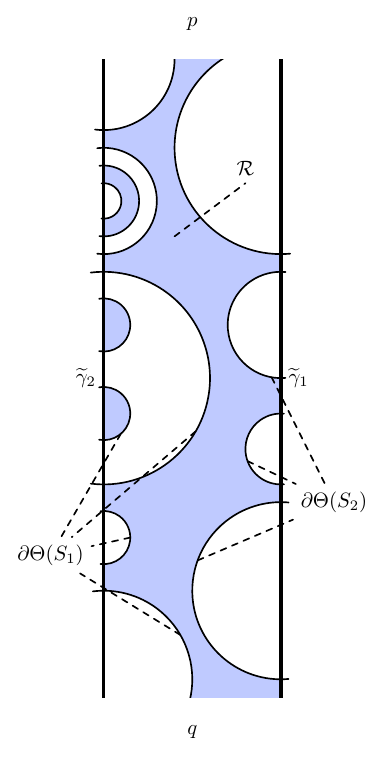}
\caption{\label{fig:orbitspace-curves2}The region \(\mathcal{R}\). It is a \(g\)-invariant infinite strip removing a \(g\)-invariant collection of disjoint closed ``half disks'' on two sides.}
\end{figure}

We note that \(\overline{R}_{\widetilde{\beta}^{i}_{k}}\cap \mathcal{R}\) is a collection of closed disks. Hence, \(\mathcal{R}\) is a \(g\)-invariant infinite strip removing a \(g\)-invariant collection of disjoint closed ``half disks'' on two sides, see Figure \ref{fig:orbitspace-curves2}.
So the interior of \(\mathcal{R}\) is a homeomorphic to an open strip, which is non-empty and contains a \(g\)-invariant line.
We mention one choice of such \(g\)-invariant line.
We can begin with \(\widetilde{\gamma}_{1}\) and replace each subinterval not in \(\Theta(\widetilde{S}_{1})\cap \Theta(\widetilde{S}_{2})\) by a subinterval of \(\partial \Theta(\widetilde{S}_{2})\). We perturb the resulting line so that it lies in \(\Theta(\widetilde{S}_{1})\cap \Theta(\widetilde{S}_{2})\).

In the general case, suppose that \(\widetilde{\gamma}_{1}\) intersects \(\widetilde{\gamma}_{2}\).
We begin with \(\widetilde{\gamma}_{1}\) and modify it step by step. Let \(u\) be a maximal closed subinterval of \(\widetilde{\gamma}_{1}\) not in \(\Theta(\widetilde{S}_{1})\cap \Theta(\widetilde{S}_{2})\).
So \(\Theta(\widetilde{S}_{2})\) has a boundary slice leaf \(\widetilde{\beta}^{2}_{j}\) such that \(u\) is a connected component of \(\overline{R}_{\widetilde{\beta}_{j}^{2}}\cap \widetilde{\gamma}_{1}\).
To simplify the notation, let \(\widetilde{\beta}_{2}= \widetilde{\beta}_{j}^{2}\) and \(D=\overline{R}_{ \widetilde{\beta}^{2}_{j}}\).
\(u\) separates \(D\) into two disks: one is compact in \(P_{\varphi}\), denoted by \(D_{c}\); the other has non-compact interior in \(P_{\varphi}\). Clearly, \(\partial D_{c}\) is the union of \(u\) and a subinterval \(t\) of \(\widetilde{\beta}\).

We claim that \(t\subset \Theta(\widetilde{S}_{1})\). Otherwise, there exists boundary slice leaf \(\widetilde{\beta}_{1}=\widetilde{\beta}_{k}^{1}\) of \(\Theta(\widetilde{S}_{1})\) with \(t\cap \widetilde{\beta}_{1}\neq \emptyset\), because \(\partial t=\partial u\subset \Theta(\widetilde{S}_{1})\).
If \(\widetilde{\beta}_{1}\) intersects the interior of \(D_{c}\), then \(t\cap \widetilde{\beta}_{1}\) contains at least two points, contradicting Lemma \ref{lem:sliceleaves_intersection}.
So \(t\) and \(\widetilde{\beta}_{1}\) intersects at a point \(x\) tangently. But around \(x\), \(t\) bounds exactly two quadrants outside \(D_{c}\). So \(t=\widetilde{\beta}_{1}\), which contradicts that \(\partial t\subset \Theta(\widetilde{S}_{1})\).

So we replace \(u\) by \(t\) and perturb \(t\) such that \(t\subset \Theta(\widetilde{S}_{1})\cap \Theta(\widetilde{S}_{2})\).
One can perform this operation \(g\)-invariantly.
As a result, we obtain an \(g\)-invariant immersed line \(l'\) in \(\Theta(\widetilde{S}_{1})\cap \Theta(\widetilde{S}_{2})\) with ideal boundary \(\{p,q\}\) and possibly infinitely many self-intersections.
Without loss of generosity, assume that all self-intersections of \(l'\) are transverse. Doing cut-and-paste at each self-intersection and discarding all circles, \(l'\) becomes a \(g\)-invariant embedded line in \(\Theta(\widetilde{S}_{1})\cap \Theta(\widetilde{S}_{2})\).
This is what we seek.
\end{proof}

\begin{remark}
If \(\varphi\) is a suspension almost pseudo-Anosov flow (without perfect fits), then the orbit space admit a singular flat metric compatible with the pair of singular foliations.
By attempting to minimize the length of \(\widetilde{\gamma}_{i}\) in its homotopic class, one may obtain a different proof.
In general, however, the orbit space is purely topological and admits no meaningful metric.
\end{remark}

It remains to see what happens when \(g\) acts elliptically (fixes a point).
We begin with a useful lemma.
\begin{lemma}
If \(g\in \pi_{1}(M)\) fixes a non-corner point \(x\in P_{\varphi}\) and fixes \(\Theta(\widetilde{S})\) for some lift \(\widetilde{S}\), then \(x\in \partial \Theta(\widetilde{S})\).
\label{lem:elliptic-boundary}
\end{lemma}
\begin{proof}
The proof is standard.
Let \(\widetilde{\gamma}\) be the orbit corresponding to \(x\).
\(x\) cannot lie in \(\Theta(\widetilde{S})\): otherwise we could choose a point \(h\) in \(\widetilde{\gamma}\cap \widetilde{S}\), and by the \(g\)-invariance of \(x\) and \(\widetilde{S}\), \(g^{k}(p)\) would also belong to their intersection, contradicting that \(\#(\widetilde{\gamma}\cap \widetilde{S})\leq 1\).
By the structure of \(\partial \Theta(\widetilde{S})\), there is a unique boundary component \(l\in \mathcal{F}^{-}\) separating \(x\) and \(\widetilde{\beta}\), hence fixed by \(g\).
Due to Proposition \ref{prop:shadow3}, \(g\) fixes a point \(y\) in \(l\).
Since \(x\) is non-corner, \(x=y\) and hence \(x\in \partial \Theta(\widetilde{S})\).
\end{proof}
In this case, \(S\) contains a closed curve which is homotopic to a power of the closed orbit associated to \(x\).

\begin{lemma}
If \(g\) fixes a unique (non-corner) regular point \(p\) in \(P_{\varphi}\), then \(\Theta(\widetilde{S}_{1})\cap \Theta(\widetilde{S}_{2})\) is non-empty and contains a \(g\)-invariant line \(\gamma\).
\label{lem:regular-invariantline}
\end{lemma}
\begin{proof}
Since \(p\) is non-corner, it must lie on \(\partial \Theta(\widetilde{S}_{i})\) by Lemma \ref{lem:elliptic-boundary}
Consequently, each of \(\Theta(\widetilde{S}_{i})\) occupies exactly two of the four quadrants around \(p\).
Assume that \(\Theta(\widetilde{S}_{1})\) and \(\Theta(\widetilde{S}_{2})\) are separating by \(\mathcal{F}^{+}(p)\). Then by Lemma \ref{prop:shadow3}, the flowline \(\varphi_{\Real}(p)\) lies below \(\widetilde{S}_{1}\) and \(\widetilde{S}_{2}\). Since both \(\widetilde{S}_{1}\) and \(\widetilde{S}_{2}\) are separating in \(\widetilde{M}\), it follows that \(\widetilde{A}\) cannot connect the negative side of \(\widetilde{S}_{1}\) to the positive side of \(\widetilde{S}_{2}\), which is contradiction.
Therefore, \(\Theta(\widetilde{S}_{1})\cap \Theta(\widetilde{S}_{2})\)  occupies at least one quadrant around \(p\).
As a result, there is a \(g\)-invariant line \(\gamma\) lying entirely in such a quadrant (see Figure \ref{fig:productannuli2}).
Moreover, if \(\Theta(\widetilde{S}_{1})\cap \Theta(\widetilde{S}_{2})\) occupies two quadrants around \(p\), we can choose the prong of \(p\) in the intersection as the desired \(g\)-invariant line \(\gamma\), more canonically.

\begin{figure}[htbp]
\centering
\includegraphics[width=13cm]{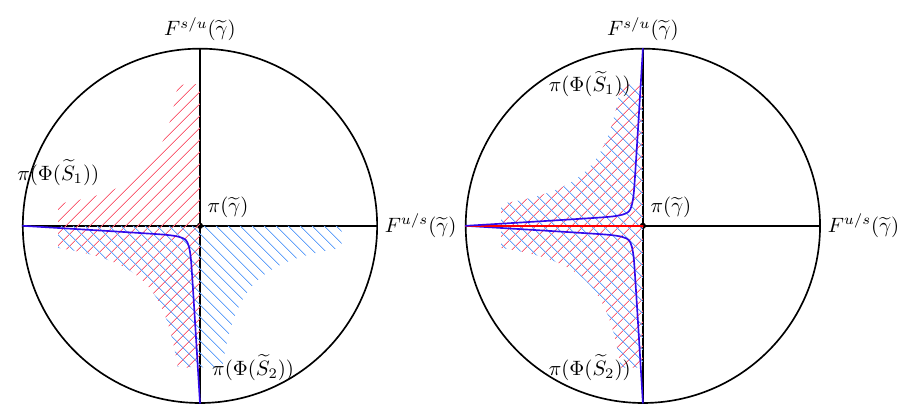}
\caption{\label{fig:productannuli2}The local picture of the product \(\varphi\)-saturated annulus in \(P_{\varphi}\) (thick red and blue lines).}
\end{figure}

Due to the characterization after Lemma \ref{lem:elliptic-boundary}, the local pictures in \(M\) of these choices are illustrated in Figure \ref{fig:productannuli1}.

\begin{figure}[htbp]
\centering
\includegraphics[width=17cm]{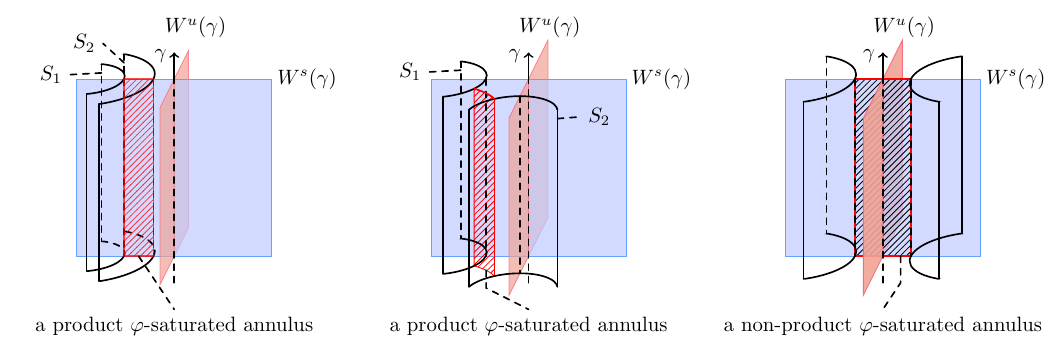}
\caption{\label{fig:productannuli1}The local picture of the (non-)product \(\varphi\)-saturated annulus in \(M\).}
\end{figure}
\end{proof}

The case when \(g\) fixes a singular point is more involved.
\begin{lemma}
If \(g\) fixes a non-corner singular point in \(P_{\varphi}\), then \(A\) is either isotopic to a product saturated annulus or isotopic to a further blow up of \(\varphi\).
\label{lem:singular-invariantline}
\end{lemma}
\begin{proof}
Assume that \(\varphi\) blows down to a pseudo-Anosov flow \(\varphi^{\natural}\). Then \(g\) fixes a unique non-corner singular point \(p_{\natural}\) in \(P_{\varphi^{\natural}}\) and the blown up tree \(T\in P_{\varphi}\) associated to \(p_{\natural}\).

As in Lemma \ref{lem:regular-invariantline}, \(\Theta(\widetilde{S}_{i})\) intersects \(N_{\epsilon}(T)\) with exactly two quadrants \(Q_{i},Q_{i}'\) and a prong or blown-up segment \(\widetilde{r}_{i}\). If \(\{Q_{1},Q_{1}'\}\cap \{Q_{2},Q_{2}'\}\)  is non-empty, we could find a \(g\)-invariant line in \(\Theta(\widetilde{S}_{1})\cap \Theta(\widetilde{S}_{2})\).
We note that \(g\) represents some multiple of the homotopy class of the singular orbit associated to \(p_{\natural}\), and fixes every vertex of \(T\) since \(g\) preserves \(\Theta(\widetilde{S}_{i})\) and \(T\).
From now on, we assume that \(\{Q_{1},Q_{1}'\}\cap \{Q_{2},Q_{2}'\}\) is empty. So \(\Theta(\widetilde{S}_{1})\) and \(\Theta(\widetilde{S}_{2})\) are disjoint.

\(\widetilde{r}_{i}\subset \widetilde{S}_{i}\) is \(g\)-invariant and projects to a simple closed curve \(r_{i}\) in \(S_{i}\) which is isotopic to one component of \(\partial A\).
Our goal is to find an annulus \(A'\) whose boundaries are isotopic to \(r_{1}\cup r_{2}\) and which is transverse to a further blowup of \(\varphi\).

Let \(\lambda_{i}\in \partial \Theta(\widetilde{S}_{i}),i=1,2\) be the unique boundary component which separates \(\Theta(\widetilde{S}_{1})\) and \(\Theta(\widetilde{S}_{2})\). Due to the existence of \(\widetilde{A}\) and orientation arguments, it is impossible that \(\lambda_{1}\) and \(\lambda_{2}\) are both stable or both unstable.
The quadrants between \(\Theta(\widetilde{S}_{1})\) and \(\Theta(\widetilde{S}_{2})\) can be grouped into two fans, each consisting of an odd number of quadrants.
We claim that \(\lambda_{1}\cap \lambda_{2}\) is a vertex of \(T\).
Otherwise, let \(b\subset T\) be a blown-up segment between \(\lambda_{1}\) and \(\lambda_{2}\). Since \(S\) is transverse to \(\varphi\) minimally, there exists \(S_{3}\subset S\) and a lift \(\widetilde{S}_{3}\) such that \(\Theta(\widetilde{S}_{3})\cap T=b\).
Consequently, \(\widetilde{S}_{3}\) is separating, disjoint from \(\widetilde{S}_{i},i=1,2\) and lies between them. \(\widetilde{A}\) must intersect \(\widetilde{S}_{3}\), contradicting that \(\operatorname{int}(\pi(\widetilde{A}))\subset M\setminus S\).

We now have exactly two ways to blow up \(\varphi\) further and minimally, each admitting an annulus \(\widetilde{A}'\) between \(\widetilde{S}_{1}\) and \(\widetilde{S}_{2}\) transverse to the new flow (see Figure \ref{fig:ex-furtherblowup}).

\begin{figure}[htbp]
\centering
\includegraphics[width=17cm]{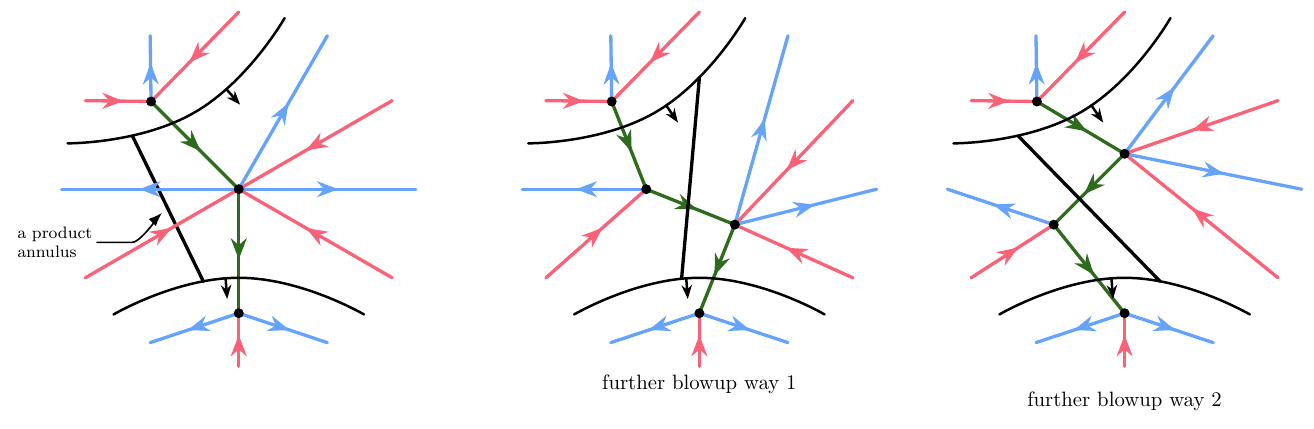}
\caption{\label{fig:ex-furtherblowup}Choose a local section near the closed orbit associated to \(p_{\natural}\); we draw the first return map. The green segments correspond to blown-up annuli; the black lines correspond to the embedded surface and the product annulus. Assume the four quadrants occupied \(Q_{1},Q_{1}',Q_{2},Q_{2}'\) are arranged clockwisely. There is two further blowups: one making \(Q_{1},Q_{2}\) adjacent, the other making \(Q_{1}',Q_{2}'\) adjacent. Each blowup admits a transverse product annulus isotopic to the original annulus.}
\end{figure}

 % We will see in Section TODO that the two blowup ways do not affect the orbit equivalence class of the flow of \(M\setminus (S\cup A')\).

Finally, \(\widetilde{A}'\) projects to an embedded annulus \(A'\) transverse to \(\varphi^{\sharp}\) and by Lemma \ref{lem:invariantline}, \(A'\) is isotopic to \(A\) and disjoint from \(S\).
\end{proof}

We have now dealt with all cases that occur when \(\varphi\) is without perfect fits and hence finished the proof of Proposition \ref{thm:annulusisotopy}.

We state a corollary for later use.
\begin{corollary}
Assume that \(M\) admits an almost pseudo-Anosov flow \(\varphi\) without perfect fits, and \(S\) is an embedded compact oriented taut surface in \(M\) positively transverse to \(\varphi\).
Let \(A\) be a homotopically non-trivial, oriented annulus with boundary components lying on the same side of \(S\). In other words, \(A\) is a non-product annulus in the sutured manifold \((M-S,S_{+},S_{-}, \emptyset)\).

Then, the core circle of \(A\) is isotopic to a multiple of some closed orbit.
\label{cor:nonproductannulus}
\end{corollary}
\begin{proof}
As in the proof of Proposition \ref{thm:annulusisotopy}, we fix a lift \(\widetilde{A}\) of \(A\), which connects lifts \(\widetilde{S}_{1}\) and \(\widetilde{S}_{2}\), and an element \(g\in \pi_{1}(M)\) fixes \(\widetilde{A}\), \(\widetilde{S}_{1}\) and \(\widetilde{S}_{2}\).
We claim that \(g\) is not of hyperbolic type, which implies that \(g\) fixes a point in \(P_{\varphi}\) and the core circle of \(A\) is isotopic to a multiple of some closed orbit due to Lemma \ref{lem:elliptic-boundary}.
Otherwise, assume that \(g\) is hyperbolic.
By Lemma \ref{lem:hyperbolic_invariantline}, \(\widetilde{S}_{1}\cap \widetilde{S}_{2}\) is non-empty and we choose \(p \in \widetilde{S}_{1}\cap \widetilde{S}_{2}\). The orbit \(\Theta^{-1}(p)\) intersects \(\widetilde{S}_{1}\) and \(\widetilde{S}_{2}\) positively. Since \(\widetilde{S}_{1},\widetilde{S}_{2}\) are separating, \(\widetilde{A}\) is impossible to connect \(\widetilde{S}_{1}\) and \(\widetilde{S}_{2}\) by orientation arguments.
This completes the proof.
\end{proof}
\subsection{Cases with perfect fits}
\label{sec:org65825c6}
One might ask whether the ``without perfect fits'' assumption can be removed.
In some cases, similar results hold (Lemma \ref{lem:parabolic-invariantline}, Lemma \ref{lem:scalloped-invariantline}).
But in general, when \(g\) fixes a corner of a lozenge chain, a product annulus in the aligned position is not guaranteed.
\begin{lemma}
If \(g\) acts freely on \(P_{\varphi}\) and is of parabolic type, then \(\Theta(\widetilde{S}_{1})\cap \Theta(\widetilde{S}_{2})\) is non-empty and contains a \(g\)-invariant line \(\gamma\).
\label{lem:parabolic-invariantline}
\end{lemma}
\begin{proof}
Let \(p\) be the fixed point of \(g\) in \(\partial P_{\varphi}\).
Both ideal endpoints of \(\widetilde{\gamma}_{i}\) are \(p\) by \(g\)-invariance. \(\widetilde{\gamma}_{i}\cup \{p\}\) bounds an open disk \(D_{i}\) in \(P_{\varphi}\cup \partial P_{\varphi}\).
Since two ideal endpoints of a slice leaf in \(\mathcal{F}^{\pm}\) cannot be the same, it follows that \(D_{i}\) is contained in \(\Theta(\widetilde{S}_{i})\). There is exactly one component \(C\) of \(D_{1}\cap D_{2}\) whose closure contains \(p\).
Moreover, \(C\) is a \(g\)-invariant open disk in \(\Theta(\widetilde{S}_{1})\cap \Theta(\widetilde{S}_{2})\). So the boundary of \(C\) is a desired \(g\)-invariant line in \(\Theta(\widetilde{S}_{1})\cap \Theta(\widetilde{S}_{2})\).
\end{proof}

\begin{lemma}
If \(g\) acts freely on \(P_{\varphi}\) but fixes a scalloped region, then \(\Theta(\widetilde{S}_{1})\cap \Theta(\widetilde{S}_{2})\) is non-empty and contains a \(g\)-invariant line \(\gamma\).
\label{lem:scalloped-invariantline}
\end{lemma}
\begin{proof}
Denote by \(D\) the scalloped region fixed by \(g\).
We claim that \(D\) is contained in both \(\Theta(\widetilde{S}_{i}),i=1,2\) and hence in \(\Theta(\widetilde{S}_{1})\cap \Theta(\widetilde{S}_{2})\).
In fact, each component of \(\partial \Theta(\widetilde{S}_{i})\) must be disjoint from \(D\). Otherwise, assume that \(l\subset \partial \Theta(\widetilde{S}_{i})\) intersects \(D\). Then, \(g^{k}(l), k\in \Integer\)  are parallel deck translates of \(l\) and are all in \(\partial \Theta(\widetilde{S}_{i})\), which contradicts the structure of scalloped regions.
Therefore, each component \(l\) of \(\partial \Theta(\widetilde{S}_{i})\) is contained in some half disk of \(P_{\varphi}\setminus D\). Further, \(\Theta(\widetilde{S}_{i})\) lies on the side of \(l\) which contains \(D\) and we finish the proof of the claim.
\end{proof}

Unfortunately, a similar result does not hold when \(g\) fixes a corner point of a chain of lozenges.
See Figure \ref{fig:counterexample-perfectfits}. \(\gamma_{0}\) and \(\gamma_{1}\) form a perfect fits; \(\gamma_{0}\) and \(\gamma_{2}\) form a perfect fits.
If \(g\) is a lift of the fundamental group of the product annulus \(A\) between \(S_{1}\) and \(S_{2}\), then \(g\) fixes the chain of lozenges between \(\gamma_{2},\gamma_{0},\gamma_{1}\).
One cannot isotope \(A\) to be a product \(\varphi\)-saturated one, or to be transverse to a further blowup.

\begin{figure}[htbp]
\centering
\includegraphics[scale=1.3]{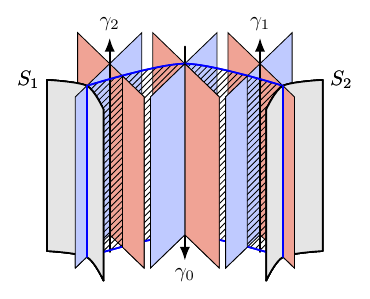}
\caption{\label{fig:counterexample-perfectfits}The local picture of a counterexample when \(g\) fixes a corner point of a chain of lozenges. The product annulus between \(S_{1}\) and \(S_{2}\) is indicated by the blue border and diagonal hatching.}
\end{figure}

\begin{remark}
In \autocite[Section 3]{HuangTaylor2026_PseudoAnosovFlowsHyperbolicGeometryAndTheCurveGraph_ARXv1}, Huang and Taylor prove that for any pseudo-Anosov flow \(\varphi\), a product annulus in \(M-S\) is either isotopic to a product saturated one, or its two boundaries are isotopic to (un)stable curves.
Their proof proceeds similarly to ours, trying to find a \(g\)-invariant line, but with a different classification scheme.
Our argument implicitly leads to their result, directly by Lemma \ref{lem:hyperbolic_invariantline}, \ref{lem:regular-invariantline}, \ref{lem:singular-invariantline}, \ref{lem:parabolic-invariantline}, \ref{lem:scalloped-invariantline}.
The obstruction to the stronger statement in Theorem \ref{thm:annulusisotopy} is product annuli whose fundamental group fixes a chain of lozenges, as mentioned above.
\label{rem:HTwork}
\end{remark}
\section{The semiflow on the guts}
\label{sec:construction}
In this section, we construct the canonical semiflow on the homology guts.
\begin{theorem}
Let \(M\) be a closed 3-manifold and \(\varphi\) be a pseudo-Anosov flow without perfect fits on \(M\).
Let \(\mathcal{C}_{2}(\varphi)\) be the carried cone of \(\varphi\) in \(H_{2}(M)\), where the negative Euler class and the Thurston norm match.
For any primitive element \(z\) in \(\mathcal{C}_{2}(\varphi)\), there is a well-defined tight semiflow on the homology guts \(\operatorname{Guts}(M,z)\) associated to \(z\).
\label{thm:flowinguts}
\end{theorem}

Before the proof, we recall the definition of the carried face (cone) and the guts.
\subsection{Carried cone (face)}
\label{sec:orgdc866ba}
Let \(\varphi\) be a pseudo-Anosov flow without perfect fits. We can define its Euler class \(e_{\varphi}\) as the Euler class of the orthogonal plane bundle.
The subset of \(\mathcal{B}_{\mathrm{Th}}(M)\) where \(-e_{\varphi}\) and the Thurston norm \(x_{M}\) match is a (possibly empty) closed cone \(\mathcal{C}_{2}(\varphi)\) over a closed face \(\sigma_{\varphi}\) of \(\mathcal{B}_{\mathrm{Th}}(M)\).
The closed cone and face are called the \emph{carried cone}  and \emph{carried face} of \(\varphi\), respectively.

In \autocite{Mosher1992_DynamicalSystemsAndTheHomologyNormOfA3ManifoldIi_PUB}, Mosher shows that \(\varphi\) \emph{represents} \(\sigma_{\varphi}\) in the sense that the carried cone \(\mathcal{C}_{2}(\varphi)\) of \(\varphi\) is dual to its Fried cone (the cone of homology directions).
Moreover, he show that any integral class in \(\mathcal{C}_{2}(\varphi)\) is represented by a surface almost transverse to the flow.
Landry improves ``homologous'' to ``isotopic''. Any taut surface \(S\) with \([S]\in \mathcal{C}_{2}(\varphi)\) is isotopic to be almost transverse to the flow. Further, the combinatorial type of the associated blowup is unique \autocite{LandryMinskyTaylor2025_TransverseSurfacesAndPseudoAnosovFlows_PUB}.
\subsection{Guts}
\label{sec:org6a4a6bf}
\begin{definition}[{\cite[Definition 4.10]{Gabai1983_FoliationsAndTheTopologyOf3Manifolds_PUB},\cite[Definition 2.17]{AgolZhang2022_GutsInSuturedDecompositionsAndTheThurstonNorm_ARXv1}}]
Let \((N,\gamma)\) be a taut sutured manifold and let \(S\) be the maximal set of pairwise disjoint non-parallel product annuli and product disks in \(N\).
We define the \emph{sutured guts} \(\mathrm{Guts}(N)\) of \((N,\gamma)\) as the sutured manifold obtained from \((N,\gamma)\) by decomposing along \(S\) and discarding \emph{windows}, namely, product sutured manifold components.

The sutured guts is well-defined up to isotopy class of inclusion map by classical JSJ decomposition theory \autocite{Johannson1979_HomotopyEquivalencesOf3ManifoldsWithBoundaries_PUB} \autocite{JacoShalen1979_SeifertFiberedSpacesIn3Manifolds}.
\end{definition}
 % TODO;: fix all blown-up annuli and blown-up complex.

\begin{remark}
In fact, cutting out a maximal set of pairwise disjoint non-parallel product annuli is sufficient. Let \(D\) be a product disk in \((N,\gamma)\) and let \(B\) be the union of components of \(\gamma\) which intersects \(D\).
Then \(\partial N_{\epsilon}(B\cup D)\) is a product annulus, and decomposing along \(\partial N_{\epsilon}(B\cup D)\) is the same as decomposing along \(D\), after discarding all product components.

When decomposing along a product annulus \(S\), different choices of orientation on \(S\) give same sutured manifold decomposition results.
Sometimes, it is more convenient to regard \(S\) as an unoriented surface and \(S_{+}\cup S_{-}\) belongs to sutures of \(N-S\) directly.
\label{rem:productdisk}
\end{remark}

Homology guts is defined for a compact, oriented, connected, irreducible 3-manifold \(M\) with toral boundary and non-degenerate Thurston norm \(x_{M}\).
As usual, we treat \((M,\partial M)\) as a sutured manifold with \(\gamma=\partial M\).

\begin{definition}[Properly norm-minimizing]
Let \(M\) be a compact 3-manifold whose boundary is a collection of tori.
An oriented surface \(S\) is called \emph{properly norm-minimizing} if \(S\) is norm-minimizing and on each boundary \(P\) of \(M\), components of \(\partial S\cap P\) have the same orientation.

When \(M\) is closed, there is no difference between norm-minimizing and properly norm-minimizing.
\end{definition}
\begin{definition}[Homology guts]
Let \(z\) be a primitive element in \(H_{2}(M,\partial M;\Integer)\).
A \emph{facet surface} \(F(z)\) for \(z\) is the union of a maximal collection of disjoint, non-parallel, properly norm-minimizing surfaces \(S_{1},\ldots ,S_{k}\) in \(M\) such that
\[ [S_{i}]=z \text{ in } H_{2}(M,\partial M;\Integer),\quad \forall i=1,\ldots,k. \]

We define the \emph{homology guts} \(\mathrm{Guts}(M,z)\) of \(M\) associated to \(z\) as the sutured guts of \(M - F(z)\) where \(F(z)\) is an arbitrary facet surface for \(z\).
\end{definition}

Most interestingly, according to results of Agol and Zhang, homology guts are a topological invariant which is intimately related to homology.
\begin{theorem}[{\cite[Theorem 1.1]{AgolZhang2022_GutsInSuturedDecompositionsAndTheThurstonNorm_ARXv1}}]
Let \(M\) be an irreducible, orientable 3-manifold with toral boundaries and non-degenerate Thurston norm and let \(z\) be a primitive element in \(H_{2}(M,\partial M;\Integer)\).
Then, up to the equivalence of guts, the homology guts \(\mathrm{Guts}(M,z)\) does not depend on the selection of facet surfaces and product decomposition surfaces.
\label{thm:AZguts}
\end{theorem}
\begin{definition}[Equivalence of guts]
Two guts \(G_{1}\) and \(G_{2}\) are equivalent if there is an isotopy \(h_{t}\) of \(M\), such that \(h_{0}= \mathrm{id}_{M}\) and \(h_{1}\) maps \(G_{1}\) to \(G_{2}\) homeomorphically and preserving sutured structure.
\end{definition}

We conclude the introduction by a classification of guts components, according to the existence of non-product annulus.
\begin{lemma}
Assume that \(M\) is atoroidal and \(N\) is a component of \(M-F(z)\) where \(F(z)\) is a facet for \(z\).
Then \(\mathrm{Guts}(N)\) has only one connected component.
If there is a non-product annulus in \(N\), then \(\mathrm{Guts}(N)\) is a 4-ST or a non-longitudinal 2-ST.
If there is no non-product annulus in \(N\), then \(\mathrm{Guts}(N)\) is an acylindrical taut sutured manifold and admits a complete finite-volume hyperbolic metric.
\label{lem:gutsclassification}
\end{lemma}
\begin{proof}
\(\mathrm{Guts}(N)\) is connected by \autocite[Lemma 3.7]{AgolZhang2022_GutsInSuturedDecompositionsAndTheThurstonNorm_ARXv1}.
The rest  follows directly from \autocite[Remark 3.9]{AgolZhang2022_GutsInSuturedDecompositionsAndTheThurstonNorm_ARXv1} and \(M\) is atoroidal. We give a short proof here for completion.

If there is no non-product annulus in \(N\), then all JSJ-annuli are product annuli. So \(\mathrm{Guts}(N)\) is acylindrical and admits a complete finite-volume hyperbolic metric by Thurston's hyperbolization theorem \autocite{Morgan1984_ChapterVThurstonsUniformizationTheoremForThreeDimensionalManifolds_PUB}.

Otherwise, assume that \(N\) contains a non-product annulus \(A\).
Let \(S_{1},S_{2}\) be the component of \(\partial N\) which intersects \(A\). We surgery \(S_{1}\cup S_{2}\) along \(A\) to obtain a subsurface \(S\) in \(N\).
If all components of \(S\) are not null-homologous, then \(S\) is parallel into \(\partial N\) and hence the guts of \(N\) is a 4-ST.
If \(S\) contains a null-homologous component \(C\), then \(C\) must be a torus and the guts of \(N\) is the 3-manifold bounded by \(C\), with two sutures on \(C\).
Since \(N\) is atoroidal, \(C\) bounds a solid torus and thus \(\mathrm{Guts}(N)\) is a non-longitudinal 2-ST.
\end{proof}
\subsection{The semiflow on the homology guts}
\label{sec:org727e907}
In the following, we assume that \(M\) is a oriented, closed 3-manifold which admits a pseudo-Anosov flow \(\varphi\) with perfect fits and \(z\in H_{2}(M,\partial M;\Integer)\) is carried by \(\varphi\), namely, \(z\in \mathcal{C}_{2}(\varphi)\).
Our goal is to construct a natural semiflow on homology guts \(\mathrm{Guts}(M,z)\).

For later convenience, we introduce the transverse facet surface.
\begin{definition}[Transverse facet surface]
A facet surface \(F(z)\) for \(z\) is called \emph{transverse} if each component of \(F(z)\) is almost transverse to \(\varphi\).
We denote by \(F_{\varphi}(z)\) a transverse facet surface for \(z\).
\end{definition}

According to results of Landry, Minsky and Taylor, any facet surface \(F(z)\) is isotopic to a transverse one \(F_{\varphi}(z)\), provided \(z\in \mathcal{C}_{2}(\varphi)\).
\begin{theorem}[{\cite[Theorem A]{LandryMinskyTaylor2025_TransverseSurfacesAndPseudoAnosovFlows_PUB}}, {\cite[Theorem B]{Landry2022_VeeringTriangulationsAndTheThurstonNorm_PUB}}]
Let \(\varphi\) be a pseudo-Anosov flow with no perfect fits on \(M\). Suppose that \(S\) is a properly embedded, oriented surface in \(M\). Then the followings are equivalent:
\begin{enumerate}
\item \(S\) can be isotoped to be almost transverse to \(\varphi\).
\item \(S\) is taut and pairs nonnegatively with the closed orbits of \(\varphi\).
\item \(S\) is taut and \([S]\in \mathcal{C}_{2}(\varphi)\).
\end{enumerate}
\end{theorem}

\begin{proof}[Proof of Theorem \ref{thm:flowinguts}]
By \autocite[Theorem A]{Landry2022_VeeringTriangulationsAndTheThurstonNorm_PUB}, \(M\) is atoroidal.

Let \(z\) be a primitive element in \(\mathcal{C}_{2}(\varphi)\) and let \(F(z)\) be a facet surface. Then \(F(z)\) is isotopic to a transverse facet surface \(F_{\varphi}(z)\).
By definition, there exists a blowup \(\varphi^{\sharp}\) of \(\varphi\) which is transverse to \(F_{\varphi}(z)\) minimally. Therefore, we give a well-defined semiflow on the sutured manifold \(M-F_{\varphi}(z)\), by simply restricting \(\varphi^{\sharp}\) to \(M-F_{\varphi}(z)\).

Let \(\mathcal{A}\) be a maximal collection of non-parallel product annuli in \(M-F_{\varphi}(z)\).
According to Proposition \ref{thm:annulusisotopy}, we find an isotopy \(h_{t}\) of \(M-F_{\varphi}(z)\) and a further blowup \(\varphi^{\sharp\sharp}\), such that \(h_{t}\) moves \(\mathcal{A}\) to \(\varphi^{\sharp\sharp}\)-aligned position \(\mathcal{A}_{\varphi}\).

Decompose \(\varphi^{\sharp\sharp}\) along \(\mathcal{A}_{\varphi}\) as in Lemma \ref{lem:decompose} and throw away product sutured manifold components. We obtain a well-defined tight semiflow on the guts \(\mathrm{Guts}(M,z)\).
\end{proof}

We conclude this section with a useful and beautiful characterization of the product sutured manifold components in \(M-(F_{\varphi}(z)\cup \mathcal{A}_{\varphi})\).
\begin{lemma}
Let \((N,N_{+},N_{-},\gamma_{N})\) be a product sutured manifold component of \(M-(F_{\varphi}(z)\cup \mathcal{A}_{\varphi})\).
Then \(\varphi^{\sharp\sharp}|_{N}\) is a product flow.

In other words, a component \(N\) in \(M-(F_{\varphi}(z)\cup \mathcal{A}_{\varphi})\) is a window if and only if \(\varphi^{\sharp\sharp}|_{N}\) is product.
\label{lem:windowproduct}
\end{lemma}

We require the following powerful theorem proven by Landry, Minsky and Taylor.
Here we state a slight variance of the theorem.
\begin{theorem}[{\cite[Theorem D]{LandryMinskyTaylor2025_TransverseSurfacesAndPseudoAnosovFlows_PUB}}]
Let \(\varphi\) be a pseudo-Anosov flow on \(M\), and let \(S_{1}\) and \(S_{2}\) be two isotopic embedded subsurfaces in \(M\) which are almost transverse to \(\varphi\).
Let \(N\) be a \(I\)-bundle cobounded by \(S_{1}\) and \(S_{2}\). (Such \(N\) is unique if \(S_{i}\) is not a fiber of \(M\).)
If \(S_{1}\) and \(S_{2}\) are minimally transverse to blowup \(\varphi^{\sharp}_{1}\) and \(\varphi^{\sharp}_{2}\), respectively, then \(\varphi_{1}^{\sharp}\) and \(\varphi_{2}^{\sharp}\) has same combinatorial type.

Furthermore, if the surfaces are transverse to a single blowup \(\varphi^{\sharp}\), then \(N\) is a product sutured manifold \(S_{1}\times I\) endowed with product flow, and hence \(S_{1}\) and \(S_{2}\) are isotopic along flowlines of \(\varphi^{\sharp}\).
\label{thm:LMTparallel}
\end{theorem}

\begin{proof}[Proof of Lemma \ref{lem:windowproduct}]
Let \((N,N_{+},N_{-},\gamma_{N})\) be a product sutured manifold component of \(M-(F_{\varphi}(z)\cup \mathcal{A}_{\varphi})\).
Suppose that \(S_{+}\subset F_{\varphi}(z)\) be the component containing \(N_{+}\) and \(S_{-}\subset F_{\varphi}(z)\) be the component containing \(N_{-}\).
If \(\gamma_{N}=\emptyset\), then the conclusion follows directly from Theorem \ref{thm:LMTparallel} since \(S_{+}\) and \(S_{-}\) are isotopic and cobound \(N\).

Let
\[S_{0}=(S_{+}\setminus N_{+}) \cup (N_{-}\cup \gamma_{N}) \]
and then perturb \(S_{0}\) and push away from \(S_{+}\cup N_{-}\) slightly.
We lean each product saturated annulus component of \(\gamma_{N}\) slightly and smooth the corner around each \(\varphi^{\sharp\sharp}\)-transverse component (see Figure \ref{fig:windows}).
As a result, we obtain a surface \(S_{0}\) which is transverse to \(\varphi^{\sharp\sharp}\) and isotopic to \(S_{1}\).

We claim that no component of \(\gamma_{N}\) is a \(\varphi^{\sharp\sharp}\)-transverse annulus. Otherwise, assume that \(A\subset \gamma_{N}\) is transverse to \(\varphi^{\sharp\sharp}\) and \(g\) is a generator of \(\pi_{1}(A)\). Then \(g\) fixes a singular point \(\widetilde{p}\) and \(A\) is transverse to a further blown up segment in \(\varphi^{\sharp\sharp}\) (see Section \ref{sec:annulusisotopy}).
In particular, considering quadrants of \(\widetilde{p}\), \(\widetilde{S}_{0}\) traverses one quadrant of \(\widetilde{S}_{+}\) and one quadrant of \(\widetilde{S}_{-}\).
However, according to Theorem \ref{thm:LMTparallel}, \(\widetilde{S}_{0}\) and \(\widetilde{S}_{+}\) should cover the same two \(\widetilde{p}\)-quadrants in \(P_{\varphi^{\sharp\sharp}}\) because they are isotopic.

Thus, each component of \(\gamma_{N}\) is a product \(\varphi^{\sharp\sharp}\)-saturated annulus.
By Theorem \ref{thm:LMTparallel}, \(S_{0}\) and \(S_{+}\) bound a product sutured manifold \(N'\) endowed with the product flow.
Clearly, \(N\) is a component of \(N'\) after cutting along \(\gamma_{N}\cap N'\).
So \(\varphi^{\sharp\sharp}|_{N}\) is also a product flow since \(\gamma_{N}\) is product saturated.

\begin{figure}[htbp]
\centering
\includegraphics[width=10cm]{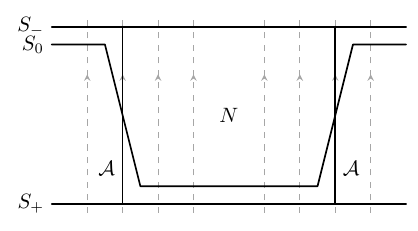}
\caption{\label{fig:windows}}
\end{figure}
\end{proof}
\section{Invariance of the semiflow in guts}
\label{sec:invariancefacet}
We precede this section with a summary of the construction of the tight semiflow on the homology guts \(\mathrm{Guts}(M,z)\).

Let \(M\) be an oriented connected closed 3-manifold and let \(\varphi\) be a pseudo-Anosov flow without perfect fits in \(M\). Denote by \(\mathcal{C}_{2}(\varphi)\) the carried cone of \(\varphi\).
Let \(z\in \mathcal{C}_{2}(\varphi)\) be a primitive element.
The construction of the tight semiflow in \(\mathrm{Guts}(M,z)\) is as follows:
\begin{construction}
\begin{enumerate}
\item Choose a facet \(F(z)\) for \(z\), namely, a maximal collection of non-parallel taut surfaces in \(M\), each of them representing \(z\).
\item Isotope \(F(z)\) to be a transverse facet \(F_{\varphi}(z)\).
\item Select a blowup \(\varphi^{\sharp}\) which is transverse to \(F_{\varphi}(z)\) minimally.
\item Choose a maximal collection \(\mathcal{A}\) of non-parallel product annuli in \(M-F_{\varphi}(z)\).
\item Ambiently isotope \(\mathcal{A}\) to be a collection \(\mathcal{A}_{\varphi}\) of product annuli and select a further blowup \(\varphi^{\sharp\sharp}\), such that each annulus in \(\mathcal{A}_{\varphi}\) is either product \(\varphi^{\sharp\sharp}\)-saturated or transverse to \(\varphi^{\sharp\sharp}\). Here \(\varphi^{\sharp\sharp}\) is minimally transverse to \(\mathcal{A}_{\varphi}\) and \(\mathcal{F}_{\varphi}(z)\).
\item Endow the sutured manifold \(M-(F_{\varphi}(z)\cup \mathcal{A}_{\varphi})\) with a semiflow by restricting \(\varphi^{\sharp\sharp}\).
\item Discard all product flow components in \(M-(F_{\varphi}(z)\cup \mathcal{A}_{\varphi})\). Finally, we get a well-defined tight semiflow on \(\mathrm{Guts}(M,z)\).
\end{enumerate}
\label{cst:flow}
\end{construction}

We say that \(F_{\varphi}(z)\) and \(\mathcal{A}_{\varphi}\) are in \emph{aligned position} with \(\varphi\).

In Theorem \ref{thm:AZguts}, Agol and Zhang prove that, up to equivalence of guts, \(\mathrm{Guts}(M,z)\) does not depend on the choice of the facet \(F(z)\) and the maximal collection \(\mathcal{A}\) of non-parallel product annuli.
In this section, we prove that up to orbit equivalence outside half annuli, the tight semiflow on \(\mathrm{Guts}(M,z)\) does not depend on choices in the steps above.
To put it precisely,
\begin{theorem}
Let \(M\) be an oriented connected closed 3-manifold and let \(\varphi\) be a pseudo-Anosov flow without perfect fits in \(M\). Suppose that \(z\in \mathcal{C}_{2}(\varphi)\) is primitive.

Then, up to orbit equivalence outside half annuli, the semiflow on \(\mathrm{Guts}(M,z)\) is independent of choices in Construction \ref{cst:flow}.
Precisely speaking, it doesn't depend on choices of the facet \(F(z)\), the product annuli collection \(\mathcal{A}\), the isotopy choice of \(F_{\varphi}(z)\) and \(\mathcal{A}_{\varphi}\), and the blowup ways of \(\varphi^{\sharp}\) and \(\varphi^{\sharp\sharp}\).
\label{thm:facetinvariance}
\end{theorem}
\subsection{Notation}
\label{sec:org467a29d}
We summarize the notation for the guts.
For a sutured manifold \(N\), let \(\mathrm{Guts}(N)\) denote its sutured guts.
Given a closed manifold \(M\) and a homology class \(z\in H_{2}(M)\), write \(\mathrm{Guts}(M,z)\) for the homology guts of \(M\) associated to \(z\).
When we need to specify a choice of facet and product annuli, we use \(\mathrm{Guts}(M,F(z)\cup \mathcal{A})\) for the sutured guts of \(M-F(z)\cup \mathcal{A}\).
In the absence of ambiguity, we abbreviate this to \(\mathrm{Guts}(F(z)\cup \mathcal{A})\).
The singular and regular guts, introduced in Section \ref{sec:singularguts}, are denoted \(\mathrm{Guts}_{s}\) and \(\mathrm{Guts}_{r}\) respectively.
Finally, we write \(\mathrm{Guts}_{\varphi}\) when we wish to emphasize that the guts carry the semiflow induced by \(\varphi\).
\subsection{Half annuli}
\label{sec:orgc6dc509}
\begin{definition}[Half annulus]
Let \(N\) be a sutured manifold and admit a compatible semiflow \(\phi\).
A \emph{half annulus} in \(N\) is an immersed, \(\phi\)-saturated annulus \(H\), so that
\begin{enumerate}
\item the interior of \(H\) is embedded,
\item one component  of \(\partial H\) covers a closed \(\phi\)-orbit \(\gamma\); the other component is a closed curve on \(\partial N\),
\item either each point in \(H\) converges to \(\gamma\) in forward time, or each point converges to \(\gamma\) in backward time.
\end{enumerate}
In other words, \(H\) is orbit equivalent to the half of a blown-up annulus (cutting along a transverse core circle).
We denote by \(\mathcal{H}(N)\) the union of half annuli in \(N\).

We say \((N_{1},\phi_{1})\) and \((N_{2},\phi_{2})\) are \emph{orbit equivalent outside half annuli} if there is a homeomorphism from \(N_{1}\setminus \mathcal{H}(N_{1})\) to \(N_{2}\setminus \mathcal{H}(N_{2})\), which maps \(\phi_{1}\)-orbits to \(\phi_{2}\)-orbits.
\label{def:halfannulus}
\end{definition}
We introduce ``outside half annuli'' to eliminate the flexibility of blowups.
In fact, any two blowups of \(\varphi\) are orbit equivalent outside their singular orbits and blown complexes.

If we view the guts \(N= \mathrm{Guts}(z)\) as a submanifold of \(M\) and in aligned position with \(\varphi\), then \(K\cap N\) are half annuli for any blown-up annulus \(K\). The converse is not always true.
We will see in Section \ref{sec:4ST}, if a component \(C\) of \(\mathrm{Guts}(z)\) is a 4-ST, then the semiflow on \(C\) contains four half annuli.
\subsection{Preparation}
\label{sec:org989ec85}
\subsubsection{Parallel}
\label{sec:org530853f}
Note that there is no affect if we add more parallel surfaces to a facet \(F(z)\).
In fact, the guts \(\mathrm{Guts}(F'(z))\) is the same as \(\mathrm{Guts}(F(z))\) where \(F'(z)\) is the added object. After isotoped to the aligned position, extra components in \(M-F_{\varphi}'(z)\) are all product flows by Lemma \ref{lem:windowproduct}.
So henceforth, we also refer to the added object as a \emph{facet}.

We state a simple lemma without proof.
\begin{lemma}
Let \((N,N_{+},N_{-},\gamma_{N})\) be a sutured manifold and let \(\psi\) be a tight semiflow on \(N\).
Let \(K\) be a connected component of \(N_{+}\), and endow \(K\times I\) with the product sutured structure and product flow.
Then attaching \(K\times \{0\}\) to \(N_{+}\) gives a semiflow on \(N\cup_{N_{+}} K\times I\)  which is orbit equivalent to \(N\).
Similar result holds for components of \(N_{-}\).

Let \(T\) be a product sutured manifold and homeomorphic to a solid torus.
Endow \(T\) with the product flow. Choose a glue map \(f:\mu\to \gamma\) which maps a suture \(\mu\) of \(T\) to a sutured annulus \(\gamma\) of \(N\) homeomorphically and preserving flow segments.
Then gluing \(T\) to \(N\) by \(f\) gives a semiflow on \(N\cup_{f}T\) which is orbit equivalent to \(N\).
\label{lem:addparallel}
\end{lemma}

\begin{corollary}
Assume that \(F(z)\) is a facet for \(z\).
Let \(F_{\varphi}^{1}(z)\) and \(F_{\varphi}^{2}(z)\) be two transverse facets associated to \(F(z)\).
Then the semiflows on \(M-F_{\varphi}^{1}(z)\) and \(M-F_{\varphi}^{2}(z)\) are orbit equivalent.
\label{lem:choice:transversefacet}
\end{corollary}
\begin{proof}
There is a correspondence between components of \(F_{\varphi}^{1}(z)\) and \(F_{\varphi}^{2}(z)\), such that, corresponding components are isotopic along flowlines.
So as sutured manifolds with semiflows, \(M-F_{\varphi}^{1}(z)\) is obtained from \(M-F_{\varphi}^{2}(z)\) by adding and removing some product flows.
Consequently, \(M-F_{\varphi}^{1}(z)\)  is orbit equivalent to \(M-F_{\varphi}^{2}(z)\) by Lemma \ref{lem:addparallel}.
\end{proof}
So we can freely isotope a transverse facet slightly, without change the semiflow.
\subsubsection{Veering triangulation}
\label{sec:veering}
We will need the language of veering triangulation.
A key advantage is that we can place the (almost) transverse surfaces in a regulated and rigid position, simplifying both proofs and visualizations.
We refer the reader to \autocite{Landry2022_VeeringTriangulationsAndTheThurstonNorm_PUB,LandryMinskyTaylor2025_TransverseSurfacesAndPseudoAnosovFlows_PUB} for a comprehensive introduction to veering triangulations.

In short, a veering triangulation is an ideal triangulation of a torally bounded 3-manifold \(W\) satisfying some co-orientation condition on faces and edges.

A \emph{veering tetrahedron} \(\Delta\) is an oriented ideal triangulation with extra data as follows:
\begin{enumerate}
\item Two faces are cooriented outward, called \emph{top faces}. Two faces are cooriented inward, called \emph{bottom faces}.
\item A dihedral angle between a top face and a bottom face is labelled \(0\). Other dihedral angles are labelled \(\pi\). The edges of \(\pi\)-angle are called the \emph{top and bottom edges} respectively.
\item The equatorial edges are assigned right or left veering in an alternating fashion as in Figure \ref{fig:veeringtetrahedron}. Some literatures use red and blue colors. We adopt both assignment styles.
\end{enumerate}

\begin{figure}[htbp]
\centering
\includegraphics[scale=1.1]{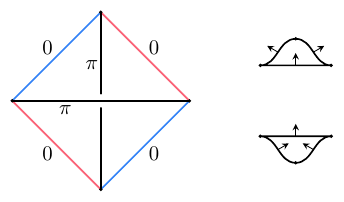}
\caption{\label{fig:veeringtetrahedron}}
\end{figure}

A \emph{veering triangulation} \(\tau\) on a compact oriented 3-manifold \(W\) is an ideal triangulation on \(\mathrm{int}(W)\), such that,
\begin{enumerate}
\item all ideal tetrahedra are veering tetrahedra,
\item the coorientations on faces are compatible the orientation of \(W\), and
\item the sum of dihedral angles around an edge is \(2\pi\). In other words, around an edge, there are one top \(\pi\)-edge, one bottom \(\pi\)-edge, and some \(0\)-edges.
\end{enumerate}

Usually, we view the veering structure as a truncated ideal triangulation on the compact manifold \(W\) and the truncated links of ideal vertices constitute a triangulation on \(\partial W\) (see Figure \ref{fig:veeringboundarypattern}).

\begin{figure}[htbp]
\centering
\includegraphics[scale=1.1]{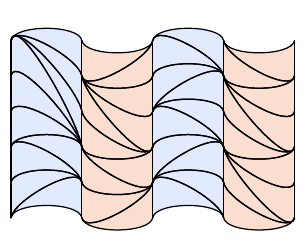}
\caption{\label{fig:veeringboundarypattern}An example of the induced triangulation on the boundary.}
\end{figure}

The 2-skeleton \(\tau^{(2)}\) is a branched surface in \(W\).
A \emph{branched surface} \(B\) with branch locus \(\mathrm{bl}(B)\) is a compact space locally modelled on left-hand side of Figure \ref{fig:branchedsurface}.
A \emph{standard neighborhood} of a branched surface \(B\) is a small tubular neighborhood \(N_{\epsilon}(B)\) foliated in a standard way by intervals.
The boundary \(\partial N_{\epsilon}(B)\) decomposes into two parts: the \emph{horizontal boundary} \(\partial _{h}N_{\epsilon}(B)\) and the \emph{vertical boundary} \(\partial _{v}N_{\epsilon}(B)\) (see right-hand side of Figure \ref{fig:branchedsurface}).
We refer to this \(I\)-bundle-like foliation as the \emph{vertical foliation} and say that \(B\) is \emph{cooriented} if this foliation is oriented.
If \(\Lambda\) is a 2-dimensional lamination, or simply an embedded surface in \(N_{\epsilon}(B)\), that is transverse to the vertical foliation, we say that \(B\) \emph{carries} \(\Lambda\). If additionally \(B\) is cooriented and \(\Lambda\) is oriented, we require that \(\Lambda\) is positively transverse to the vertical foliation.
We refer the reader to \autocite{FloydOertel1984_IncompressibleSurfacesViaBranchedSurfaces_PUB} and \autocite{Oertel1984_IncompressibleBranchedSurfaces_PUB} for more details on branched surfaces.

As in Figure \ref{fig:veeringboundarypattern}, \(\tau\) induces a triangulation on \(\partial T_{\tau}=\partial (M\setminus \mathrm{int}(T_{\tau}))\). Each triangle is a tip of a truncated veering tetrahedron and is a disk with two cusps inherited from veering structure.
A triangle is called \emph{upward} if it has two sides cooriented outward, and called \emph{downward} otherwise. 
\autocite{FuterGueritaud2013_ExplicitAngleStructuresForVeeringTriangulations_PUB} shows that the union of all upward triangles constitutes a collection of annuli, called \emph{upward ladders}. Similarly, the union of downward ones constitutes \emph{downward ladders}.
The upward and downward ladders appears alternatively and meet along a collection of curves, which is called \emph{ladderpole curves}. The slope of the ladderpole curves is called the \emph{ladderpole slope}.

\begin{figure}[htbp]
\centering
\includegraphics[scale=1.2]{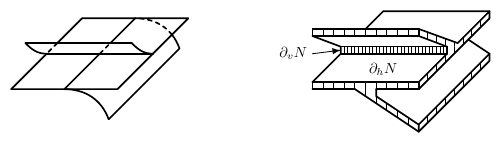}
\caption{\label{fig:branchedsurface}}
\end{figure}

For a closed, oriented 3-manifold \(M\), we say that \(M\) admits a veering triangulation if there exists a collection of disjoint simple closed curves \(\{\gamma_{i}\}_{i\in I}\) such that  \(M\setminus \cup_{i\in I} N_{\epsilon}(\gamma_{i})\) admits a veering triangulation.
In this case, we embed truncated \(\tau\) into \(M\) naturally (not properly).
We call the union of toral neighborhoods \(N_{\epsilon}(\gamma_{i})\) a \emph{tube system} \(T_{\tau}\) for \(M\) and \(\tau\) is a veering triangulation of \(M\) \emph{relative to} \(T_{\tau}\).
\begin{definition}
A oriented embedded surface \(S\subset M\) is \emph{relatively carried by} \(\tau\) if
\begin{enumerate}
\item \(S \setminus \mathrm{int}(T_{\tau})\) is carried by the branched surface \(\tau^{(2)}\) in \(M\setminus \mathrm{int}(T_{\tau})\), and
\item each component of \(S\cap T_{\tau}\) is either a meridional disk or an annulus.
\end{enumerate}
\end{definition}
In fact, by \autocite[Lemma 3.3]{Landry2018_TautBranchedSurfacesFromVeeringTriangulations_PUB}, boundaries of annular components in \(S\cap T_{\tau}\) are ladderpole curves.
A component of \(T_{\tau}\) is called \emph{ladderpole tube} if \(S\cap T_{\tau}\) are ladderpole curves, and called \emph{meridian tube} otherwise.

Let \(S\) be a surface in \(M\) relatively carried by \(\tau\).
We note that the position is not combinatorially unique in general. We mention one important case.
Suppose that \(A\) is an annulus component of \(S\cap T_{\tau}\) whose boundary components are in adjacent ladderpole curves (the boundary of a ladder).
In \autocite{LandryMinskyTaylor2025_TransverseSurfacesAndPseudoAnosovFlows_PUB}, such \(A\) is called a \emph{removable annulus}. This is because we may apply an \emph{annulus move} as defined in \autocite[Section 4.5]{Landry2022_VeeringTriangulationsAndTheThurstonNorm_PUB} to push \(A\) out of \(T_{\tau}\) and eliminate this component.

Let \(\varphi\) be a pseudo-Anosov flow without perfect fits on \(M\) and \(\left\{ \gamma_{i} \right\}_{i\in I}\) be the set of its singular orbits.
Then \(T_{\varphi}=\cup_{i\in I}N_{\epsilon}(\gamma_{i})\) is a tube system and \(M^{\circ}=M\setminus T_{\varphi}\) admits a veering triangulation \(\tau\).
We can assume that \(\tau^{(2)}\) is a smooth cooriented branched surface in \(M\), which is positively transverse to the flowlines of \(\varphi\), according to \autocite[Theorem 5.1]{LandryMinskyTaylor2022_FlowsGrowthRatesAndTheVeeringPolynomial_PUB}.
An embedded, oriented surface \(S\) is isotoped to be almost transverse to \(\varphi\) if and only if \(S\) is isotoped to be relatively carried by \(\tau\) \autocites[Theorem B]{Landry2022_VeeringTriangulationsAndTheThurstonNorm_PUB}[Theorem A]{LandryMinskyTaylor2025_TransverseSurfacesAndPseudoAnosovFlows_PUB}.
The combinatorial position is not unique, but different positions are isotopic along flowlines (rather than combinatorially), due to Theorem \ref{thm:LMTparallel}.

One can isotope any transverse facet to be relatively carried by \(\tau\), denote by \(F_{\tau}(z)\).
\subsection{Choices of aligned positions}
\label{sec:choice:alignedposition}
Fix a transverse facet \(F_{\varphi}(z)\) and a maximal collection \(\mathcal{A}\) of non-product annuli.
Fix a further blowup \(\varphi^{\sharp\sharp}\) associated to \(\mathcal{A}\).
We isotope \(\mathcal{A}\) to aligned position \(\mathcal{A}_{\varphi}\), and the semiflow on \(\mathrm{Guts}(F_{\varphi}(z)\cup \mathcal{A}_{\varphi})\) is obtained by restricting \(\varphi^{\sharp\sharp}\).
\begin{lemma}
The choice of aligned position \(\mathcal{A}_{\varphi}\) does not affect the semiflow on \(\mathrm{Guts}(F_{\varphi}(z)\cup \mathcal{A}_{\varphi})\).
In other words, if \(\mathcal{A}_{\varphi}^{1}\) and \(\mathcal{A}_{\varphi}^{2}\) are two \(\varphi^{\sharp\sharp}\)-aligned position of \(\mathcal{A}_{\varphi}\), then \(\mathrm{Guts}(F_{\varphi}(z)\cup \mathcal{A}_{\varphi}^{1})\) and \(\mathrm{Guts}(F_{\varphi}(z)\cup \mathcal{A}_{\varphi}^{2})\) are orbit equivalent.
\label{lem:choice:alignedposition}
\end{lemma}
\begin{proof}
We first show that different aligned positions of a single product annulus induce orbit equivalent flows.
Assume that \(A\) be a component of \(\mathcal{A}_{\varphi}\) and let \(A^{c}\) denote the union of the other components.
Let \(A'\) be a product annulus in \(M-F_{\tau}(z)\), ambiently isotopic to \(A\).
Assume that \(A'\) is in aligned position and disjoint from \(A^{c}\).
Then \(A\) and \(A'\) bound a product sutured torus \(T\) since \(M\) is atoroidal. Moreover, \(\varphi^{\sharp\sharp}|_{T}\) is a product flow by the proof of Lemma \ref{lem:windowproduct}.
The result follows directly from Lemma \ref{lem:addparallel}.
If \(A\) and \(A'\) are not disjoint, we throw away homotopically trivial components in the cut-and-paste \(A\oplus A'\).
It yields two annuli \(A_{0}\) and \(A_{0}'\), disjoint and ambiently isotopic to \(A\) and \(A'\). \(A_{0}\) remains in aligned position.
Then the semiflows induced by \(A\) and \(A'\) are both orbit equivalent to the one induced by \(A_{0}\), by Lemma \ref{lem:addparallel}.

In general, any component of \(\mathrm{Guts}(F_{\varphi}(z)\cup \mathcal{A}_{\varphi}^{1})\) differs from some component of \(\mathrm{Guts}(F_{\varphi}(z)\cup \mathcal{A}_{\varphi}^{2})\) by adding and removing product flows along the boundary. So they are orbit equivalent by Lemma \ref{lem:addparallel}.
\end{proof}
Now we can choose the aligned position freely.
\subsection{Singular graph and singular guts}
\label{sec:singularguts}
To prove the remaining independence, we introduce the regular and singular guts.
\begin{definition}[Singular graph]
Let \(F_{\tau}(z)\) be a transverse facet relatively carried by \(\tau\) and \(T\) be a ladderpole tube.
We define the \emph{singular graph} \(G_{T}(F_{\tau}(z))\) associated to \(T\) and \(F_{\tau}(z)\) as follows.
The vertices of \(G_{T}(F_{\tau}(z))\) is the set of ladderpole curves \(\mathrm{ld}_{T}\) in \(T\). Two vertices are connected by an edge if they are boundaries of an annular component of \(F_{\tau}(z)\cap T\).

Because the combinatorial positions are not unique, we add parallel components to \(F_{\tau}(z)\) such that the singular graph to every tube is maximal.
We say that the resulting \(F_{\tau}(z)\) is \emph{maximal} with respect to the veering triangulation \(\tau\).
\end{definition}
We always assume \(F_{\tau}(z)\) is maximal with respect to \(\tau\).
If an edge connects two adjoint vertices, then it corresponds to a removable annulus.

Differents facets may induce different singular graphs. But non-isolated vertices and ``outmost edges'' depend only on the homology class \(z\).
\begin{lemma}
Let \(F_{1}\) and \(F_{2}\) be two maximal transverse facets of \(z\) with respect to \(\tau\) and let \(T\) be a ladderpole tube of \(F_{1}\). Then \(T\) is also a ladderpole tube of \(F_{2}\).
If \(a\) is a non-isolated vertex of \(G_{T}(F_{1})\), then \(a\) is also non-isolated in \(G_{T}(F_{2})\)
\label{lem:singularinvariance1}
\end{lemma}
\begin{proof}
Let \(\gamma\) be the (singular) closed orbit contained in \(T\). If \(T\) is not a ladderpole tube for \(F_{2}\), then \(\gamma\) and \(z\) have positive intersection number.
Contradiction.

Let \(a\) be a non-isolated vertex of \(G_{T}(F_{1})\), and \(S_{1}\subset F_{1}\) be the component meeting \(a\). Assume that \(a\) is isolated in \(G_{T}(F_{2})\).
By \autocite[Lemma 3.2]{Scharlemann1989_SuturedManifoldsAndGeneralizedThurstonNorms_PUB}, the cut-and-paste \(S_{1}\oplus kF_{2}\) is isotoped slightly to be disjoint from \(F_{2}\) and preserves relatively carried by \(\tau\). \(S_{1}\oplus kF_{2}\) still meets \(a\).
One can add the parallel component \(S_{1}\oplus kF_{2}\) to \(F_{2}\), contradicting maximality of \(F_{2}\).
\end{proof}
We arrange the vertices clockwisely on a circle.
Each vertex has a well-defined coorientation induced by the pseudo-Anosov flow, clockwisely or counterclockwisely.
If two vertices are connected by an edge, they must have opposite coorientation.
We denote the non-isolated vertices of \(G_{T}(F_{\tau})\) by \(v_{1},v_{2},\ldots,v_{n},v_{n+1}=v_{1}\) clockwisely.
\begin{lemma}
Let \(T\) be a ladderpole tube.
\begin{enumerate}
\item If more than one components of \(F_{\tau}\cap T\) meet \(v_{i}\), then between two adjacent components, a product annulus connects them near \(v_{i}\).
\item If \(v_{i}\) and \(v_{i+1}\) have opposite coorientation, then there is an edge connecting \(v_{i}\) and \(v_{i+1}\).
\item If \(v_{i}\) and \(v_{i+1}\) have same coorientation, then no edge connects \(v_{i}\) and \(v_{i+1}\).
In this case, let \(S_{i}\) be the component which meets \(v_{i}\) and is closest to \(v_{i+1}\) and let \(S_{i+1}\) be the component which meets \(v_{i+1}\) and is closest to \(v_{i}\). Then, a product annulus in \(T\) connects \(S_{i}\) and \(S_{i+1}\).
\end{enumerate}
\label{lem:singularinvariance2}
\end{lemma}
\begin{proof}
\(\partial T-F_{\tau}\) is a union of annuli. If two adjacent components \(C_{1},C_{2}\) meet \(v_{i}\), then the component of \(\partial T-F_{\tau}\) between them is the required product annulus.

Assume that \(v_{i}\) and \(v_{i+1}\) have opposite coorientation and assume that no edge connects them.
The component of \(\partial T-F_{\tau}\) between \(v_{i}\) and \(v_{i+1}\) is a non-product annulus \(A\).
Denote by \(F_{\tau}'\) the surgery of \(F_{\tau}\) along \(A\).
By Lemma \autocite[Lemma 3.2]{Scharlemann1989_SuturedManifoldsAndGeneralizedThurstonNorms_PUB}, \(F_{\tau}'\oplus kF_{\tau}\) is isotoped slightly to be disjoint from \(F_{\tau}\) and preserves relatively carried by \(\tau\).
A component \(S\) of \(F_{\tau}'\oplus kF_{\tau}\) connects \(v_{i}\) and \(v_{i+1}\) and is carried by \(\tau\), so it is not homologically trivial by \autocite[Lemma 5.2]{Landry2022_VeeringTriangulationsAndTheThurstonNorm_PUB}.
One can add \(S\) to \(F_{\tau}\), contradicting the assumption that \(F_{\tau}\) is maximal and no edge connects \(v_{i}\) and \(v_{i+1}\).

For the third statement, the component of \(\partial T-F_{\tau}\) between \(v_{i},v_{i+1}\) is the required product annulus.
\end{proof}

Let \(\mathcal{A}_{T}\) denote the union of product annuli arising in case 1 and 3, and let \(\mathcal{S}_{T}\) be the set of ladderpole annuli appearing in case 2.
By Lemma \ref{lem:singularinvariance2}, the sutured submanifold \(P_{T}\) bounded by \(\mathcal{A}_{T}\) and \(\mathcal{S}_{T}\) is either a sutured torus or empty, and depends only on the homology class \(z\).
(For a meridian tube, both \(\mathcal{A}_{T}\) and \(\mathcal{S}_{T}\) are defined to be emptyset.
Let \(T_{\tau}\) be the tube system. Define
\[ \mathcal{A}_{\tau} = \bigcup_{T\subset T_{\tau} \text{ is a tube}} \mathcal{A}_{T},\]
\[ \mathcal{S}_{\tau} = \bigcup_{T\subset T_{\tau} \text{ is a tube}} \mathcal{S}_{T},\]
\[ P_{\tau} = \bigcup_{T\subset T_{\tau} \text{ is a tube}} P_{T}.\]
The annuli in case 1 are product saturated, denoted by \(\mathcal{A}_{\tau}^{\parallel}\), while those from case 3 are isotoped to be transverse to a further blowup, denoted by \(\mathcal{A}_{\tau}^{\pitchfork}\)

Let \(F_{\tau}(z)\) be a maximal tranvserse facet and let \(\mathcal{A}_{\varphi}\) be a maximal collection of product annuli in aligned position with \(\varphi^{\sharp\sharp}\).
Due to Section \ref{sec:annulusisotopy} and Section \ref{sec:choice:alignedposition}, we assume that transverse components of \(\mathcal{A}_{\varphi}\) are in \(P_{\tau}\), and hence disjoint from \(\mathcal{A}_{\tau}\).
Product saturated components \(\mathcal{A}_{\varphi}^{\parallel}\) of \(\mathcal{A}_{\varphi}\) are disjoint from \(\mathcal{A}_{\tau}^{\pitchfork}\) but may intersect \(\mathcal{A}_{\tau}^{\parallel}\).
We choose a small product saturated neighborhood \(N\) of \(\mathcal{A}_{\varphi}^{\parallel}\cup \mathcal{A}_{\tau}^{\parallel}\). Then \(\partial N\) is a collection of product saturated annuli satisfying that \(\partial N\) is disjoint from \(\mathcal{A}_{\tau}\) and \(\partial N\cup \mathcal{A}_{\tau}^{\parallel}\) is maximal.
The semiflow on \(\mathrm{Guts}(F_{\tau}(z)\cup \partial N\cup \mathcal{A}_{\tau})\) is orbit equivalent to the semiflow on \(\mathrm{Guts}(F_{\tau}(z)\cup \mathcal{A}_{\varphi})\).
Therefore, we always assume that \(\mathcal{A}_{\varphi}\) is disjoint from \(\mathcal{A}_{\tau}\).
Moreover, we can assume that product saturated components of \(\mathcal{A}_{\varphi}\) are in \(M\setminus P_{\tau}\), while transverse components are in \(P_{\tau}\).

Each component of \(\mathcal{A}_{\tau}\) is parallel to a component of \(\mathcal{A}_{\varphi}\) by maximality of \(\mathcal{A}_{\varphi}\).
Let \(\mathcal{A}_{\tau}'\) denote the parallel copy of \(\mathcal{A}_{\tau}\) in \(\mathcal{A}_{\varphi}\). It is clear that \(\mathcal{A}_{\tau}'\cup \mathcal{S}_{\tau}\) bounds a sutured manifold \(P_{\tau}'\).
Set \(X=M- (F_{\tau}(z)\cup \mathcal{A}_{\tau}')\).
Let the \emph{singular part} of \(X\) be the union of those components of \(X\) contained in \(P_{\tau}'\); let the \emph{regular part} be the union of remaining components.
After further cutting along \(\mathcal{A}_{\varphi}\) and discarding product components in the singular part (resp. the regular part), we obtain the \emph{singular guts} \(\mathrm{Guts}_{s}(F_{\tau}(z)\cup \mathcal{A}_{\varphi})\) (resp. the \emph{regular guts} \(\mathrm{Guts}_{r}(F_{\tau}(z)\cup \mathcal{A}_{\varphi})\)).
As a sutured manifold, the singular guts coincides with the (horizontal) guts of \(P_{\tau}\); consequently \(\mathrm{Guts}_{s}(F_{\tau}(z)\cup \mathcal{A}_{\varphi})\) depends only on \(z\) (again as a sutured manifold).

A local example is illustrated schematically in Figure \ref{fig:singulargraph}.
Let \(T\) be a ladderpole tube, whose boundary is represented by a purple circle.
A maximal facet \(F(z)\) intersects \(T\) with some ladderpole annuli, drawn as black arcs.
The black arrows indicate the coorientations of the taut surfaces.
The singular graph associated to \(F(z)\) and \(T\) is shown in Figure \ref{fig:singulargraph}(b); it has four isolated vertices and eight non-isolated ones.
The 
In Figure \ref{fig:singulargraph}(c), the parallel copy \(\mathcal{A}_{\tau}'\) of \(\mathcal{A}_{\tau}\) in \(\mathcal{A}_{\varphi}\) is drawn in red arcs.
Figure \ref{fig:singulargraph}(d) shows how \(\mathcal{A}_{\tau}'\) becomes transverse to a further blowup \(\varphi^{\sharp\sharp}\).
The sutured tori \(P_{\tau}'\), bounded by \(\mathcal{A}_{\tau}'\) and \(F(z)\), is shadowed by gray lines.
After removing \(F(z)\) and \(P_{\tau}'\), we obtain the regular part, shown in Figure \ref{fig:singulargraph}(e).
Inside \(T\), the regular part is of two possible types, \emph{blunt} or \emph{sharp}.
A sharp component contains a closed orbit, and two half annuli that separate it into two parts.
A blunt component contains a closed orbit and at most one half annulus.
In Section \ref{sec:facetchoice}, we will truncate all sharp components to obtain the truncated regular guts (See Figure \ref{fig:singulargraph}(f)).

\begin{figure}[htbp]
\centering
\includegraphics[scale=0.8]{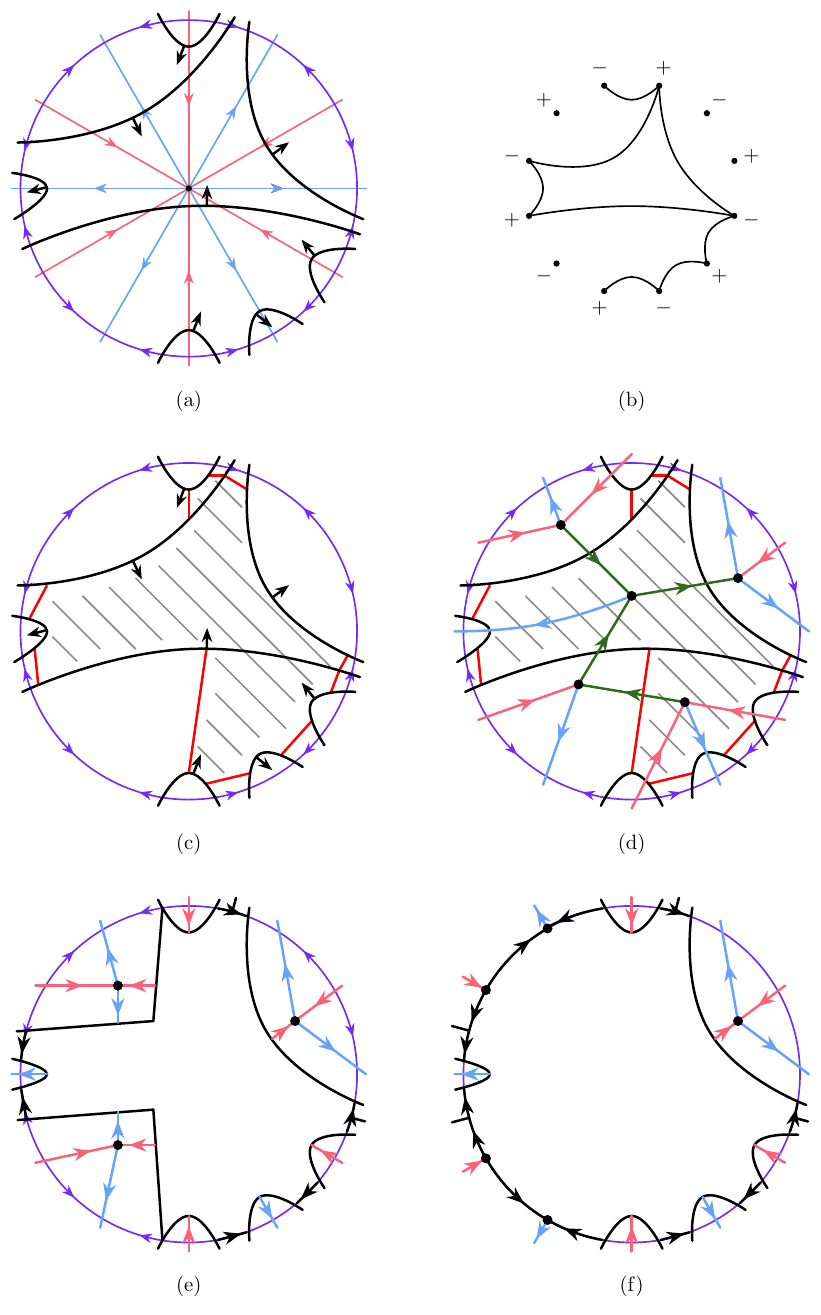}
\caption{\label{fig:singulargraph}}
\end{figure}
\subsection{4-ST and 2-ST}
\label{sec:4ST}
In preceding section, we separate the homology guts into two parts, the singular one and the regular one.
We will see the semiflow on the singular guts is simple and unique.
\begin{proposition}
Let \(M\) be an oriented connected closed 3-manifold and let \(\varphi\) be a pseudo-Anosov flow without perfect fits in \(M\). Suppose that \(z\in \mathcal{C}_{2}(\varphi)\) is primitive.

Then, up to orbit equivalence outside half annuli, the semiflow on the singular guts \(\mathrm{Guts}_{s}(M,z)\) is independent of choices mentioned in Theorem \ref{thm:facetinvariance}.
\label{prop:singulargutsinvariance}
\end{proposition}

\begin{proof}
As a sutured manifold, the singular guts \(\mathrm{Guts}_{s}(z)\) depends on \(z\) and is the (horizontal) guts of some sutured tori \(P_{\tau}\).
So, \(\mathrm{Guts}_{s}(z)\) consists of 4-STs and non-longitudinal 2-STs by Lemma \ref{lem:gutsclassification}.

Now we characterize the semiflow on 4-STs and non-longitudinal 2-STs.
Let \(C\) be a 4-ST component of \(\mathrm{Guts}_{s}(z)\), and let \(N\) be the component of \(M-F_{\tau}(z)\) that contains \(C\).
Then \(C\) contains a non-product annulus \(A\).
As in Lemma \ref{lem:singular-invariantline}, we find a \(\varphi^{\sharp\sharp}\)-saturated annulus \(A'\) that is isotopic to \(A\) and contains a closed orbit \(\gamma\).
\(A'\) is still disjoint from \(F_{\tau}(z)\) and hence lies in \(N\). Since the flow on the windows is product by Lemma \ref{lem:windowproduct}, \(A'\) is contained in \(C\).
Assume that \(A'\) connects two components \(S_{1},S_{2}\) of \(F_{\tau}(z)\). We surgery \(S_{1}\cup S_{2}\) along \(A'\) and obtain a \(\varphi^{\sharp\sharp}\)-transverse surface \(S_{0}\).
\(S_{0}\) is parallel to a component \(F_{\tau}(z)\) (along flowliens), so we assume that \(S_{0}\) is a component of \(\partial N\).
There are four product \(\varphi^{\sharp\sharp}\)-saturated annuli \(\{B_{i}\}_{1\leq i\leq 4}\) between \(S_{0}\) and \(S_{1}\cup S_{2}\).
Decompose along \(\cup _{i}B_{i}\) and throwing away windows, we obtain a 4-ST which contains \(\gamma\).
This 4-ST is \(C\) up to adding or removing some product flows on the boundary like Lemma \ref{lem:addparallel}, because \(\mathrm{Guts}(N)=C\) is connected.
In particular, \(C\) contains exactly one closed orbit and four half annuli \(\mathcal{H}(C)\), and \(C\setminus \mathcal{H}(C)\) consists of four copies of \([0,1)\times S^{1}\times [0,1]\) with the product flow.
So the semiflow in the 4-ST component is unique up to orbit equivalence outside half annuli, regardless of the choices in Theorem \ref{thm:facetinvariance}.

Let \(C\) be a non-longitudinal 2-ST component of \(\mathrm{Guts}_{s}(z)\).
The annulus \(A'\) in preceding paragraph may not exist.
Nevertheless, we can still construct a \(\varphi^{\sharp\sharp}\)-saturated subset \(S'\), which is the mapping torus of a star-shape tree and contains a closed orbit \(\gamma\).
\(S'\) must be disjoint from \(F_{\varphi}(z)\), so \(S'\subset \mathrm{Guts}(N)\).
As before, we surgery along \(S'\) to obtain a \(\varphi^{\sharp\sharp}\)-transverse surface \(S_{0}\) and there are two product saturated annuli between \(S_{0}\) and \(F_{\varphi}(z)\) near \(\gamma\).
Decomposing along these two annuli and throwing away windows part yields a non-longitudinal 2-ST which contains \(\gamma\).
This 2-ST is \(C\) up to adding and removing product flows on the boundary.
In particular, \(C\) contains one closed orbit and two half annuli \(\mathcal{H}(C)\).
So the semiflow in the non-longitudinal 2-ST component is unique up to orbit equivalence outside half annuli, regardless of the choices in Theorem \ref{thm:facetinvariance}.

In summary, the semiflow on \(\mathrm{Guts}_{s}(z)\) depends only on \(z\), up to orbit equivalence outside half annuli.
\end{proof}

If, in addition, the blowup satisfies Condition \hyperref[LH]{LH}, semiflows on the singular guts are orbit equivalent (without outside the half annuli).
Unfortunately, the similar result doesn't hold for the regular guts.

\textbf{Local Hyperbolic Condition (LH):} each closed orbit in the blown-up complex is a pseudo-hyperbolic orbit. \label{LH}

\begin{proposition}
Suppose that the choice of \(\varphi^{\sharp\sharp}\) in Construction \ref{cst:flow} satisfies Condition \hyperref[LH]{LH}.

Then, up to orbit equivalence, the semiflow on the singular guts \(\mathrm{Guts}_{s}(z)\) is independent of other choices in Construction \ref{cst:flow}.
\label{prop:extra:singulargutsinvariance}
\end{proposition}

The blowup construction in \autocite[Section 3.1]{LandryMinskyTaylor2025_TransverseSurfacesAndPseudoAnosovFlows_PUB} satisfies Condition \hyperref[LH]{LH} automatically.
In the simplest blowup of a quadrant, the local return map is topologically conjugate to
\[ \begin{pmatrix}
1/2&0\\0&2
\end{pmatrix} \]
a linear map. Then we replace the fixed point by the interval of unit tangent vectors pointing into the quadrant.
Both fixed points created by blowing up are pseudo-hyperbolic orbit, with one coordinate given by the slope.
Repeating for new fixed points gives a blowup with more fixed points.
So all closed orbits are pseudo-hyperbolic.

It's interesting to know whether Proposition \ref{prop:extra:singulargutsinvariance} holds without Condition \hyperref[LH]{LH}.

\begin{lemma}
Suppose that \((T,R_{+},R_{-},\gamma)\) is a 4-ST. Then up to orbit equivalence, there is only one semiflow \(\varphi\) satisfying the following conditions:
\begin{enumerate}
\item \(\varphi\) is tight.
\item \(\varphi\) contains exactly one bi-infinite flowline \(\lambda\) and \(\lambda\) is a pseudo-hyperbolic closed orbit.
\end{enumerate}
\label{lem:extra:4-ST}
\end{lemma}
\begin{proof}
By definition, let \(N\) be the closed neighborhood of \(\lambda\) which is orbit equivalent to a standard neighborhood of a standard pseudo-hyperbolic orbit.
Let \(W^{s}(\lambda)\) be the set of points in \(T\) whose forward rays converges to \(\lambda\) and let \(W^{u}(\lambda)\) be the set of points in \(T\) whose backward rays converges to \(\lambda\).
For each point \(x\in W^{s}(\lambda)\setminus \lambda\), the backward direction of \(x\) must intersect \(R_{+}\) because \(\lambda\) is the only bi-infinite flowline.
Similarly, the forward direction of \(x\in W^{u}(\lambda)\setminus \lambda\) intersects \(R_{-}\).
Due to continuity of \(\varphi\) and the compactness of \(N\), we choose a smaller \(N\), so that:
\begin{enumerate}
\item for each point \(x\in N\setminus (W^{s}\cup W^{u})\), the forward direction of \(x\) intersects \(R_{-}\) and the backward direction of \(x\) intersects \(R_{+}\), and
\item \(N\cap (W^{u}(\lambda)\cup W^{s}(\lambda))\) has only one connected component.
\end{enumerate}

All flowlines in \(T\setminus (W^{s}(\lambda)\cup W^{u}(\lambda))\) are compact segments. Otherwise, assume that a point \(x\notin T\setminus (W^{s}\cup W^{u})\) has a well-defined forward ray. There exists a sequence of time \(t_{i}\to \infty\) such that \(\varphi_{t_{i}}(x)\) converges to a point \(x_{0}\) and the flowline traversing \(x_{0}\) is bi-infinite.
So \(x_{0}\in \lambda\) and \(\varphi_{t_{i}}(x)\in N\) for sufficiently large \(t_{i}\).
However, each point in \(N\) is contained in either \(W^{s}(\lambda)\cup W^{u}(\lambda)\) or a compact flowline. Hence, the claim holds.

Consequently, \(\lambda\) is a regular pseudo-hyperbolic closed orbit with four prongs and \(W^{s}(\lambda)\cup W^{u}(\lambda)\) separates \(N\) and \(T\) into four quadrants.
For any subset \(U\subset  T\), let
\[ \varphi_{\Real}(U):=\{x\in T: \exists t\in \Real, \varphi(x)\in U \}. \]
Then \(\varphi_{\Real}(N)\) is a solid torus and \(\partial \varphi_{\Real}(N)\cap R_{+}\) is an annular neighborhood of \(\partial W^{s}(\lambda)\)

Suppose that \(\varphi_{1}\) and \(\varphi_{2}\) are two semiflows in \(T\) satisfying the conditions, with associated objects \(N_{i},W^{s}(\lambda_{i}),W^{u}(\lambda_{i})\) for \(i=1,2\) as defined in previous paragraphs.
We choose an orbit equivalence \(f:N_{1}\to N_{2}\) since both of them are neighborhoods of a standard regular pseudo-hyperbolic orbit.
We will extend \(f\) to \(W^{s}(\lambda_{1})\), \(W^{u}(\lambda_{1})\), and to \(\varphi_{\Real}(N_{1})\), and to \(\mathrm{cl}(T\setminus \varphi_{\Real}(N_{1}))\).

The metric completion of \(W^{s}(\lambda_{i})\setminus \lambda_{i}\) is two closed annuli.
Let \(A_{1}\) be one component of metric completion of \(W^{s}(\lambda_{1})\setminus \lambda_{1}\) and let \(A_{2}\) be the component of metric completion of \(W^{s}(\lambda_{2})\setminus \lambda_{2}\) which intersects \(\mathrm{int}(f(N_{1}\cap A_{1}))\).
For \(i=1,2\), choose any separating arc \(\alpha_{i}\) in \(A_{i}\) transverse to \(\varphi_{i}|_{A_{i}}\), so that \(\alpha_{2}\) contains \(f(\alpha_{1}\cap N)\) as a subinterval.
Then \(\varphi_{i}\) induces a \emph{contracting} return map \(r_{i}\) on \(\alpha_{i}\), namely, \(r_{i}^{k}(\alpha_{i})\) is a strictly decreasing collection of closed subintervals in \(\alpha_{i}\) and \(\cap _{k=1}^{\infty}r_{i}^{k}(\alpha_{i})\) is a point \(\lambda_{i}\cap \alpha_{i}\).
Extend \(f\) to \(f:A_{1}\to A_{2}\) as follows. First, extend the map to \(\alpha_{1}\) via the relation \(f(r_{1}^{-1}(x))=r_{2}^{-1}(f(x))\) iteratively; then, to the whole \(A_{1}\) by choosing an arbitrary homeomorphism on \(A_{1}\setminus (\alpha_{1}\cup N_{1})\) that preserves the flow structure.
Now, the domain of \(f\) becomes \(N_{1}\cup W^{s}(\lambda_{1}) \cup W^{u}(\lambda_{1})\).

Write \(\widetilde{N}_{i}=N_{i}\cup W^{s}(\lambda_{i})\cup W^{u}(\lambda_{i})\).
For \(i=1,2\), the metric completion of \(\varphi_{\Real}(N_{i})\setminus \widetilde{N}_{i}\) has eight components \(\{C^{(k)}_{i}\}\). Each component \(C_{i}^{(k)}\) is a solid torus and inherits from \(\varphi_{i}\) a product flow. 
Further, each \(C_{1}^{(k)}\) corresponds to a unique component \(C_{2}^{(k)}\) which satisfies
\[ f(C_{1}^{(k)}\cap \widetilde{N}_{1})=C_{2}^{(k)}\cap \widetilde{N}_{2}.  \]
On each component \(C_{1}^{(k)}\), \(f\) is well-defined on \(C_{1}^{(k)}\cap \widetilde{N}_{1}\) and preserves the flowlines on \(C_{1}^{(k)}\cap \widetilde{N}_{1}\). Since \(\varphi_{1}|_{C_{1}^{(k)}}\) and \(\varphi_{2}|_{C_{2}^{(k)}}\) are product flows, \(f\) can be extended to \(C_{1}^{(k)}\) preserving flowlines.
More precisely, up to homeomorphism, we write \(C_{1}^{(k)}=C_{2}^{(k)}= S^{1}\times I\times I\) and require that \(S^{1}\times \{0\}\times I\) is \(C_{i}^{(k)}\cap (W^{s}\cup W^{u})\), \(S^{1}\times I\times \{0\}\) is \(C_{i}^{(k)}\cap N_{i}\), and \(x\times \{0\}\times I,x\in S^{1}\) are flowlines,.
More precisely, up to homeomorphism, we write \(C_{1}^{(k)}=C_{2}^{(k)}= S^{1}\times I\times I\). Further, we require that, \(S^{1}\times \{0\}\times I\) corresponds to \(C_{i}^{(k)}\cap (W^{s}\cup W^{u})\) and \(S^{1}\times I\times \{0\}\) corresponds to \(C_{i}^{(k)}\cap N_{i}\); meanwhile, \(x\times \{i\}\times I,x\in S^{1},i\in I\) are flowlines and \(f\) is identity on \(S^{1}\times \{0\}\times I\) and \(S^{1}\times I\times \{0\}\).
We extend \(f\) directly by
\[ f(x,i,j)=(x,i,j) \]
which preserves flowlines clearly.
Note that the extension of \(f\) is not unique.

The complement of \(\varphi_{\Real}(N_{i})\) in \(T\) is four solid tori, each of which inherits a product flow from \(N_{i}\).
As in preceding paragraph, we extend \(f\) to \(T\setminus \varphi_{\Real}(N_{i})\) and hence the whole \(T\).
Consequently, \(\varphi_{1}\) and \(\varphi_{2}\) are orbit equivalent.
\end{proof}
\begin{remark}
The orbit equivalence constructed above may not preserve a pre-assigned (un)stable foliation structure. This is due to the existence of compact flowlines: the (un)stable manifold associated to a compact flowline is not canonically defined.
\end{remark}

Via the similar proof, the semiflow on a non-longitudinal 2-ST is also unique, provided some necessary assumptions and up to orbit equivalence.
\begin{lemma}
Suppose that \((T,R_{+},R_{-},\gamma)\) is a non-longitudinal 2-ST. Then up to orbit equivalence, there is only one semiflow \(\varphi\) satisfying the following conditions.
\begin{enumerate}
\item \(\varphi\) is tight.
\item \(\varphi\) contains exactly one bi-infinite flowline \(\lambda\) and \(\lambda\) is a pseudo-hyperbolic closed orbit.
\end{enumerate}
\label{lem:extra:2-ST}
\end{lemma}

Proposition \ref{prop:extra:singulargutsinvariance} follows directly from Lemma \ref{lem:extra:4-ST}, Lemma \ref{lem:extra:2-ST} and the characterization of the semiflows on sutured tori in Proposition \ref{prop:singulargutsinvariance}.
\subsection{Choices of blowup way}
\label{sec:org43f9ed4}
Recall that we choose a further blowup \(\varphi^{\sharp\sharp}\), such that \(F(z)\) and \(\mathcal{A}\) can be isotoped to be in aligned position.
As mentioned in Section \ref{sec:pre:blowup} and Lemma \ref{lem:singular-invariantline}, the choice of \(\varphi^{\sharp\sharp}\) is flexible.
But we have,
\begin{lemma}
Fix a maximal transverse facet \(F_{\tau}(z)\) and a maximal collection \(\mathcal{A}\) of product annuli.
Suppose that \(\varphi_{1}^{\sharp\sharp}\) and \(\varphi_{2}^{\sharp\sharp}\) are two further blowup such that \(\mathcal{A}\) can be isotoped to aligned position.
For \(i=1,2\), \(\varphi_{i}^{\sharp\sharp}\) induces a semiflow \(\phi_{i}\) on \(\mathrm{Guts}(F_{\tau}(z)\cup \mathcal{A})\).
Then, \(\phi_{1}\) and \(\phi_{2}\) are orbit equivalent outside half annuli.
\label{lem:choice:blowup}
\end{lemma}
\begin{proof}
By Proposition \ref{prop:singulargutsinvariance}, restricting to singular guts, \(\phi_{1}\) and \(\phi_{2}\) are orbit equivalent outside half annuli.
We handle the regular guts in the following.

Recall that \(T_{\tau}\) is the tube system of \(\varphi\) and contains all singular orbits of \(\varphi\).
Isotope \(\mathcal{A}\) to \(\varphi_{i}^{\sharp\sharp}\)-aligned position \(\mathcal{A}_{i}\).
Without loss of generosity, assume that
\begin{enumerate}
\item the blowdown map \(\pi_{i}: (M,\varphi_{i}^{\sharp\sharp})\to (M,\varphi)\) is identity on \(M\setminus T_{\tau}\).
\item product saturated components are in \(M\setminus T_{\tau}\) and transverse components are in \(T_{\tau}\) (see Section \ref{sec:singularguts}).
\item \(\mathcal{A}_{1}\) and \(\mathcal{A}_{2}\) use same product saturated components, and all transverse components are in the singular guts.
\end{enumerate}

Let \(X_{i}\) be the blown-up complex of \(\varphi_{i}^{\sharp\sharp}\). Set \(\Pi=\pi_{2}^{-1}\pi_{1}\).
\(\Pi\) maps \(M\setminus X_{1}\) to \(M\setminus X_{2}\) homeomorphically, and conjugate \(\varphi_{1}^{\sharp\sharp}\) to \(\varphi_{2}^{\sharp\sharp}\) outside the blown-up complex.
By our assumption, \(\Pi\) is identity on \(M\setminus T_{\tau}\).

Write \(\mathrm{RG}_{i}= \mathrm{Guts}_{r}(F_{\tau}(z)\cup \mathcal{A}_{i})\) endowed with the semiflow \(\phi_{i}\).
For a component \(C\) of the guts, let \(C^{\circ}\) denote \(C\) removing all half annuli.
Let \(C_{1}\) be a component of \(\mathrm{RG}_{1}\).
If \(C_{1}\) is disjoint from \(T_{\tau}\), then \(C_{1}\) is identical with a component \(C_{2}\) of \(\mathrm{RG}_{2}\), since \(\mathcal{A}_{1}\) and \(\mathcal{A}_{2}\) use same product saturated components.
Otherwise, assume \(C_{1}\cap T_{\tau}\) is nonempty. Components of \(C_{1}\cap T_{\tau}\) are of two possible types, \emph{blunt or sharp} (see Figure \ref{fig:singulargraph}(e)).
A sharp component contains a closed orbit, and two half annuli that separate it into two parts
A blunt component contains a closed orbit and at most one half annulus.
There is a unique component \(C_{2}\) of \(\mathrm{RG}_{2}\) so that \(C_{2}\cap (M\setminus T_{\tau})=C_{1}\cap (M\setminus T_{\tau})\).
Components of the symmetric difference \(C_{2}^{\circ}\mathbin\triangle \Pi(C_{1}^{\circ})\) are either a product flow, or a degenerate one.
A degenerate product flow is obtain from \(D\times I\) by collapsing \(\partial D\times I\), where \(D\) is an arbitrary set.
Therefore, \(C_{2}^{\circ}\) is orbit equivalent to \(\Pi(C_{1}^{\circ})\) and hence \(C_{1}^{\circ}\), similar to Lemma \ref{lem:addparallel}.
\end{proof}
\subsection{Choices of product annuli}
\label{sec:orgb6cd628}
Suppose \(F_{\varphi}(z)\) is a transverse facet and transverse to a blowup \(\varphi^{\sharp}\) minimally.
Let \(\mathcal{A}\) be a maximal collection of non-parallel product annuli.
Fix a further blowup \(\varphi^{\sharp\sharp}\). We isotope \(\mathcal{A}\) to aligned position \(\mathcal{A}_{\varphi}\).
In Lemma \ref{lem:choice:alignedposition}, we show that different aligned position choices give same semiflow on \(\mathrm{Guts}(F_{\varphi}(z)\cup \mathcal{A})\).
\begin{lemma}
Fix a maximal transverse facet \(F_{\tau}(z)\).
If \(\mathcal{A}^{1}\) and \(\mathcal{A}^{2}\) are two maximal collections of product annuli in \(M-F_{\tau}(z)\).
Then the semiflows on \(\mathrm{Guts}(F_{\tau}(z)\cup \mathcal{A}^{1})\) and \(\mathrm{Guts}(F_{\tau}(z)\cup \mathcal{A}^{2})\) are orbit equivalent outside half annuli.
\label{lem:choice:productannuli}
\end{lemma}
\begin{proof}
It suffices to show the semiflows on the regular guts are orbit equivalent outside half annuli.
Let \(R\) be the regular part of \(M-F_{\tau}(z)\).
As before, different blowups are assumed to be identical in \(M\setminus T_{\tau}\).
Up to preposing the orbit equivalence \(\Pi\) in Lemma \ref{lem:choice:blowup}, the semiflow on \(R\) outside half annuli is unique. We fix a semiflow \(\phi\) on \(R\) induced by some blowup \(\varphi^{\sharp\sharp}\).
Choose the aligned position \(\mathcal{A}_{\varphi}^{1}\) and \(\mathcal{A}_{\varphi}^{2}\) so that, components in \(M\setminus T_{\tau}\) are all product saturated.

Recall that the regular guts \(\mathrm{Guts}_{r}(F_{\tau}(z)\cup \mathcal{A}^{i}_{\varphi})\) is obtained by cutting \(R\) along \(\mathcal{A}^{i}_{\varphi}\) and removing all product flows.
Define \(\mathrm{Guts}_{r}(F_{\tau}(z)\cup \mathcal{A}^{1}_{\varphi}\cup \mathcal{A}^{2}_{\varphi})\) by cutting \(R\) along \(\mathcal{A}_{1}\cup \mathcal{A}_{2}\) and removing all product flows.
We will show both \(\mathrm{Guts}_{r}(F_{\tau}(z)\cup \mathcal{A}_{i})\) are orbit equivalent to \(\mathrm{Guts}_{r}(F_{\tau}(z)\cup \mathcal{A}_{1}\cup \mathcal{A}_{2})\).

Let \(N\) be a component of the regular part.
By \autocite[Lemma 3.7]{AgolZhang2022_GutsInSuturedDecompositionsAndTheThurstonNorm_ARXv1}, \(\mathrm{Guts}(N)\) has only one connected component since \(N\) is horizontally prime.
Therefore, \(N\) is the result of gluing \(\mathrm{Guts}(N)\) and some windows along sutures.
Any product annulus in \(N\) is isotopic into the windows part \(W\) by maximality. 
Let \(A\) be a product annulus in \(\mathcal{A}_{\varphi}^{2}\) and \(D\) is a component of \(A-\mathcal{A}_{\varphi}^{1}\).
If \(D\) is in the windows part of \(N-\mathcal{A}_{\varphi}^{1}\), cutting along \(D\) won't affect the semiflow on guts since \(D\) is product saturated.
Assume that \(D\) is in the guts component.
If \(D\) is a product annulus, then \(D\) is parallel to \(\mathcal{A}_{\varphi}^{1}\).
If \(D\) is a product disk, we perform surgery along \(D\) as in Remark \ref{rem:productdisk} and discard all homotopically trivial components.
This yields product annuli \(D'\) in the guts component and \(\mathrm{Guts}(N,\mathcal{A}_{\varphi}^{1}\cup D)\) is orbit equivalent to \(\mathrm{Guts}(N,\mathcal{A}_{\varphi}^{1}\cup D')\).
In both cases, cutting along \(D\) doesn't affect the semiflow on guts.
Therefore, both \(\mathrm{Guts}_{r}(F_{\tau}(z)\cup \mathcal{A}_{i})\) are orbit equivalent to \(\mathrm{Guts}_{r}(F_{\tau}(z)\cup \mathcal{A}_{1}\cup \mathcal{A}_{2})\) outside half annuli.
\end{proof}
\subsection{Choices of facet}
\label{sec:facetchoice}
\begin{lemma}
Let \(z\in \mathcal{C}_{2}(\varphi)\) be a primitive homology class.
Then different choices of the facet \(F(z)\) for \(z\) does not alter the semiflow on \(\mathrm{Guts}(M,z)\) up to orbit equivalence outside half annuli.
\label{lem:choice:facet}
\end{lemma}
It suffices to show that different facet choices induce the same semiflow on regular guts.

Let \(F_{1}(z)\) and \(F_{2}(z)\) be two maximal transverse facets in minimal position.
We fix a maximal collection \(\mathcal{A}_{i}\) of product annuli in \(M-F_{i}(z)\) for \(i=1,2\).
Assume that all transverse components are in \(T_{\tau}\), and all product saturated components are in \(M-T_{\tau}\).
Write \(\mathrm{RG}_{i}= \mathrm{Guts}_{r}(F_{i}(z)\cup \mathcal{A}_{i})\) endowed with the semiflow \(\phi_{i}\) induced by restricting \(\varphi_{i}^{\sharp\sharp}\).
As in the proof of Lemma \ref{lem:choice:blowup}, components of \(\mathrm{RG}_{i}\cap T_{\tau}\) are of two possible types, blunt or sharp.
Moreover, we can modify \(\Pi=\pi_{2}^{-1}\pi_{1}\) to be an orbit equivalent map outside half annuli between \(\mathrm{RG}_{1}\cap T_{\tau}\) and \(\mathrm{RG}_{2}\cap T_{\tau}\), where \(\pi_{i}: \varphi_{i}^{\sharp\sharp}\to \varphi_{i}\) is the blowdown map.
So without loss of generosity, we assume that the semiflows on \(\mathrm{RG}_{1}\cap T_{\tau}\) and \(\mathrm{RG}_{2}\cap T_{\tau}\) are the same.

\begin{figure}[htbp]
\centering
\includegraphics[scale=1.4]{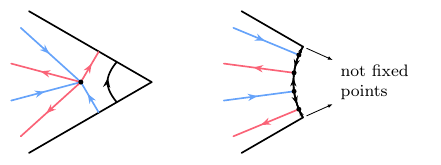}
\caption{\label{fig:cuttail}Truncate the tail of a sharp component.}
\end{figure}

We introduce the \emph{truncated regular guts}.
In a nutshell, we truncate each sharp component, see Figure \ref{fig:cuttail}.
As a sutured manifold, the truncated regular guts is homeomorphic to the regular guts.

Fix a Fried's blowup \(\varphi^{\circ}\) on \(M^{\circ}=M-T_{\tau}\subset M\).
We regard it as a ``flow'' on \(M\) but only defined on \(M^{\circ}\).
Then \(F_{1}(z)\) and \(F_{2}(z)\) are transverse to \(\varphi^{\circ}\).
Since \(\varphi_{i}^{\sharp\sharp}\) and \(\varphi^{\circ}\) are identical outside a small neighborhood \(U\) of \(\mathrm{cl}(T_{\tau})\), we consider the ``mix'' \(\varphi^{\mathrm{m}}\) of them.
Precisely speaking, \(\varphi^{\mathrm{m}}\) is the same to \(\varphi_{i}^{\sharp\sharp}\) and \(\varphi^{\circ}\) outside \(U\).
On blunt components \(\varphi^{\mathrm{m}}=\varphi_{i}^{\sharp\sharp}\), while on sharp components \(\varphi^{\mathrm{m}}=\varphi^{\circ}\).
Therefore, the domain of \(\varphi^{\mathrm{m}}\) is \(M- (U\cap F_{i}(z))-P_{\tau}\).
We define the \emph{truncated regular guts} \(\mathrm{RG}_{i}^{\circ}\) to be \(\mathrm{RG}_{i}\) removing all sharp components in \(T_{\tau}\), which admits a compatible semiflow (not tight) by restricting \(\varphi^{\mathrm{m}}\) directly.
In other words, we truncate the tails of all sharp components (see Figure \ref{fig:cuttail}).
Figure \ref{fig:singulargraph}(f) gives a concrete local example.
Clearly, the truncated regular guts is orbit equivalent to the regular guts outside half annuli, blown-up annuli, and isolated product flows.
Hence, it suffices to show \(\mathrm{RG}_{1}^{\circ}\) and \(\mathrm{RG}_{2}^{\circ}\) are orbit equivalent outside half annuli and blown-up annuli.

Now we define the truncated regular guts \(\mathrm{RG}_{1,2}^{\circ}\) associated to \(F_{1}(z)\cup F_{2}(z)\).
Recall that \(F_{\varphi}^{1}(z)\) and \(F_{\varphi}^{2}(z)\) share the same \(P_{\tau}\) (see Section \ref{sec:singularguts}).
As a sutured submanifold, we define \(\mathrm{RG}_{1,2}^{\circ}\) to be the guts components of \(M-F_{\varphi}^{1}(z)\cup F_{\varphi}^{2}(z)\) that are not contained in \(P_{\tau}\).
\begin{lemma}
The truncated regular guts \(\mathrm{RG}_{1,2}^{\circ}\) admits a semiflow by restricting the mix \(\varphi^{\mathrm{m}}\) of \(\varphi^{\circ}\) and \(\varphi^{\sharp\sharp}_{i}\).
\end{lemma}
\begin{proof}
Since both \(F_{\varphi}^{1}\) and \(F_{\varphi}^{2}\) are (honestly) transverse to \(\varphi^{\circ}\), \(M-F_{\varphi}^{1}(z)\cup F_{\varphi}^{2}(z)\) inherits a ``semiflow'' from \(\varphi^{\circ}\).
The annulus isotopy theorem Proposition \ref{thm:annulusisotopy} still holds here.
A product annulus in \(M-F_{\varphi}^{1}(z)\cup F_{\varphi}^{2}(z)\) is isotopic to either a product \(\varphi^{\circ}\)(\(\varphi\))-saturated annulus in \(M-P_{\tau}\) or an annulus in \(P_{\tau}\).
Cut along a maximal collection of product saturated annuli and throw away all product flow components and all components in \(P_{\tau}\).
Finally, we replace the semiflow near blunt components by the restriction of \(\varphi_{i}^{\sharp\sharp}\), and remove the tails of sharp components.
Then we obtain a semiflow on \(\mathrm{RG}_{1,2}^{\circ}\), which is clearly the restriction of \(\varphi^{\mathrm{m}}\).
\end{proof}

Now we show that both of \(\mathrm{RG}_{1}^{\circ}\) and \(\mathrm{RG}_{2}^{\circ}\) are orbit equivalent to \(\mathrm{RG}_{1,2}^{\circ}\).
Write \(F_{i}=F_{\varphi}^{i}\). Endow the regular part \(X_{1}\) of \(M- F_{1}\)  with the restriction of \(\varphi^{\mathrm{m}}\).
Since \(F_{1}\) and \(F_{2}\) represent same homology class \(z\) (up to scalar), there exists \(k>0\) such that \(F_{2}\oplus kF_{1}\) is isotoped slightly to be disjoint from \(F_{1}\).
We throw away all homologically trivial components and denote the rest by \(F_{2,1}\).

Let \(C\) be a component of \(F_{2}\setminus F_{1}\).
If \(\chi(C)<0\), then \(C\) belongs to a component \(S\) of \(F_{2,1}\).
By the maximality, an isotopy \(\mu_{t},t\in I\) takes \(S\) to a component \(S_{1}\) of \(F_{1}\) along flowlines.
There are two types of \(\partial C\). 
Under \(\mu\), some components of \(\partial C\) pass from one side to the other side, called \emph{unfixed boundary circles}.
Let \(\partial _{u}C\) be the union of unfixed boundary circles. Then \(\mu_{I}(\partial _{u}C)\) is a collection of product saturated annuli and \(C\cup \mu_{I}(\partial _{u}C)\cup \mu_{1}(C)\) bounds a product flow in \(M\).
Consequently, \(C\cup \mu_{I}(\partial _{u}(C))\cup \mu_{1}(C)\)  bounds a product flow in \(X_{1}\).
Therefore, cutting along \(C\) is equivalent to cutting along \(\mu_{I}(\mathcal{C})\) and does not affect the semiflow on the truncated guts up to orbit equivalence.

Assume that \(C\) is a non-product annulus. Then the component of \(M-F_{1}\) containing \(C\) is a sutured torus \(U\) plusing windows \(W\). \(U\) is either a 4-ST or a non-longitudinal 2-ST.
Notice that \(C\) is \(\varphi^{\circ}\)(\(\varphi\))-transverse.
If \(C\subset W\), then cutting along \(C\) does not affect the semiflow on the guts since \(W\setminus C\) remains product.
Otherwise, \(C\setminus W\) is a transverse annulus in \(U\).
The conclusion follows from the fact that cutting along a transverse annulus in \(U\) won't change the semiflow on \(U\) up to orbit equivalence.

Assume that \(C\) is a product annulus. Then by Theorem \ref{thm:annulusisotopy} and Lemma \ref{lem:windowproduct}, there exists a further blowup \(\varphi^{\sharp\sharp}\) such that,
\begin{enumerate}
\item \(C\) is ambiently isotopic to an annulus \(C_{\varphi}\) in \(\varphi^{\sharp\sharp}\)-aligned position, and
\item \(C\cup C_{\varphi}\) and \(F_{1}\) cobound solid tori with product flow.
\end{enumerate}
Therefore, cutting along \(C\) is equivalent to cutting along \(C_{\varphi}\) and does not affect the semiflow on the truncated guts up to orbit equivalence.

In summary, both of \(\mathrm{RG}_{i}^{\circ}\) are orbit equivalent to \(\mathrm{RG}_{1,2}^{\circ}\).
Combining the discussion of singular guts components, it follows that the semiflows on \(\mathrm{Guts}(F_{1})\) and \(\mathrm{Guts}(F_{2})\) are orbit equivalent outside half annuli.
We complete the proof of Lemma \ref{lem:choice:facet}.

Theorem \ref{thm:facetinvariance} follows directly from combining Lemma \ref{lem:choice:blowup}, Lemma \ref{lem:choice:productannuli} and Lemma \ref{lem:choice:facet}.

\begin{remark}
The reason we split the proof into the singular guts and regular guts cases is that, 
\emph{bad regions} appear in singular guts \autocite[Definition 3.11]{AgolZhang2022_GutsInSuturedDecompositionsAndTheThurstonNorm_ARXv1}.
Precisely speaking, assume that \(F_{1}\) and \(F_{2}\) are two facet surface representing a primitive class \(z\in H_{2}(M,\partial M)\) and are in minimal position.
Let \(P\) be a component \(K\) of \(M- (F_{1}\cup F_{2})\). If the coorientation of \(P\cap F_{1}\) is the same as the coorientation of \(P\cap F_{2}\), we call \(P\) is a \emph{bad boundary} and \(J\) is a \emph{bad region}.
Bad regions encodes the difference between the guts of \(M-(F_{1}\cup F_{2})\) and \(\mathrm{Guts}(F_{i})\). In fact, for each bad region, there is a 4-ST guts component of \(M-(F_{1}\cup F_{2})\) in the same component of \(M-F_{i}\). If we remove the bad regions and their corresponding 4-STs, the remaining part is equivalent to \(\mathrm{Guts}(F_{i})\).

In the flow settings, there is no ``well-defined'' semiflow on bad regions. Fortunately, all bad regions are contained in the singular guts.
If \(K\) is a bad region, \(K\) contains a non-product annulus connecting two components of \(F_{i}\).  By Corollary \ref{cor:nonproductannulus}, one may assume that \(K\) contains a closed orbit \(\gamma\) and \(W^{s}(\gamma)\) or \(W^{u}(\gamma)\) (depending on the coorientation of \(K\cap F_{i}\)) intersects \(K\) along four simple closed curves.
This only happens when \(\gamma\) is a blowup of some singular closed orbit, and such component must be in the singular guts.
\label{rem:badregions}
\end{remark}
\section{Invariance on the Thurston open face}
\label{sec:invarianceopenface}
Let \(M\) be an oriented connected closed 3-manifold and let \(\varphi\) be a pseudo-Anosov flow without perfect fits in \(M\). Suppose that \(z\in \mathcal{C}_{2}(\varphi)\) is primitive.
In preceding section, we show that the semiflow on \(\mathrm{Guts}(M,z)\) is invariant with respect to choices in construction.
We denote by \(\mathrm{Guts}_{\varphi}(M,z)\) the guts with the semiflow induced by \(\varphi\).

In \autocite{AgolZhang2022_GutsInSuturedDecompositionsAndTheThurstonNorm_ARXv1}, Agol and Zhang show that the guts are an invariant associated to an open face (cone) of the Thurston norm unit ball.
\begin{theorem}[{\cite[Corollary 4.4]{AgolZhang2022_GutsInSuturedDecompositionsAndTheThurstonNorm_ARXv1}}]
Assume that \(M\) be an oriented connected closed 3-manifold with non-degenerate Thurston norm.
Let \(y,z\) be two elements in an open face of the Thurston norm unit ball.
Then the guts \(\mathrm{Guts}(M,y)\) is equivalent to \(\mathrm{Guts}(M,z)\).
\end{theorem}
An analogous result holds for the canonical flow on it.
\begin{theorem}
Assume that \(M\) be an oriented connected closed 3-manifold which admits a pseudo-Anosov flow \(\varphi\) without perfect fits.
Suppose that \(y,z\in \mathcal{C}_{2}(\varphi)\) are primitive and in the same open face.
Then, the guts \(\mathrm{Guts}_{\varphi}(M,y)\) is orbit equivalent to \(\mathrm{Guts}_{\varphi}(M,z)\) outside half annuli.
\label{thm:openfaceinvariance}
\end{theorem}
We prove this theorem below.

Recall that in Section \ref{sec:singularguts}, for a ladderpole tube \(T\), we define the singular graph \(G_{T}(F_{\tau}(z))\) and singular guts components associated to a homology class \(z\in \mathcal{C}_{2}(\varphi)\).
We show that the non-isolated vertices set of the singular graph is also an invariant over an open face.
\begin{lemma}
Let \(y,z\in \mathcal{C}_{2}(\varphi)\) are primitive and in the same open face and let \(T\) be a tube.
If \(T\) is a ladderpole tube of \(y\), then \(T\) is also a ladderpole tube of \(z\).
Further, if \(a\) is a non-isolated vertex of \(G_{T}(F_{\tau}(y))\), then \(a\) is also non-isolated in \(G_{T}(F_{\tau}(z))\).
\label{lem:singularinvariance3}
\end{lemma}
\begin{proof}
Let \(\gamma\) be the singular closed orbit contained in \(T\), and denote by \([\gamma]\) the homology class of \(\gamma\).
Since \(T\) is a ladderpole tube of \(y\), the (homological) intersection number \(\left< [\gamma],y \right>=0\). So
\[ \{x\in H_{2}(M): \left< [\gamma],x \right>=0\}\cap \mathcal{C}_{2}(\varphi) \]
is a closed face of \(\mathcal{C}_{2}(\varphi)\) containing \(y\).
Hence, \(\left< [\gamma],z \right>=0\) and \(T\) is also a ladderpole tube of \(z\).

For simplicity, we denote by \(V(x)\) the non-isolated vertices set of singular graph associated to \(x\) and \(T\).
We prove that \(V(y)\subset V(z)\) and hence \(V(y)=V(z)\) by a symmetric argument.
Choose \(u\in \mathcal{C}_{2}\) such that \(z\) is in the segment \((y,u)\).
Suppose that \(z=ay+bu\), \(a,b\in \Rational\).
Let \(v\in V(y)\) and let \(S_{y}\) be any taut surface relatively carried by \(\tau\) such that \([S_{y}]\in \Integer y\) and \(S_{y}\) meets \(v\).
Choose any relatively carried taut surface \(S_{u}\) with \([S_{u}]=u\).
Then \(S=kaS_{y}\oplus kbS_{u}\) is an embedded surface representing powers of \(z\), where \(ka,kb\in \Integer\).
Let \(S_{0}\) be the component of \(S\) that meets \(v\).
Then \(S_{0}\) is relatively carried by \(\tau\) and hence taut \autocite[Theorem B]{Landry2022_VeeringTriangulationsAndTheThurstonNorm_PUB}. (\(S\) may contain homologically trivial component.)
One extends \(S_{0}\) to be a facet \(F(z)\) and hence \(v\in V(z)\).
Therefore, \(V(y)\subset V(z)\).
\end{proof}
So by Proposition \ref{prop:singulargutsinvariance}, we have
\begin{corollary}
Let \(y,z\in \mathcal{C}_{2}(\varphi)\) are primitive and in the same open face.
Then the singular guts \(\mathrm{Guts}_{s}(y)\) associated to \(y\) is orbit equivalent to \(\mathrm{Guts}_{s}(z)\) outside half annuli.
\end{corollary}

To prove Theorem \ref{thm:openfaceinvariance}, it suffices to show that the regular guts \(\mathrm{Guts}_{r}(y)\) is orbit equivalent to \(\mathrm{Guts}_{r}(z)\) outside half annuli.
To this end, we employ branched surfaces to characterize facets associated to \(y,z\) in a unified frame.
We refer readers to Section \ref{sec:veering} for a short introduction to branched surfaces.

We quote a simple lemma from \autocite{AgolZhang2022_GutsInSuturedDecompositionsAndTheThurstonNorm_ARXv1}.
\begin{lemma}[{\cite[Proposition 4.2]{AgolZhang2022_GutsInSuturedDecompositionsAndTheThurstonNorm_ARXv1}}]
Let \(M\) be a closed 3-manifold with non-degenerate Thurston norm.
If \(u\) and \(v\) are in the same closed Thurston face, then there exist embedded norm-minimizing surfaces \(S_{u}\) and \(S_{v}\) with \([S_{u}]=u\) and \([S_{v}]=v\), such that, for any positive integers \(a\) and \(b\), \(aS_{u}\oplus bS_{v}\) is norm-minimizing.
\label{lem:AZsurfaces}
\end{lemma}

Let \(y,z\in \mathcal{C}_{2}(\varphi)\) be in the same open face. Then there exists \(u,v\in \mathcal{C}_{2}(\varphi)\)  such that both \(y\) and \(z\) are contained in the open segment \((u,v)\).
By Lemma \ref{lem:AZsurfaces}, choose embedded norm-minimizing surfaces \(S_{u}\) and \(S_{v}\) such that \(aS_{u}\oplus bS_{v}\) is norm-minimizing for any \(a,b\in \Integer_{+}\). In particular, \(S_{u}\setminus S_{v}\) and \(S_{v}\setminus S_{u}\) has no component of nonnegative Euler characteristic.

We construct a taut oriented branched surface \(B\) that carries \(S_{u}\) and \(S_{v}\) simultaneously as in \autocite[Proposition 8]{Oertel1985_HomologyBranchedSurfaces_PUB}.
Roughly speaking, replace each circle of \(S_{u}\cap S_{v}\) to an annulus of contact.
Then \(B\) carries \(aS_{u}\oplus bS_{v}\) for any positive integers \(a,b\).
Let \(N(B)\) denote the \(I\)-fibered neighborhood of \(B\) (see Figure \ref{fig:proof:branchedsurface}).
Interestingly, we can transform facets associated to different \(aS_{u}\oplus vS_{v}\) to each other via \(N(B)\).

\begin{figure}[htbp]
\centering
\includegraphics[scale=1]{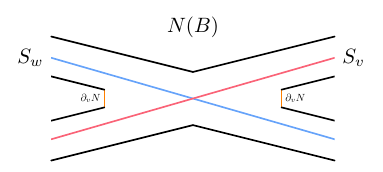}
\caption{\label{fig:proof:branchedsurface}}
\end{figure}

Assume that \(S=aS_{u}\oplus bS_{v}\) for some positive integers \(a\) and \(b\). Then as in \autocite[Section 6]{TollefsonWang1994_BranchedSurfacesAndThurstonsNormOnHomology_PUB}, isotope \(2S\) so that it contains the relative interior of \(\partial _{h}N(B)\).
We split \(N(B)\) along \(2S\) and obtain a honest \(I\)-bundle \(L_{B}(S)\), viewed as a product sutured manifold.
We claim that no component of \(L_{B}(S)\) is \(P\times I\) which satisfies: \(P\) is a punctured disk in the interior of \(2S\) and there is a boundary curve \(\alpha\) of \(P\) such that \(\alpha\times\{0\}\) and \(\alpha\times \{1\}\) bound disks in \(2S\) containing \(P\times \{0\}\) and \(P\times \{1\}\) respectively.
Otherwise, by the construction, a component of \(S_{u}\cap S_{v}\) bounds disks in \(S_{v}\) and \(S_{u}\). So there is a disk component in \(S_{v}\setminus S_{u}\), contradicting the construction of \(S_{v}\) and \(S_{u}\).
In particular, \(\partial _{v}(N)\) is a union of product annuli in \(M-2S\).
So if we decompose \(M\) along \(2S\) and along \(\partial _{v}N(B)\) further, we obtain \(L_{B}(S)\) again, together with \(M-N(B)\).

Let \(F_{S}\) be a facet that represents \(au+bv\) and contains \(S\). We isotope \(2F_{S}\) such that \(2F_{S}\) intersects \(\partial _{v}N(B)\) minimally.
For each component \(Z\) of \(F_{S}\), \(Z\cap L_{B}(S)\) is a horizontal taut surface parallel to \(R_{+}(L_{B}(S))\). In particular, \(Z\cap \partial _{v}N(B)\) is a collection of circles parallel to the sutures.
Hence, \(F_{S}\cap (M\setminus N(B))\) is a maximal collection of horizontal surfaces in \(M-N(B)\).
Otherwise, if there is a missing horizontal surface \(F_{0}\), we can add parallels of \(R_{+}(L_{B}(S))\) to \(F_{0}\) and get a surface in \(M\) that represents \(au+bv\).
This contradicts the maximality of \(F_{S}\).
We note that a surface \(Q\) in \(M-2S\subset M\) is homologous to \(S\) if and only if \(Q\cap (M-N(B))\) is a horizontal surface in \(M-N(B)\).

Suppose that \(S'=a'S_{u}\oplus b'S_{v}\) is an another surface. Then \(N(B)\) carries \(S'\). We decompose \(N(B)\) along \(2S'\) and \(\partial_{v}N(B)\) to obtain \(L_{B}(S')\) and \(M-N(B)\).
Recall that \(F_{S}\) is a facet containing \(S\), and \(F_{S}\cap (M-N(B))\) is a maximal collection \(H_{S}\) of horizontal surfaces in \(M-N(B)\), where \(2S\cap (M-N(B))\) is the \(R_{\pm}\)-part of \(M-N(B)\).
If we glue parallels of \(S'\cap N(B)\) to \(H_{S}\), the result \(F_{S'}\) is a facet that represents \(a'u+b'v\) and contains \(S'\).
In summary, by changing the portions in \(N(B)\), we modify a facet to another facet that represents a different homology class in the same open face.
Figure \ref{fig:proof:branchedsurface1} concretely illustrates how different \(aS_{u}\oplus bS_{v}\) differ inside \(N(B)\).

\begin{figure}[htbp]
\centering
\includegraphics[scale=0.8]{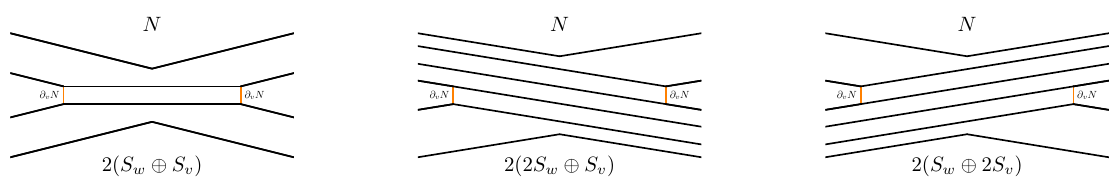}
\caption{\label{fig:proof:branchedsurface1}}
\end{figure}

Let \(T_{\tau}\) be a tube system and \(\tau\) is the associated veering triangulation.
Suppose that \(y,z\in \mathcal{C}_{2}(\varphi)\) are in the open segment \((u,v)\), and that \(S_{u},S_{v}\) are the surfaces constructed above.
As in Lemma \ref{lem:choice:blowup}, we assume that the blunt part of \(\mathrm{Guts}_{r}(u)\) and \(\mathrm{Guts}_{r}(v)\) are identical.
So there is a mixed flow \(\varphi^{\mathrm{m}}\) whose restriction yields the semiflow on the truncated regular guts \(\mathrm{Guts}^{\circ}_{r}(u)\) and \(\mathrm{Guts}_{r}^{\circ}(v)\) simultaneously.
We regard \(\varphi^{\mathrm{m}}\) as embedded in \(M\). Then \(S_{u}\) and \(S_{v}\) is genuinely transverse to \(\varphi^{\mathrm{m}}\) (though \(\varphi^{\mathrm{m}}\) is not defined on the whole \(M\)).
Choose \(N(B)\) small such that the verticle \(I\)-foliation on \(N(B)\) matches \(\varphi^{\mathrm{m}}\).

Suppose that \(F(y)\) is a facet that represents \(y=au+bv\) and contains \(S_{y}=aS_{u}\oplus bS_{v}\).
Set \(z=a'u+b'v\) and let \(S_{z}=a'S_{u}\oplus b'S_{v}\).
We transform the facet \(F(y)\) to be a facet \(F(z)\) via modifying portions in \(N(B)\).
Fix a collection of product annulus \(\mathcal{A}^{s}\) inside \(T_{\tau}\) to separate singular guts and regular guts of \(F(y)\) (see Section \ref{sec:singularguts}).
Due to Lemma \ref{lem:singularinvariance3}, \(\mathcal{A}^{s}\) also separate singular guts and regular guts of \(F(z)\).
Besides \(\mathcal{A}^{s}\), we fix a maximal collection \(\mathcal{A}^{p}\) of product \(\varphi^{\mathrm{m}}\)(\(\varphi\))-saturated annulus for \(F(y)\) outside \(N(B)\cup T_{\tau}\).
If we cut along \(N(B)\), \(\partial _{v}N(B)\), \(\mathcal{A}^{s}\), \(F(y)\), \(\mathcal{A}^{p}\) in sequence, and discard components in \(T_{\tau}\) and product flows, then we obtain the truncated regular guts \(\mathrm{Guts}_{r}^{\circ}(y)\).
Similarly, if we cut along \(N(B)\), \(\partial _{v}N(B)\), \(\mathcal{A}^{s}\), \(F(z)\), \(\mathcal{A}^{p}\) in sequence, and discard product flows and singular guts components, we obtain the truncated regular guts \(\mathrm{Guts}^{\circ}_{r}(z)\).
By construction, \(F(y)\) and \(F(z)\) are identical outside \(N(B)\), and the discarded components in \(N(B)\) are either as product flows or as parts of singular guts.
Recall that flows on both \(\mathrm{Guts}_{r}^{\circ}(y)\) and \(\mathrm{Guts}_{r}^{\circ}(z)\) are the restriction of \(\varphi^{\mathrm{m}}\) and we can reproduce the semiflow on regular guts \(\mathrm{Guts}_{r}\) from truncated version \(\mathrm{Guts}^{\circ}_{r}\) uniquely up to orbit equivalence outside half annuli.
Therefore, the semiflow on \(\mathrm{Guts}_{r}(y)\) and \(\mathrm{Guts}_{r}(z)\) are orbit equivalent outside half annuli.
This completes the proof of Theorem \ref{thm:openfaceinvariance}.
\section{Applications}
\label{sec:application}
\subsection{Cluster property}
\label{sec:org85181c5}
Let \(M\) be a closed 3-manifold which admits a pseudo-Anosov flow \(\varphi\) without perfect fits.
Recall that \(\varphi\) represents a (possibly empty) closed cone \(\mathcal{C}_{2}(\varphi)\subset H_{2}(M;\Real)\).
Denote by \(\mathcal{C}_{\mathrm{Fr}}(\varphi)\subset H_{1}(M;\Real)\) the closed dual cone of \(\mathcal{C}_{2}(\varphi)\), i.e., the minimal cone containing homology classes of closed orbits of \(\varphi\).
Each closed face \(F\) of \(\mathcal{C}_{\mathrm{Fr}}(\varphi)\) corresponds to a unique open face \(F^{\vee}\) of \(\mathcal{C}_{2}(\varphi)\) and vice verse.

We show that the guts associated to \(F^{\vee}\) encodes exactly closed orbits whose homology classes are in \(F\).
\begin{theorem}
Let \(F\) be a closed face of \(\mathcal{C}_{\mathrm{Fr}}(\varphi)\) and let \(F^{\vee}\) be its dual open face in \(\mathcal{C}_{2}(\varphi)\).
Assume that \(\gamma\) is a closed orbit of \(\varphi\).
Then \([\gamma]\in F\) if and only if \(\gamma\) can be isotopic into \(\mathrm{Guts}(F^{\vee})\).
Equivalently, \([\gamma]\in F\) if and only if \(\pi_{1}(\gamma)\) and \(\pi_{1}(\mathrm{Guts}(F^{\vee}))\), as conjugate classes, have nonempty intersection.
\label{thm:encodeclosedorbit}
\end{theorem}
\begin{proof}
Assume that \(\gamma\) is a closed orbit of \(\varphi\) and \([\gamma]\in F\). Suppose that \(\gamma\) is regular.
We choose any transverse facet \(F_{\varphi}(z)\) for \(z\in \iota(F)\). By our construction, there is a further blowup \(\varphi^{\sharp\sharp}\) and a maximal collection \(\mathcal{A}\) of product annuli in \(\varphi^{\sharp\sharp}\)-aligned position, so that, throwing away product flows yields the guts \(\mathrm{Guts}(\iota(F))\).
Since each component of \(F_{\varphi}(z)\) are positively transverse to \(\varphi\), it follows that \(\gamma\subset M-F_{\varphi}(z)\).
No closed orbit in \(M-F_{\varphi}(z)\) will be cut by a product \(\varphi^{\sharp\sharp}\)-saturated annulus or a \(\varphi^{\sharp\sharp}\)-transverse annulus.
Each product flow \(W\) contains no closed orbit so \(\gamma\) is not in \(W\).
Consequently, \(\gamma\) lies in \(\mathrm{Guts}(\iota(F))\).

Suppose that \(\gamma\) is singular.
We claim that \(\varphi^{\sharp\sharp}\) contains at least one closed orbit isotopic to \(\gamma\).
Let \(f:(M,\varphi^{\sharp\sharp})\to (M,\varphi)\) be the semi-conjugate map. Then \(f^{-1}(\gamma)\) is either a singular orbit or a blown-up complex.
Choose any transverse section of the blown-up complex \(T\) so that \(f(T)\) intersects \(\gamma\) once.
The return map of \(T\) is a rotation, so it have a fixed vertex in \(T\).
The blown up closed orbit \(\gamma'\) that corresponds to the fixed point is isotopic to \(\gamma\).
As in preceding paragraph, \(\gamma'\) will be contained in \(\mathrm{Guts}(\iota(F))\).

Let \(V\) be the linear subspace of \(H_{1}(M;\Real)\) spanned by \(F\) and let \(k= \dim V\).
By \autocite[Theorem 1.3]{AgolZhang2022_GutsInSuturedDecompositionsAndTheThurstonNorm_ARXv1}, the rank of \(i:H_{1}(\mathrm{Guts}(\iota(F)))\to H_{1}(M)\) is \(k\), and thus the image of \(i\) is \(V\).
Assume that \(\gamma\) is a closed orbit and \([\gamma]\notin F\). Then \([\gamma]\) is not in the image of \(i\) and hence \(\gamma\) cannot be isotoped into \(\mathrm{Guts}(\iota(F))\).
\end{proof}
In particular, we have
\begin{corollary}
Suppose that \(\varphi\) represents a top-dimensional closed cone \(\mathcal{C}_{2}(\varphi)\).
Then a closed orbit of \(\gamma\) is homologically trivial if and only if \(\gamma\) can be isotoped into \(\mathrm{Guts}(K)\) where \(K\) is the top open face of \(\mathcal{C}_{2}(\varphi)\).
\label{cor:homologicallytrivialorbits}
\end{corollary}
Moreover, if two pseudo-Anosov flows without perfect fits represent an open face simultaneously, then their closed orbits with certain homology classes will lie in same guts.
So classifying pseudo-Anosov flows without perfect fits may reduce to classifying semiflows on the guts and gluing flexibility.

Another corollary is about the cluster property proposed by Yi Liu.
Assume in addition that \(M\) is a hyperbolic 3-manifold.
We recall some hyperbolic geometric facts first.
The fundamental group of any hyperbolic 3-manifold is word-hyperbolic (or simply hyperbolic) in the sense of Gromov \autocite{Gromov1987_HyperbolicGroups_PUB}.
After choosing an orthonormal frame of reference at a basepoint, the fundamental group can be uniquely identified with a cocompact Kleinian group, namely, a cocompct discrete subgroup of \(\mathrm{PSL}(2,\Complex)\).
Any finitely generated subgroup of a cocompact Kleinian group is either \emph{geometrically finite} or \emph{geometrically infinite}.
Each geometrically finite subgroup is a quasiconvex subgroup of the nonelementary word-hyperbolic fundamental group, in the sense of geometric group theory \autocite[Proposition 4.4.2]{AschenbrennerFriedlWilton2015_3ManifoldGroups_PUB}.
Each geometrically infinite subgroup is not quasiconvex, but a virtual fiber, 
We refer the reader to \autocite[Section 4.1 and 4.4]{AschenbrennerFriedlWilton2015_3ManifoldGroups_PUB} for an informative survey on hyperbolic 3-manifolds and all direct references to the facts.

We prove a lemma which is well-known to experts but requires technically deep consequences of hyperbolization.
\begin{lemma}
Let \(P\) be an nonempty open face of the Thurston norm unit ball \(\mathcal{B}_{\mathrm{Th}}(M)\), and let \(Q\) be a component of \(\mathrm{Guts}(P)\).
Fix basepoints \(x_{0}\in M\), \(x_{Q}\in Q\) and a path from \(x_{0}\) to \(x_{Q}\).

Then \(\pi_{1}(Q,x_{0})\) is an infinite-index quasi-convex subgroup of \(\pi_{1}(M,x_{0})\).
\end{lemma}
\begin{proof}
If \(P\) is a fibered face, then \(\mathrm{Guts}(P)\) is emptyset and there are nothing to prove.
Let \(z\neq 0\in P\) and choose a facet \(F(z)\) associated to \(z\).

Each component of \(F(z)\) is quasi-Fuchsian (or equivalently geometrically finite \autocite[Lemma 4.66]{Ohshika2001_DiscreteGroups_PUB}). Otherwise, there is a component \(S\) which is embedded and a virtual fiber \autocite[Theorem 4.1.2]{AschenbrennerFriedlWilton2015_3ManifoldGroups_PUB}.
By \autocite[Theorem 10.5]{Hempel2004_3Manifolds_PUB}, \(S\) is either a (honest) fiber or a semifiber which is homologically trivial. Both cases contradict that \(S\) represents a non-fibered class \(z\neq 0\).
Let \(N\) be the component of \(M- F(z)\) that contains \(Q\). Then \(\pi_{1}(N,x_{0})\) is a geometrically finite subgroup of \(\pi_{1}(M,x_{0})\) of infinite index with non-empty domain of discontinuity \autocite{Thurston1982_GeometryAndTopologyOf3Manifolds}. (In fact, \(N\) admits a infinite-volume complete hyperbolic structure.)
Thurston shows that any finitely generated subgroup of \(\pi_{1}(N,x_{0})\) is still geometrically finite \autocite[Theorem 4.111]{Kapovich2010_HyperbolicManifoldsAndDiscreteGroups_PUB}.
\(\pi_{1}(Q)\) is finitely generated since \(Q\) is compact. So \(\pi_{1}(Q)\) is geometrically finite and hence quasi-convex.

We give an another way to prove the result, which is similar to \autocite[Lemma 9.3]{Liu2020_VirtualHomologicalSpectralRadiiForAutomorphismsOfSurfaces_PUB}.
Let \(H=\pi_{1}(Q)\).
Since \(H\) is finitely generated, it remains to prove that \(H\) is not virtual fiber.
Then \(H\) is contained in the kernel of the surjective composite homomorphism of groups:
\[\zeta: \pi_{1}(M,x_{0})\to H_{1}(M;\Integer)\xlongrightarrow{z}\Integer,  \]
(see the proof of Theorem \ref{thm:encodeclosedorbit}).
So \(H\) is of infinite index.
If \(H\) is not quasiconvex, then it is a virtual fiber, namely, there is a finite-index subgroup \(G'\) of \(G=\pi_{1}(M,x_{0})\) such that \(H\cap G'\) is normal in \(G'\) and \(G'/(H\cap G')\) is infinite cyclic. (\(H\cap G'\) corresponds to a fiber in the finite cover of \(M\) with respect to \(G'\))
\(\zeta\) factor through the infinite cyclic quotient \(G'\to G'/(H\cap G')\cong \Integer\) . However, \(G'\to G'/(H\cap G')\) corresponds to a fibration, so the kernel of \(\zeta|_{G'}\) is finitely generated.
This is due to a well-known fibering criterion in 3-manifold topology given by Stallings \autocite{Stallings1962_FiberingCertain3Manifolds_PUB}.
Therefore, \(\ker \zeta\) is finitely generated, but \(\zeta\) is not a fibered class, which is a contradiction.
\end{proof}

\begin{remark}
A related concept is the pared guts associated to a pared 3-manifold.
We refer the reader to \autocite{Agol2010_MinimalVolumeOrientableHyperbolic2Cusped3Manifolds_PUB} for an introduction to pared guts.
There is subtle difference between pared guts and sutured guts.
The pared guts is always an acylindrical pared manifold and hence admits a complete hyperbolic structure of finite volume with geodesic boundary \autocite{Morgan1984_ChapterVThurstonsUniformizationTheoremForThreeDimensionalManifolds_PUB}.
If a sutured manifold \(N\) is horizontally prime and atoroidal, then the sutured guts of \(N\) is either sutured torus or an acylindrical taut sutured manifold.
In the latter case, the sutured guts matches the pared guts with suitable pared assignment to \(N\) (and hence admits a complete hyperbolic structure of finite volume with geodesic boundary by same reason).
\end{remark}

\begin{definition}
Let \(F\) be a closed face of \(\mathcal{C}_{\mathrm{Fr}}(\varphi)\) of positive codimension.
A finite collection \(\mathcal{H}\) of subgroups of \(G=\pi_{1}(M)\) satisfies the \emph{cluster property} with respect to \(F\) if
\begin{enumerate}
\item each element of \(\mathcal{H}\) is an infinite-index quasi-convex subgroups of \(\pi_{1}(M)\),
\item if \(\gamma\) is a closed orbit of \(\varphi\) and \([\gamma]\in F\), then some power of \(\gamma\), as a free-homotopy loop, represents a conjugacy class in \(\pi_{1}(M)\) that has nonempty intersection with some element \(H\in \mathcal{H}\).
\end{enumerate}
\label{def:cluster}
\end{definition}
In \autocite{Liu2020_VirtualHomologicalSpectralRadiiForAutomorphismsOfSurfaces_PUB}, Liu gave a construction using Markov partition in the suspension flow case.
This plays a central role in many works of Liu \autocite{Liu2020_VirtualHomologicalSpectralRadiiForAutomorphismsOfSurfaces_PUB,Liu2022_FiniteVolumeHyperbolic3ManifoldsAreAlmostDeterminedByTheirFiniteQuotientGroups_PUB,Liu2025_TheEulerClassOneConjectureForFillableContactStructures_ARXv2}.
For example, based on cluster property and virtual specialness toolkits, Liu showed that any fibered cone can virtually become as complicated as you want.
Liu's construction depends on the choice of Markov partition and is merely a collection of subgroups.
Theorem \ref{thm:encodeclosedorbit} shows that the guts give a canonical submanifold realization of cluster property since the guts is a topological invariant of \(M\). We note that Theorem \ref{thm:encodeclosedorbit} is stronger than the cluster property.
\subsection{Distinguished orbits}
\label{sec:orgc6240c2}
Beyond the fundamental group, the sutured structure of the guts also encodes some typical closed orbits.
Assume that \(z\in \mathcal{C}_{2}(\varphi)\).
\begin{definition}[Distinguished orbit]
Let \(\gamma\) be a closed orbit of \(\varphi\) and let \(P\) be a prong of \(\gamma\).
Then \(\gamma\) is said to be \(z\)-\emph{distinguished} in \(P\)-direction if there exists a surface \(S\) with \([S]=z\) such that \(P\cap S\) contains a closed curve component.
Equivalently, according to \autocite[Lemma 8.3]{LandryMinskyTaylor2022_FlowsGrowthRatesAndTheVeeringPolynomial_PUB}, a power of \(\gamma\) is homotopic to a closed leaf of either \(W_{S}^{s}\) or \(W_{S}^{u}\), where \(W_{S}^{s/u}= W^{s/u}\cap S\) on \(S\).
\end{definition}
\begin{lemma}
Let \(\gamma\) be a closed orbit of \(\varphi\) and let \(P\) be a prong of \(\gamma\).
If \(\gamma\) is \(z\)-distinguished in \(P\)-direction, then there exists a facet \(F(z)\) such that:
\begin{enumerate}
\item there is an immersed annulus \(A\subset P\) whose interior is embedded and whose boundary consists of \(\gamma\) and a closed curve in a component \(S\) of \(F(z)\), and
\item \(A\) is disjoint from other components of \(F(z)\).
\end{enumerate}
\label{lem:distinguished}
\end{lemma}
\begin{proof}
We give two ways to prove this lemma.
By definition, there exists a surface \(S\) with \([S]\) and an annulus \(A\subset P\) connecting \(\gamma\) and a closed curve in \(S\).
We extend \(S\) to a transverse facet \(F(z)\).

Assume that a component \(S'\) intersects \(\mathrm{int}(P)\).
We claim that \(S'\cap P\) is a closed curve in \(S'\) and the conlusion follows by choosing a nearest \(S'\).
Since \(S'\) is homologous to \(S\), \(\gamma\) is also disjoint from \(S'\).
Consider the orbit space projection map \(\pi:\widetilde{M}\to P_{\varphi}\).
Choose a lift \(\widetilde{P}\) of \(P\) in \(\widetilde{M}\), whose projection image is a prong of \(\pi(\widetilde{\gamma})\).
There is a lift \(\widetilde{S}'\) of \(S'\) so that \(\pi(\widetilde{S}')\) contains \(\pi(\mathrm{int}(\widetilde{P}))\) and is disjoint from \(\pi(\widetilde{\gamma})\).
So \(\pi(\widetilde{\gamma})\in \partial \pi(\widetilde{S}')\), and hence \(\widetilde{S}'\cap \widetilde{P}\) is \(\pi_{1}(\gamma)\)-invariant and projects a closed curve in \(S'\) (see Lemma \ref{lem:elliptic-boundary} or \autocite[Lemma 8.3]{LandryMinskyTaylor2022_FlowsGrowthRatesAndTheVeeringPolynomial_PUB}).
Therefore, \(S'\cap P\) is a closed curve in \(S'\).

Another way relies on veering triangulation toolkits, which is a bit complicated but more figurative.
Let \(T_{\tau}\) be a tube system which contains all singular orbits and \(\gamma\) and let \(T_{\gamma}\) be the tube containing \(\gamma\). Assume that \(\tau\) be a veering triangulation on \(M\setminus T_{\tau}\).
We show that \(S\) can be isotoped to be carried by \(\tau\) and intersect \(T_{\gamma}\) in a ladderpole annulus.
We pull the closed curve of \(S\) inside \(T_{\tau}\) along \(P\) and puncture it as in Figure \ref{fig:pullbackpuncture}.
Denote the punctured surface by \(S^{\circ}\) and let \(M^{\circ}=M\setminus T_{\tau}\).
\(M^{\circ}\) admits a Fried blowup \(\varphi^{\circ}\) of \(\varphi\), and \(S^{\circ}\) is a properly embedded surface in \(M^{\circ}\) and transverse to \(\varphi^{\circ}\).
By \autocite[Theorem 5.1]{LandryMinskyTaylor2025_TransverseSurfacesAndPseudoAnosovFlows_PUB}, we can isotope \(S^{\circ}\) to be carried by \(\tau\). We cap off \(S^{\circ}\) by the removed (ladderpole) annulus in \(T_{\gamma}\), then the resulting \(S\) intersect \(T_{\gamma}\) in a ladderpole annulus.
Assume that \(F(z)\) is a maximal \(\tau\)-facet, then components of \(F(z)\) intersects \(T_{\gamma}\) at some ladderpole annuli. We just choose a \(S\) nearest to \(\gamma\).
\end{proof}

\begin{figure}[htbp]
\centering
\includegraphics[scale=0.9]{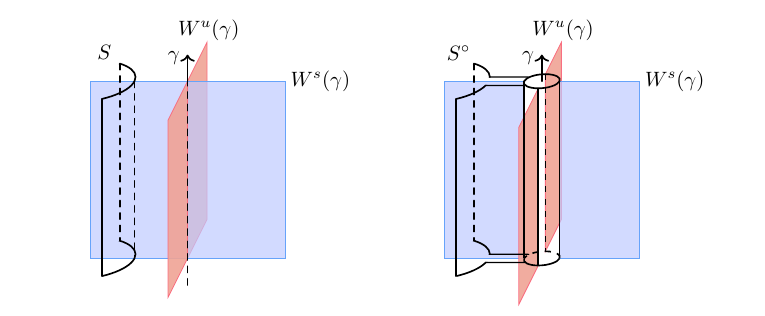}
\caption{\label{fig:pullbackpuncture}Pull a closed curve on \(S\) and puncture it.}
\end{figure}

\begin{corollary}
Let \(z\in \mathcal{C}_{2}(\varphi)\).
Let \(D_{n}(z)\) be the set of regular closed orbits which is \(z\)-distinguished in at least \(n\) directions. Then
\begin{enumerate}
\item each element in \(D_{1}(z)\) has a power which is isotoped into \(R_{\pm}\)-part of \(\mathrm{Guts}(z)\),
\item \(\#D_{2}(z)\leq\) the number of sutured annuli in \(\mathrm{Guts}(z)\),
\item \(\#D_{4}(z)\leq\) the number of 4-STs in \(\mathrm{Guts}(z)\).
\end{enumerate}
\label{cor:distinguished}
\end{corollary}
\begin{proof}
The first statement is a direct corollary of Lemma \ref{lem:distinguished}.
Assume that \(\gamma\) is a regular closed orbit and is \(z\)-distinguished in two directions. If \(\gamma\) has only two prongs, then \(\gamma\) is in a non-longitudinal 2-ST and corresponds to a sutured annulus.
If \(\gamma\) has four prongs and distinguished in two adjacent directions, then \(\gamma\) is isotoped into the a sutured annulus.
If \(\gamma\) is distinguished in two non-adjacent direction, \(\gamma\) is in a 4-ST.
The correspondence is injective since \(\varphi\) is without perfect fits and hence has no freely-homotopic closed orbits.
Therefore, the cardinality of \(D_{2}(z)\) is less or equal to the number of sutured annuli in \(\mathrm{Guts}(z)\).
The final statement is immediate.
\end{proof}
\subsection{Growth rate and entropy}
\label{sec:org4aa82c5}
In this section, we show that, for positive classes in \(H^{1}(M)\), the growth rate and entropy of the canonical semiflow exists.
In particular, on the closure of the fibered cone, there is a well-defined growth rates functions and entropy function on the guts.

Let \(F\) be an open face of \(\mathcal{C}_{2}(\varphi)\). Denote by \(\varphi_{F}\) the canonical semiflow on \(\mathrm{Guts}(M,F)\) and by \(\mathcal{O}(\varphi_{F})\) the set of closed orbits of \(\varphi_{F}\).
We define the positive class cone \(\mathcal{C}^{+}(F)\subset H^{1}(M)\) as the set of cohomology classes which is positive on \(\mathcal{O}(\varphi_{F})\).
We define the exponential growth rate for a positive class \(\eta\in \mathcal{C}^{+}(F)\) as
\[ \mathrm{gr}_{\varphi_{F}}(\eta)= \lim_{L\to \infty} \#\{\gamma\in \mathcal{O}(\varphi_{F}): \eta(\gamma)\leq L\}^{\frac{1}{L}}. \]
We note that \(\mathcal{C}^{+}(F)\) is an open cone in \(H^{1}(M)\), and \(\gamma\in \mathcal{O}(\varphi_{F})\) if and only if \(\zeta(\gamma)=0\) for some \(\zeta\in F\) according to Theorem \ref{thm:encodeclosedorbit}.
\begin{thm:growthrate}[Growth rate]
For any positive class \(\eta\in H^{1}(M)\), the growth rate \(\mathrm{gr}_{\varphi_{F}}(\eta)\) exists.
The growth rate \(\mathrm{gr}_{\varphi_{F}}(\eta)\)  is strictly greater than 1 for every positive \(\eta\in H^{1}(M)\) if and only if \(\varphi_{F}\) contains infinitely many primitive closed orbits.
\label{thm:growthrate}
\end{thm:growthrate}
\begin{proof}
The proof proceeds similarly to \autocite[Section 7, 8]{LandryMinskyTaylor2022_FlowsGrowthRatesAndTheVeeringPolynomial_PUB}

Let \(\tau\) be the veering triangulation associated to \(\varphi\). We choose any \(z\in F\) and any maximal \(\tau\)-transverse facet \(F(z)\).
Let \(\Phi\) be the flow graph of \(\tau\). We refer the reader to \autocite[Section 6]{LandryMinskyTaylor2022_FlowsGrowthRatesAndTheVeeringPolynomial_PUB} for a detailed introduction to the flow graph, and to \autocite{Mcmullen2015_EntropyAndTheCliquePolynomial_PUB} for the growth rate of a directed graph.
The flow graph can be embedded into \(M\) so that its edges are all positively transverse to \(F(z)\).
Denote by \(\Phi|F(z)\) the subgraph of \(\Phi\) removing edges which intersect \(F(z)\).
\(\eta\) pulls back to a strongly positive class in \(H^{1}(M-F(z))\).
\autocite[Theorem 8.1]{LandryMinskyTaylor2022_FlowsGrowthRatesAndTheVeeringPolynomial_PUB} applies here with some modifications.
Let \(\mathcal{O}^{r}(\varphi)\) denote the set of regular orbits.
Since singular orbits are only linear growth rate, \autocite[Theorem 8.1]{LandryMinskyTaylor2022_FlowsGrowthRatesAndTheVeeringPolynomial_PUB} implies that the regular exponential growth rate
\[ \mathrm{gr}_{\varphi_{F}}^{r}(\eta)= \lim_{L\to \infty} \#\{\gamma\in \mathcal{O}^{r}(\varphi_{F}): \eta(\gamma)\leq L\}^{\frac{1}{L}} \]
exists and equal to the growth rate of the directed graph \(\Phi|F(z)\) .
In \(\varphi_{F}\), a singular orbit may be blown up to finitely many closed orbit, which have linear growth rate.
So 
\[ \mathrm{gr}_{\varphi_{F}}(\eta)= \lim_{L\to \infty} \#\{\gamma\in \mathcal{O}(\varphi_{F}): \eta(\gamma)\leq L\}^{\frac{1}{L}} \]
exists and equals to the regular exponential growth rate.

The last statement follows directly from \autocite[Proposition 7.7]{LandryMinskyTaylor2022_FlowsGrowthRatesAndTheVeeringPolynomial_PUB}.
\end{proof}
Different choices of facet \(F(z)\) may give different graphs \(\Phi|F(z)\).
But a canonical minimal flow graph \(\Phi|F\) exists for an open face \(F\).
This is based on \autocite[Lemma 3.2]{Scharlemann1989_SuturedManifoldsAndGeneralizedThurstonNorms_PUB} and the proof of Theorem \ref{thm:openfaceinvariance}.

Let \(z\in F\). We form the minimal subgraph \(\Phi|z\) as follows.
An edge \(e\) of \(\Phi\) is in \(\Phi|z\) if a surface \(S\) with \([S]= z\) is relatively carried by \(\tau\) and intersects \(e\).
In fact, this subgraph can be achieved by a facet \(F(z)\).
First, we choose any transverse facet \(F(z)\) which is relatively carried by \(\tau\).
If an edge \(e\in \Phi|F(z)\)  is not in \(\Phi|z\), then a surface \(S\) is relatively carried by \(\tau\) and intersects \(e\). If \(S\) is disjoint from \(F(z)\), we add it to \(F(z)\) directly.
Otherwise, by \autocite[Lemma 3.2]{Scharlemann1989_SuturedManifoldsAndGeneralizedThurstonNorms_PUB}, a homologically-nontrivial component of \(S\oplus kF(z)\) is disjoint from \(kF(z)\), is relatively carried by \(\tau\) and intersects \(e\). We add this component to \(F(z)\).
After finitely many steps, we obtain a facet \(F(z)\) which achieve the minimal subgraph \(\Phi|z\).

For two different \(y,z\in F\), the minimal subgraphs \(\Phi|y\) and \(\Phi|z\) are the same.
As in the proof of Theorem \ref{thm:openfaceinvariance}, let \(y,z\) be contained in the open segment \((u,v)\) and \(S_{u}\) and \(S_{v}\) are constructed in Lemma \ref{lem:AZsurfaces}.
Let \(N(B)\) be the regular neighborhood of the branched surface formed by \(S_{u}\) and \(S_{v}\).
Via \(N(B)\), we can transform a facet associated to \(y\) to a facet associated to \(z\), and vice verse.
In particular, we can transform a facet \(F(y)\) which achieves \(\Phi|y\) to a facet which achieves \(\Phi|z\).
Therefore, \(\Phi|y\) and \(\Phi|z\) are identical.

Based on the argument above, a more canonical characterization of \(\Phi|F\) is as follows.
An edge \(e\) of \(\Phi\) is in \(\Phi|F\) if a surface \(S\) with \([S]\in F\) is relatively carried by \(\tau\) and intersects \(e\).

Due to \autocite[Thoerem 8.1]{LandryMinskyTaylor2022_FlowsGrowthRatesAndTheVeeringPolynomial_PUB}, the growth rate \(\mathrm{gr}_{\varphi_{F}}(\eta)\) is equal to the growth rate of the directed graph \(\Phi|F\).
Also, it equals the reciprocal of the smallest root of the specialization \(V^{\eta}_{\varphi_{F}}\) of the veering polynomial.

\begin{remark}
The positive cone \(\mathcal{C}^{+}(F)\subset H^{1}(M)\) is not intrinsic for the semiflow \(\varphi_{F}\) on the guts.
But the method here does not apply to proving the existence of the growth rate for the positive class in \(H^{1}(\mathrm{Guts}(F))\) (which is positive on \(\mathcal{O}(\varphi_{F})\)).
If we introduce the ``strongly positive'' definition (see \autocite[Section 8.2]{LandryMinskyTaylor2022_FlowsGrowthRatesAndTheVeeringPolynomial_PUB}), the growth rate for a strongly positive class in \(H^{1}(M-F(z))\) exists.
(Roughly speaking, a strong positive class is positive, in addition, on those curves on the boundary which are homotopic to a closed orbit.)
But it seems still hard for the strongly positive class in \(H^{1}(\mathrm{Guts}(F))\).
The difficulty is that we don't know how to place product annuli in a ``good'' position with the flow graph.
\end{remark}

As the logarithm of the growth rate, the entropy plays a more central role in many fields.
We summarize the basic properties of the entropy function.
The associated \emph{entropy function} is
\[ \begin{aligned}
\mathrm{ent}_{\varphi_{F}}: \mathcal{C}^{+}(F) &\to [0,+\infty)\\
\eta&\mapsto \log( \mathrm{gr}_{\varphi_{F}}(\eta) ).
\end{aligned} \]
Combining Theorem \ref{thm:growthrate} and \autocite[Theorem 9.1]{LandryMinskyTaylor2022_FlowsGrowthRatesAndTheVeeringPolynomial_PUB}, we have
\begin{theorem}[Entropy]
The entropy function \(\mathrm{ent}_{\varphi_{F}}\) is continuous, convex, and homogeneous of degree -1.
\label{thm:entropy}
\end{theorem}
\section{Further discussions and questions}
\label{sec:discussion}
\subsection{Boundary case}
\label{sec:org3f66450}
It is not hard to generalize results in this paper with some modifications, to a compact manifold with toral boundaries that admits a pseudo-Anosov flow.
We refer readers to \autocite[Section 3]{LandryMinskyTaylor2025_TransverseSurfacesAndPseudoAnosovFlows_PUB} for a detailed definition of almost pseudo-Anosov flow on compact manifolds.
In short, we say that a compact manifold with toral boundaries admits a pseudo-Anosov flow if it admits a Fried's blowup of a pseudo-Anosov flow on some closed orbits.

In the boundary case, the guts associated to a homology class is still well-defined.
But the well-definedness of the guts on open faces needs extra assumptions, that is,
\begin{theorem}[{\cite[Theorem 1.2]{AgolZhang2022_GutsInSuturedDecompositionsAndTheThurstonNorm_ARXv1}}]
Let \(M\) be an orientable, irreducible 3-manifold with toral boundaries and non-degenerate Thurston norm.
Let \(y,z\) be two elements in an open face of the Thurston norm unit sphere.
If there is an open segment \((v,w)\) containing \(y,z\) such that the restriction of \(v\) and \(w\) on each toral boundary are not in opposite orientations, then the guts \(\mathrm{Guts}(y)\) is equivalent to \(\mathrm{Guts}(z)\).
\end{theorem}
In fact, the branched surface trick in Section \ref{sec:invarianceopenface} requires this assumption.

In the aspect of flows, in Lemma \ref{lem:invariantline}, two product annulus (together with the facet) may cobound a \(T^{2}\times I\) where \(T^{2}\) is a torus.
In Lemma \ref{lem:gutsclassification}, there is an extra case \(T^{2}\times I\) with sutures on its boundary.
Proposition \ref{prop:singulargutsinvariance} can be generalized to the \(T^{2}\times I\) case verbatim since we have fixed the flow on the boundary.
\subsection{Characterization of flows on the guts}
\label{sec:orge8c5520}
A combinatorial characterization has many advantages.
Traditionally, one uses a Markov partition to characterize a flow.
The Markov partition is useful to characterize orbits, but it is neither canonical or well-behaved for transverse surfaces.
For pseudo-Anosov flows without perfect fits (equivalently, without freely-homotopic closed orbits), Agol and Guéritaud proposed a construction of veering triangulation.
Veering triangulations characterize transverse surfaces well, see \autocite{Landry2022_VeeringTriangulationsAndTheThurstonNorm_PUB,LandryMinskyTaylor2025_TransverseSurfacesAndPseudoAnosovFlows_PUB}.

This paper constructs a canonical semiflow on the guts.
One may construct a Markov partition for it, which will encode closed orbits combinatorially.
The semiflow on the guts still contains no freely-homotopic closed orbits.
Two natural questions arise:
\begin{enumerate}
\item is there a combinatorial model, like the veering triangulation, that encodes transverse surfaces or other objects in the guts elegantly (or more generally, any sutured manifold with a tight semiflow) and has a correspondence theory between semiflows and the combinatorial model?
In \autocite{Mosher1996_LaminationsAndFlowsTransverseToFiniteDepthFoliations}, Mosher proposed dynamic pairs to characterize the semiflows on a sutured manifold. But different dynamic pairs may induce same semiflows.
\item does Transverse Surface Theorem \autocite{LandryMinskyTaylor2025_TransverseSurfacesAndPseudoAnosovFlows_PUB} still hold, in some form, for a sutured manifold with a tight semiflow?
In \autocite{HuangTaylor2026_ConstructingDepthOneLaminationsTransverseToPseudoAnosovFlows_ARXv1}, Huang and Taylor prove a homologous version of the Transverse Surface Theorem in the first cut case.
Precisely speaking, for any nonnegative class \(\alpha\) in \(H^{1}(M-S)\), they construct an almost transverse properly embedded surface in \(M-S\) representing \(\alpha\).
\end{enumerate}
At least for veering triangulation, it seems difficult to generalize the combinatorial tricks in \autocite{Landry2022_VeeringTriangulationsAndTheThurstonNorm_PUB} to the guts case and to see how cutting product annuli affects the flow graph of veering triangulation; one might not be able to isotope a product annulus to a ``good'' position with veering triangulation.
\autocite[Section 11.4]{LandryTsang2025_EndperiodicMapsSplittingSequencesAndBranchedSurfaces_PUB} ask similarly, if there is a combinatorial way to recover the stable and unstable veering branched surfaces from some cellular decomposition in sutured case.

Broadly speaking, \autocite{LandryMinskyTaylor2025_TransverseSurfacesAndPseudoAnosovFlows_PUB} and this paper cut a flow into small pieces, while \autocite{Mosher1996_LaminationsAndFlowsTransverseToFiniteDepthFoliations} and \autocite{LandryTsang2025_EndperiodicMapsSplittingSequencesAndBranchedSurfaces_PUB} build a flow from pieces up to the whole.
To obtain a flow analogue of Gabai's sutured decomposition sequence, the sutured version of Transverse Surface Theorem is necessary.
The classification of tight semiflows on a fixed sutured manifold is also a natural question.
Tsang asks whether one can characterize all tight semiflows that only have finitely many primitive closed orbits, except the canonical ones on 4-STs and non-longitudinal 2-STs.
\subsection{Sutured version of cones}
\label{sec:org216be60}
Fried's work \autocite[Exposé 14]{FathiLaudenbachPoenaruKimMargalit2012_ThurstonsWorkOnSurfacesMn48_PUB} shows that a suspension flow represents a fibered cone of Thurston norm unit ball.
And the dual cone is spanned by the homology classes of all closed orbits.
Similar results hold for pseudo-Anosov without perfect fits (and veering triangulation) \autocite{LandryMinskyTaylor2022_FlowsGrowthRatesAndTheVeeringPolynomial_PUB}.
In general, the \emph{sutured Thurston norm} generalizes the Thuston norm to the sutured manifolds.
One may ask whether we can compute the sutured Thurston face by dynamical infomations.
Further, one may ask if a dual relation holds for the tight semiflows without freely-homotopic orbits: are the cone spanned by closed orbits and the cone spanned by (almost) transverse surfaces dual?
\subsection{Floer homology, volume, entropy and the guts}
\label{sec:org88e9d27}
Both hyperbolic geometry and Heegaard Floer homology have immense applications to the study of 3-manifolds.
However, until now, there are few known connections between them.
Lin and Lipnowski asked that is there any relationship between the topological invariants arising from the hyperbolic geometry and from Floer homology.

One promising path bridging them comes from the connection between Heegaard Floer invariants of a mapping torus and the fixed points of its monodromy, which originated in Cotton-Clay's work on symplectic Floer homology \autocite{Cotton-Clay2009_SymplecticFloerHomologyOfAreaPreservingSurfaceDiffeomorphisms_PUB}.
Recently, Liu \autocite{Liu2024_EntropyVersusVolumeViaHeegaardDiagrams_ARXv2} showed a way to relate entropy and volume, starting from volume, and then proceeding through Heegaard splitting representation length, the rank of Floer homology, fixed points of the monodromy, and finally arrive at the entropy.

The sutured Floer homology is invariant under horizontal and vertical sutured decomposition, and the volume of hyperbolic manifold is controlled by the volume of its guts \autocite[Theorem 9.1]{AgolDunfieldStormThurston2005_LowerBoundsOnVolumesOfHyperbolicHaken3Manifolds_ARXv2} \autocite[Theorem 5.5]{CalegariFreedmanWalker2009_PositivityOfTheUniversalPairingIn3Dimensions_PUB}.
So semiflows on the guts may give a new bridge to relate entropy and volume.
Can we derive a version of Floer homology directly from the semiflow on sutured manifold?
In \autocite{AlfieriTsang2025_HeegaardFloerTheoryAndPseudoAnosovFlowsI_ARXv1,AlfieriTsang2025_HeegaardFloerTheoryAndPseudoAnosovFlowsIi_ARXv1}, Alfieri and Tsang find a way to compute the Floer homology of a manifold by the dyanmics of a supported fully-punctured pseudo-Anosov flow with no perfect fits.
The chain complex is generated by certain closed orbits, like the periodic Floer homology.
But the general case seems very hard.
\appendix
\section{}
\label{app:A}
We give a concise proof of Proposition \ref{prop:shadow3}.
\begin{prop:shadow3}
Let \(l\) be a component of the frontier of \(\Theta(\widetilde{S})\) in \(P_{\varphi}\). Then,
\begin{enumerate}
\item \(l\) is a slice leaf of \(\mathcal{F}^{s/u}\). If it is a slice leaf \(l\) of \(\mathcal{F}^{s}\), then as a sequence of point \(\{p_{n}\}_{n\in \Number}\) tends to a point in \(l\), the preimages \((\Theta|_{\widetilde{S}})^{-1}(p_{n})\) in \(\widetilde{S}\) escape in the positive direction.
Similar conclusion holds for unstable boundary slice leaves.
\item \(l\) is periodic, namely, there is an element \(g\in \pi_{1}(M)\) such that \(g\) fixes \(l\) and hence a unique point in \(l\).
\end{enumerate}
\end{prop:shadow3}
\begin{proof}
According to \autocite[Proposition 4.1]{Fenley2009_GeometryOfFoliationsAndFlowsI_PUB} or \autocite[7.8]{LandryMinskyTaylor2025_TransverseSurfacesAndPseudoAnosovFlows_PUB}, any component \(l\) is a slice leaf of \(\mathcal{F}^{s/u}\) and if \(l\) is a stable slice leaf, then \(\widetilde{S}\) escape in the positive direction when approaching \(\Theta^{-1}(l)\).

We now focus on the second statement.
For compact \(M\) and transitive \(\varphi\), this was proven in \autocite[8.7]{LandryMinskyTaylor2025_TransverseSurfacesAndPseudoAnosovFlows_PUB} using veering triangulation.
Their method, however, does not generalize easily to the non-transitive case.

An earlier result for closed manifolds \(M\) and pure pseudo-Anosov flow \(\varphi\) appears in \autocite[Section 5]{Fenley1999_SurfacesTransverseToPseudoAnosovFlowsAndVirtualFibersIn3Manifolds_PUB}. Fenley's approach can be adapted to non-transitive setting with minor modifications.
We provide a concise adaptation below.

We recall the settings first.
Fix a universal lift \(\widetilde{S}\) of \(S\) and denote by \(\pi:\widetilde{S}\to S\) the covering map.
Let \(W_{\varphi}^{s}\) be the stable foliation in \(M\) associated to \(\varphi\) and \(W_{\varphi}^{u}\) be the unstable one. We note that all blown up annuli belong to \(W_{\varphi}^{s}\) and \(W_{\varphi}^{u}\) simultaneously in our notation.
The universal lifts of \(W^{s}_{\varphi}\) and \(W_{\varphi}^{u}\) project to two singular foliations on \(P_{\varphi}\), denoted by \(\mathcal{F}^{s}_{\varphi}\) and \(\mathcal{F}_{\varphi}^{u}\) respectively. 
Since \(S\) is transverse to \(\varphi\), there is a pair of induced singular foliations on \(S\), denoted by \(W_{S}^{s}\) and \(W_{S}^{u}\). If \(S\) intersects a blown-up annulus \(A\), their intersection is a closed curve on \(S\), called a \emph{blown-up closed leaf}.
Evidently, both \(W_{S}^{s}\) and \(W_{S}^{u}\) contain all blown-up closed leaves and are transverse elsewhere.
Moreover, the universal lift \(\widetilde{W}_{S}^{s/u}\) of these two foliation is identical with the restriction of \(\mathcal{F}_{\varphi}^{s/u}\) on \(\Theta(\widetilde{S})\).

\textbf{Claim:} Let \(\beta\) be a closed leaf of \(\mathcal{F}_{S}^{s}\). The projection of any lift \(\widetilde{\beta}\subset \widetilde{W}_{S}^{s}\) to the orbit space \(P_{\varphi}\) is either
\begin{enumerate}
\item an open blown-up segment, or
\item a stable prong of a periodic point lying on \(\partial \Theta(\widetilde{S})\).
\end{enumerate}

\textbf{Proof:} Such claim generalize \autocite[Proposition 5.7]{Fenley1999_SurfacesTransverseToPseudoAnosovFlowsAndVirtualFibersIn3Manifolds_PUB} to almost pseudo-Anosov case.
Let \(\mathcal{F}^{s}(\widetilde{\beta})\) be the stable leaf containing \(\widetilde{\beta}\) and there is a deck transformation \(g\in \pi_{1}(S)\) which preserves \(\mathcal{F}^{s}(\widetilde{\beta})\).
Then, \(g\) fixes a point \(\gamma\) in \(\mathcal{F}^{s}(\widetilde{\beta})\).
We observe that \(\gamma\) cannot lie in \(\Theta(\widetilde{S})\): otherwise we could choose a point \(p\) in \(\gamma\cap \widetilde{S}\), and by the \(g\)-invariance of \(\gamma\) and \(\widetilde{S}\), \(g^{k}(p)\) would also belong to their intersection, contradicting that \(\#(\gamma\cap \widetilde{S})\leq 1\).
By the structure of \(\partial \Theta(\widetilde{S})\), there is a unique boundary component \(l\in \mathcal{F}^{-}\) separating \(\gamma\) and \(\widetilde{\beta}\), hence fixed by \(g\).
Without loss of generosity, we assume that \(\gamma\in l\).
Clearly, \(\widetilde{\beta}\) has a ray asymptotic to \(\gamma\), and thus belongs to one prong of \(\mathcal{F}^{s}(\gamma)\).
If \(\widetilde{\beta}\) has non-empty intersection with blown-up trees, then we fall into case 1 directly by the structure of \(\Theta(\widetilde{S})\) near blown-up trees (Proposition \ref{prop:shadow2}).
Otherwise, \(\widetilde{\beta}\) blows down to an ordinary stable slice leaf in a pseudo-Anosov flow, and hence equal to one stable prong of \(\gamma\). \hfill\(\blacksquare\)

% TODO: add the definition of blown up tree on orbit space 

We remove all Reeb annuli in \(\mathcal{F}_{S}^{s}\) and collapse the two boundary closed leaves of each such annulus to a single leaf. This yields a Reebless foliation, which is straightened to be a geodesic lamination \(\mathcal{G}_{S}^{s}\).
Since the deep region between a perfect fit contains no singularities, \autocite[Lemma 5.13, Lemma 5.14]{Fenley1999_SurfacesTransverseToPseudoAnosovFlowsAndVirtualFibersIn3Manifolds_PUB} still hold in the almost pseudo-Anosov case.

Assume that \(\mathcal{G}\) is a minimal sublamination which is not a closed leaf.
We claim that \(\mathcal{G}\) is equal to \(\mathcal{G}_{S}^{s}\). In particular, \(\mathcal{G}_{S}^{s}\) is minimal.
Clearly, \(\mathcal{G}\) cannot intersect a small neighborhood of blown-up closed leaves: otherwise, \(\mathcal{G}\) would contain a leaf limit to the blown-up closed leaves.
The proof of \autocite[Proposition 5.10]{Fenley1999_SurfacesTransverseToPseudoAnosovFlowsAndVirtualFibersIn3Manifolds_PUB} carries over verbatim. Consequently, each leaf of \(\mathcal{G}\) is not thick, namely, has exactly one preimage in the Reebless part of \(\mathcal{F}_{S}^{s}\).
Proceeding along the proof of \autocite[Theorem 5.9]{Fenley1999_SurfacesTransverseToPseudoAnosovFlowsAndVirtualFibersIn3Manifolds_PUB}, we know that if \(\mathcal{G}\) contains no closed leaf, then the complement region of \(\mathcal{G}\) is the union of ideal polygons, and hence \(\mathcal{G}\) is equal to \(\mathcal{G}_{S}^{s}\).

Let \(l\) be an arbitrary stable component of \(\partial \Theta(\widetilde{S})\).
We choose a regular point \(p\in l\). Let \(\widetilde{L}=\mathcal{F}^{u}(p)\cap \Theta(\widetilde{S})\) be the bi-infinite leaf of \(\widetilde{\mathcal{F}}_{S}^{u}\).
If \(l\) intersects a blown-up segment or \(\widetilde{L}\) projects to a closed leaf, the conclusion follows directly.

Let \(\mathcal{G}\) be a minimal sublamination of the closure of \(\pi(\widetilde{L})\).
If \(\mathcal{G}\) is only a closed leaf, then the ray \(\widetilde{L}\) projects to a leaf ray \(L\) on \(S\), which spirals around a closed leaf \(\alpha\).
We choose a lift of \(\alpha\) in \(\Theta(\widetilde{S})\), denote by \(\widetilde{\alpha}\). So, as in lemma \ref{lem:elliptic-boundary}, \(\widetilde{\alpha}\) is either an open blown-up segment or a stable prong of a periodic point.
In both cases, lifts of \(L\) sufficiently close to \(\widetilde{\alpha}\) must limit to a periodic boundary component of \(\Theta(\widetilde{S})\) by the local stucture near \(\widetilde{\alpha}\). The existence of sufficiently close lift is guaranteed by the fact that \(L\) spirals around \(\alpha\).
Therefore, \(\widetilde{L}\) limits to a periodic boundary component, namely \(l\).

Otherwise, \(\mathcal{G}\) is equal to \(\mathcal{G}_{S}^{s}\). In particular, \(\widetilde{L}\) limits on itself.
If \(\widetilde{L}\) is not thick, then \(l\) is periodic by Claim 1 in the proof of \autocite[Proposition 5.10]{Fenley1999_SurfacesTransverseToPseudoAnosovFlowsAndVirtualFibersIn3Manifolds_PUB}.
Otherwise, \(\widetilde{L}\) is thick, which is impossible by the arguments in \autocite[Proposition 5.10]{Fenley1999_SurfacesTransverseToPseudoAnosovFlowsAndVirtualFibersIn3Manifolds_PUB}.

In summary, \(l\) is periodic. The conclusion for unstable boundary leaves holds similarly.
\end{proof}
\printbibliography
\end{document}